\documentclass[leqno, 11pt, a4paper]{amsart}

\usepackage{amssymb}
\usepackage{a4wide}
\usepackage{centernot} 
\usepackage{bbm}
\usepackage[cal=euler, scr=boondoxo]{mathalfa}
\usepackage[colorlinks=true, linkcolor=blue,citecolor=blue]{hyperref}
\usepackage{graphicx}
\newtheorem{theorem}{Theorem}[section]
\newtheorem{lemma}[theorem]{Lemma}
\newtheorem{prop}[theorem]{Proposition}

\newtheorem{corollary}[theorem]{Corollary}

\theoremstyle{definition}

\theoremstyle{remark}
\newtheorem{remark}[theorem]{Remark}

\numberwithin{equation}{section}

\newcommand{\lip}{\mathrm{Lip}}
\newcommand{\IND}{\mathbbm{1}}

\newcommand{\capa}{\mathrm{cap}}
\DeclareMathOperator{\supp}{\mathrm{supp}}
\newcommand{\var}{\mathrm{Var}}

\newcommand{\tr}{\mathrm{tr}}

\newcommand{\De}{\mathrm{d}}

\newcommand{\cA}{\ensuremath{\mathcal A}}
\newcommand{\cB}{\ensuremath{\mathcal B}}
\newcommand{\cC}{\ensuremath{\mathcal C}}
\newcommand{\cD}{\ensuremath{\mathcal D}}
\newcommand{\cE}{\ensuremath{\mathcal E}}
\newcommand{\cF}{\ensuremath{\mathcal F}}
\newcommand{\cG}{\ensuremath{\mathcal G}}
\newcommand{\cH}{\ensuremath{\mathcal H}}

\newcommand{\cK}{\ensuremath{\mathcal K}}

\newcommand{\cR}{\ensuremath{\mathcal R}}
\newcommand{\cS}{\ensuremath{\mathcal S}}

\newcommand{\cU}{\ensuremath{\mathcal U}}

\newcommand{\bbE}{\ensuremath{\mathbb E}}

\newcommand{\bbL}{\ensuremath{\mathbb L}}

\newcommand{\bbN}{\ensuremath{\mathbb N}}

\newcommand{\bbP}{\ensuremath{\mathbb P}}
\newcommand{\bbQ}{\ensuremath{\mathbb Q}}
\newcommand{\bbR}{\ensuremath{\mathbb R}}

\newcommand{\bbX}{\ensuremath{\mathbb X}}
\newcommand{\bbY}{\ensuremath{\mathbb Y}}
\newcommand{\bbZ}{\ensuremath{\mathbb Z}}
\begin{document}

\title[Entropic repulsion by disconnection]{Entropic repulsion for the Gaussian free field conditioned on disconnection by level-sets}


\author{Alberto Chiarini}
\address{Mathematics Department, UCLA}
\curraddr{520, Portola Plaza, 90095 Los Angeles, USA}
\email{chiarini@math.ucla.edu}
\thanks{}

\author{Maximilian Nitzschner}
\address{Departement Mathematik, ETH Z\"urich}
\curraddr{101, R\"amistrasse, CH-8092 Z\"urich, Switzerland}
\email{maximilian.nitzschner@math.ethz.ch}
\thanks{}

\begin{abstract} We investigate level-set percolation of the discrete Gaussian free field on $\bbZ^d$, $d\geq 3$, in the strongly percolative regime. 
We consider the event that the level-set of the Gaussian free field below a level $\alpha$ disconnects the discrete blow-up of a compact set $A\subseteq \bbR^d$ from the boundary of an enclosing box. We derive asymptotic large deviation upper bounds on the probability that the local averages of the Gaussian free field deviate from a specific multiple of the harmonic potential of $A$, when disconnection occurs. 
These bounds, combined with the findings of the recent article~\cite{duminil2019equality}, show that conditionally on disconnection, the Gaussian free field experiences an entropic push-down proportional to the harmonic potential of $A$. 
In particular, due to the slow decay of correlations, the disconnection event affects the field on the whole lattice. 
Furthermore, we provide a certain `profile' description for the field in the presence of disconnection. We show that while on a macroscopic scale the field is pinned around a level proportional to the harmonic potential of $A$, it locally retains the structure of a Gaussian free field shifted by a constant value. 
Our proofs rely crucially on the `solidification estimates' developed in~\cite{nitzschner2017solidification} by A.-S.\ Sznitman and the second author.
\end{abstract}

\subjclass[2010]{}
\date{November 25, 2019}
\dedicatory{}
\maketitle

\tableofcontents

\section{Introduction}
In this article we investigate level-set percolation for the discrete Gaussian free field on $\bbZ^d$, $d \geq 3$. The study of percolation phenomena for level-sets of the Gaussian free field was initiated in the eighties (see~\cite{bricmont1987percolation, lebowitz1986percolation, molchanov1983percolation}) and has attracted considerable interest in recent years (see~\cite{drewitz2017sign, nitzschner2018entropic, rodriguez2013phase,sznitman2015disconnection,sznitman2018macroscopic}), especially due to the presence of its slowly decaying spatial correlations. 

We aim at understanding the effects of conditioning the Gaussian free field to a certain \emph{rare event} arising from level-set percolation. More specifically, we consider the event that the level-set below $\alpha\in\bbR$ of the Gaussian free field \emph{disconnects} the discrete blow-up of a compact set $A\subseteq \bbR^d$ from the boundary of an enclosing box. The level $\alpha$ is chosen such that the level-set above $\alpha$ of the Gaussian free field is in a strongly percolative regime and thus the disconnection event becomes atypical.
The results we obtain share a similar spirit to (classical) entropic repulsion phenomena that  were for instance studied in~\cite{bolthausen2001entropic, bolthausen1995entropic, deuschel1999pathwise}. 
Roughly speaking, it is known that conditioning the Gaussian free field to be positive over the discrete blow-up of a compact set (with certain regularity assumptions) leads to an upward shift in its average. In our case, using a recently established equality of certain critical parameters (see~\cite{duminil2019equality}), it will turn out that conditioning on disconnection entails a pinning of the average of the Gaussian free field locally to $-(h_*-\alpha)h_A$, where $h_A$ is the harmonic potential of the set $A$ and $h_*$ is the threshold for level-set percolation of the Gaussian free field, if the set $A$ is sufficiently regular. This conclusion complements and refines the findings of~\cite{nitzschner2018entropic}.
The study of the Gaussian free field conditioned on certain events is in general a difficult problem, since the conditional measures are usually non-Gaussian. In the case of classical entropic repulsion, one can use Brascamp-Lieb inequalities to overcome this issue, see e.g.~\cite{deuschel1999pathwise}, but an extension of these methods to our context is not obvious. \vspace{\baselineskip}

We will now describe the model and our results in a more detailed way. Consider $\bbZ^d$, $d \geq 3$ and let $\bbP$ be the law on $\bbR^{\bbZ^d}$ so that
\begin{equation}
\begin{minipage}{0.8\linewidth}
\label{eq:GFFDef}
  under $\bbP$, the canonical field $(\varphi_x)_{x \in \bbZ^d}$ is a centered Gaussian field with covariance $\bbE[\varphi_x\varphi_y] = g(x,y)$ for all $x,y \in \bbZ^d$,
\end{minipage}
\end{equation}
where $g(\cdot,\cdot)$ denotes the Green function of the simple random walk on $\bbZ^d$, see~\eqref{eq:GreenFunction}. For $\alpha \in \bbR$, one defines the \textit{level-set} above $\alpha$ by
\begin{equation}
E^{\geq \alpha} = \lbrace x \in \bbZ^d; \varphi_x \geq \alpha \rbrace.
\end{equation}
There are three critical levels $0 < \overline{h} \leq h_\ast \leq h_{\ast\ast} < \infty$ relevant to the study of the percolation of $E^{\geq \alpha}$. The strongly non-percolative regime for $E^{\geq \alpha}$ corresponds to $\alpha > h_{\ast\ast}$, the strongly percolative regime corresponds to $\alpha < \overline{h}$, where $h_{\ast\ast}$ and $\overline{h}$ are defined in equation (0.6) of~\cite{rodriguez2013phase} and equation (5.3) of~\cite{sznitman2015disconnection} respectively, and $h_\ast$ denotes the threshold for percolation of $E^{\geq \alpha}$. It has been proven in~\cite{drewitz2018phase} that $\overline{h}>0$ for all dimensions $d \geq 3$, extending a previous result from~\cite{drewitz2017sign} which showed that $h_\ast > 0$ for all $d \geq 3$. More recently, it has been established in~\cite{duminil2019equality} that 
\begin{equation}
\label{eq:Equality_crit_par}
\overline{h} = h_\ast = h_{\ast\ast},
\end{equation}
using methods similar as in~\cite{duminil2017sharp}. In the remainder of this text, we will often deliberately formulate results in terms of $\overline{h}$ or $h_{\ast\ast}$ instead of $h_\ast$ to emphasize the effect of the respective strongly (non-)percolative nature of the regimes in question, and also to keep a consistency with earlier works.

Consider now a compact set $A \subseteq \bbR^d$ with non-empty interior contained in the interior of a box of side-length $2M$, $M > 0$, centered at the origin. The discrete blow-up of $A$ and the boundary of the discrete blow-up of its enclosing box are defined as
\begin{equation}
\label{eq:BlowUp_Boundary}
A_N = (NA) \cap \bbZ^d, \qquad S_N = \lbrace x \in \bbZ^d; |x|_\infty = \lfloor MN\rfloor \rbrace,
\end{equation}
respectively, with $|\cdot|_\infty$ denoting the sup-norm of a vector and $\lfloor\cdot \rfloor$ the integer part of a real number. One main object of interest will be the \textit{disconnection event}
\begin{equation}
 \label{eq:DisconnectionEvent}
 \cD^\alpha_N = \Big\lbrace A_N \stackrel{\geq \alpha}{\centernot\longleftrightarrow} S_N \Big\rbrace,
 \end{equation}
which corresponds to the absence of a nearest-neighbor path in $E^{\geq \alpha}$ connecting the sets $A_N$ and $S_N$.  Theorems 2.1 and 3.1 of~\cite{nitzschner2018entropic} provide large deviation lower and upper bounds for $\bbP[\mathcal{D}^\alpha_N]$ in terms of the Brownian capacity of $A$ ($\mathring{A}$) (see e.g.\ p.57-58 of~\cite{port2012brownian} for a definition), namely for $\alpha < h_{\ast\ast} ( = h_\ast)$ it holds that
\begin{equation}
\label{eq:DiscLowerBound}
\liminf_{N \rightarrow \infty} \frac{1}{N^{d-2}} \log \bbP[\mathcal{D}^\alpha_N] \geq - \frac{1}{2d}(h_{\ast\ast} - \alpha)^2 \capa(A),
\end{equation}
whereas for $\alpha < \overline{h} ( = h_\ast)$, it holds that
\begin{equation}
\limsup_{N \rightarrow \infty} \frac{1}{N^{d-2}} \log \bbP[\mathcal{D}^\alpha_N] \leq - \frac{1}{2d}(\overline{h} - \alpha)^2 \capa(\mathring{A}).
\end{equation}
Remarkably, by the equality of the critical parameters~\eqref{eq:Equality_crit_par} a combination of these results provides for any \emph{regular} compact set $A$ (in the sense that $\capa(A) = \capa(\mathring{A})$) the exact asymptotic behavior for $\bbP[\mathcal{D}^\alpha_N]$.
 What underlies the above results is the following effect:  In the percolative regime $\alpha<h_*$, the most `efficient' way in which the field can achieve a situation where its level-set below level $\alpha$ disconnects $A_N$ from $S_N$ is to have a down-shift of size $-(h_* - \alpha)h_A(x/N)$ at each site $x\in \bbZ^d$, where $h_A$ is the harmonic potential of $A$  (see~\eqref{eq:HarmonicPot}), illustrated in Figure~\ref{fig:shiftdown} below. The main results of this article, Theorems \ref{thm:MainUpperBound} and \ref{thm:ProfileDescription}, aim at quantifying this claim.
 \begin{figure}[h]\label{fig:shiftdown}
  \centering
  \includegraphics[width=0.8\textwidth]{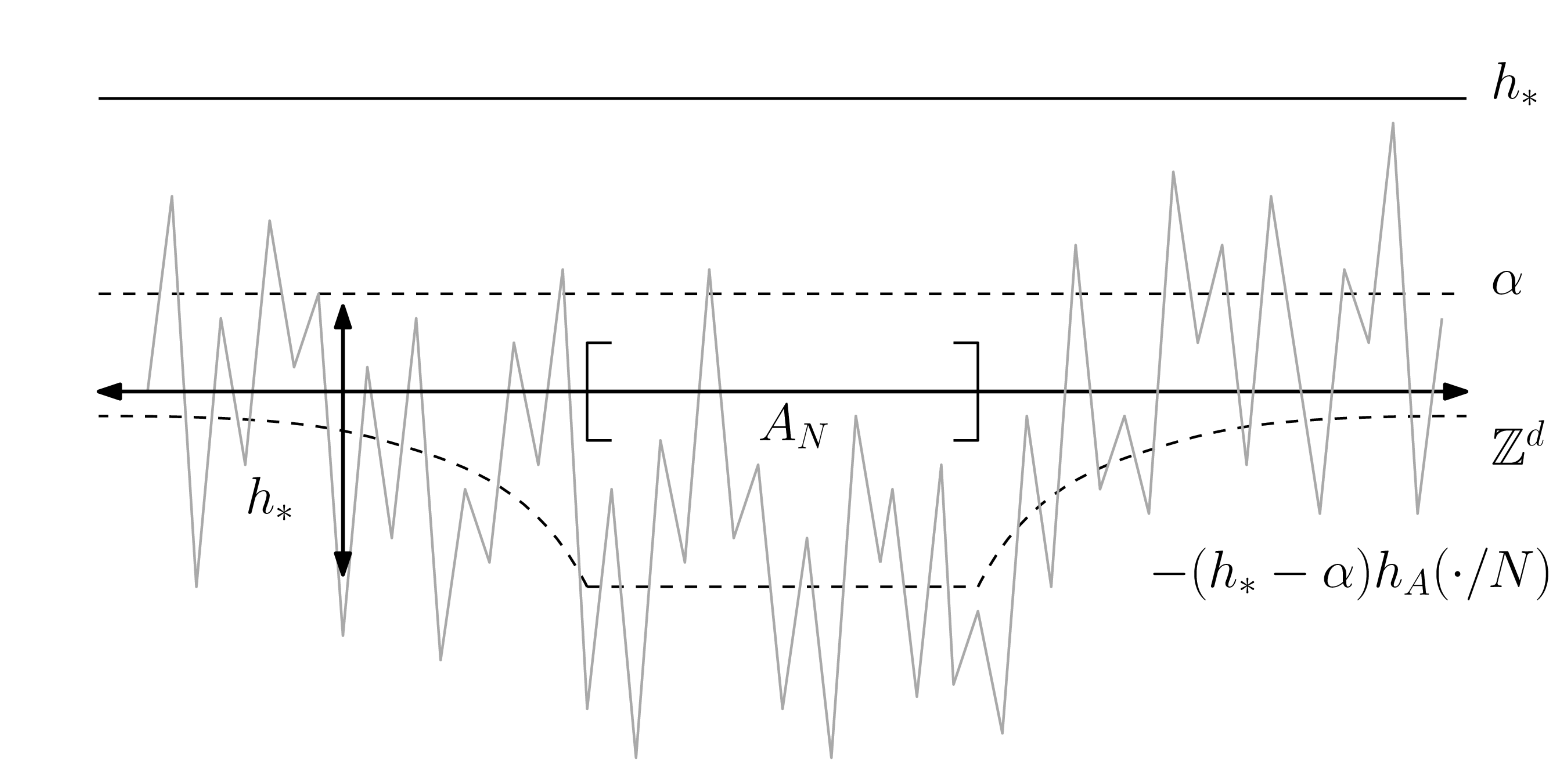} 
  \caption{Field conditioned on disconnection (informal picture). }
 \end{figure}

To investigate the influence of disconnection on the Gaussian free field, we introduce the random measure on $\bbR^d$, 
\begin{equation}\label{eq:measureGFF}
\bbX_N = \frac{1}{N^d} \sum_{x \in \bbZ^d} \varphi_x \delta_{\frac{x}{N}},
\end{equation}
and we define for any continuous, compactly supported function $\eta : \bbR^d \rightarrow \bbR$, and any signed Radon measure $\nu$ on $\mathbb{R}^d$ 
\begin{equation}
\label{eq:NotationIntegral}
\langle \nu,\eta \rangle = \int \eta(x) \nu(\De x).
\end{equation}

If $\nu(\De x) = f(x)\De x$, we write for simplicity $\langle f, \eta\rangle$ instead of $\langle \nu, \eta \rangle$. Our main result comes in Section 4 with Theorem \ref{thm:MainUpperBound}. We show that for $\alpha < \overline{h} ( = h_\ast)$, $\Delta > 0$ and $\lip_1(J)$, the class of Lipschitz functions supported in a compact subset $J$ of $\bbR^d$ with the sum of their sup-norm and Lipschitz constant bounded by one (see \eqref{eq:DefLip_1}), one has the asymptotic upper bound 
\begin{equation}
\begin{split}
\label{eq:Localization}
\limsup_{N \rightarrow \infty} \frac{1}{N^{d-2}} \log \, & \bbP\bigg[ \sup_{\eta \in \lip_1(J)} \left\vert \langle \bbX_N,\eta \rangle -  \langle \cH^\alpha_{\mathring{A}},\eta \rangle \right\vert \geq \Delta ; \cD^\alpha_N \bigg]  \\
& \leq -\frac{1}{2d}(\overline{h}-\alpha)^2 \capa(\mathring{A})
- c_1(\Delta,\alpha,J),
\end{split}
\end{equation}
where we defined $\cH^\alpha_{\mathring{A}} = -(\overline{h}-\alpha)h_{\mathring{A}}$, and $ c_1(\Delta,\alpha, J)$ is a positive constant depending on $\Delta$, $\alpha$ and $J$ as well as on $A$, $M$ and $d$. Let us point out that $\{ \sup_{\eta \in \lip_1(J)} |\langle \bbX_N, \eta \rangle - \langle \cH^\alpha_{\mathring{A}},\eta\rangle| \geq \Delta \}$ is measurable as $\lip_1(J)$ is separable with respect to $\| \cdot \|_\infty$, see below \eqref{eq:Theorem4_1_claim} for details. Since the critical parameters $\overline{h}, h_\ast$ and $h_{\ast\ast}$ coincide and if $A$ is regular in the sense that $\capa(A) = \capa(\mathring{A})$, one obtains by combining this bound with~\eqref{eq:DiscLowerBound} that for any $\alpha < h_\ast$, 
\begin{equation}
\label{eq:LocalizationCond}
\lim_{N \rightarrow \infty} \bbE\bigg[ \sup_{\eta \in \lip_1(J)}  \big| \langle \bbX_N,\eta \rangle - \langle \cH^\alpha_A,\eta \rangle \big|\wedge 1\Big\vert \cD^\alpha_N \bigg] = 0,
\end{equation}
see Corollary \ref{thm:Corollary42}. Thus, conditionally on the disconnection event $\cD^\alpha_N$, local macroscopic averages of the Gaussian free field are indeed pinned with high probability to $\cH^\alpha_A$.  
\vspace{\baselineskip} 

The exponential rate in~\eqref{eq:Localization} appears in a rather non-explicit way. It turns out that, if we are only interested in showing the entropic push-down of the field, we can obtain a more concrete (although not sharp) bound for the rate. 
In fact, we will show in Theorem \ref{thm:pushdown} of Section 3 that for $\alpha < \overline{h} ( = h_\ast)$, $\Delta > 0$ and a continuous, non-negative and compactly supported function $\eta : \bbR^d \rightarrow \bbR$, one has the asymptotic upper bound
\begin{equation}
\label{eq:PushDown}
\begin{aligned}
  \limsup_{N \rightarrow \infty} \frac{1}{N^{d-2}} \log\, & \bbP\Big[\langle\bbX_N,\eta \rangle \geq \langle \cH^\alpha_{\mathring{A}},\eta \rangle + \Delta ; \cD^\alpha_N \Big] \\ & \leq -\frac{1}{2d}(\overline{h}-\alpha)^2 \capa(\mathring{A})  - \frac{\Delta^2}{2d} \frac{1}{E(\eta)},
\end{aligned}
\end{equation}
where $E(\eta) = \int \eta(x) g_{BM} (x,y) \eta(y) \De x \De y$ is the energy associated to the function $\eta$, with $g_{BM} (x,y)$ being the Green function of the standard Brownian motion on $\bbR^d$. This result substantially strengthens Theorem 4.3 of~\cite{nitzschner2018entropic}, where $\eta$ was assumed to be the indicator function of a non-empty open subset with closure contained in $\mathring{A}$. 
In contrast to~\eqref{eq:Localization}, we get an explicit rate in~\eqref{eq:PushDown} because we can rely on a pointwise solidification upper bound (see Lemma~\ref{prop:harmonic_bound} below) which does not have a corresponding lower bound. Thus, in the derivation of~\eqref{eq:Localization}, which has the additional difficulty of being uniform in $\eta\in \lip_1(J)$, we need to replace these pointwise estimates with a weaker energy bound.

As a corollary to~\eqref{eq:PushDown}, if $\capa(A) = \capa(\mathring{A})$ and using~\eqref{eq:Equality_crit_par} one obtains that
\begin{equation}
    \limsup_{N \rightarrow \infty} \frac{1}{N^{d-2}} \log \bbP\Big[\langle\bbX_N,\eta \rangle \geq \langle \cH^\alpha_A,\eta \rangle + \Delta | \cD^\alpha_N \Big]  \leq - \frac{\Delta^2}{2d} \frac{1}{E(\eta)}.
  \end{equation}
This bound should be compared to the case without conditioning, where a direct computation gives
\begin{equation}
  \lim_{N \rightarrow \infty} \frac{1}{N^{d-2}} \log \bbP\Big[\langle \bbX_N, \eta\rangle \geq  \Delta  \Big]  = - \frac{\Delta^2}{2 d} \frac{1}{ E(\eta)}.
\end{equation}
In words: conditioned on disconnection, for the field to lie above its average is at least as costly as in the case without conditioning.
\vspace{\baselineskip}

To gain a deeper understanding of the local behavior of $\bbP[\ \cdot \ |\cD^\alpha_N]$ for large $N$, we introduce a certain `profile' description in the spirit of~\cite{bolthausen1993critical} for the Gaussian free field conditioned on disconnection. In essence, such a description enables us to track the behavior of the Gaussian free field simultaneously on a `global' scale as well as on a `local' scale. Roughly speaking, while on a global scale the local average of the Gaussian free field is pinned at $\cH^\alpha_A$ under $\bbP[\ \cdot \ |\cD^\alpha_N]$, it locally looks like a Gaussian free field shifted by a constant value. To rigorously capture this phenomenon, we define the random measure on $\bbR^d\times \bbR^{\bbZ^d}$
\begin{equation}
\bbY_N = \frac{1}{N^d} \sum_{x \in \bbZ^d} \delta_{\frac{x}{N}} \otimes \delta_{\tau_x \varphi},
\end{equation}
where $(\tau_x \varphi)_y = \varphi_{x+y}$ for all $x,y\in \bbZ^d$. We show in Theorem \ref{thm:ProfileDescription} of Section 5, that for any $\Delta > 0$, $\alpha < \overline{h} ( = h_\ast)$, and functions $\eta : \bbR^d \to \bbR$ and $F : \bbR^{\bbZ^d} \to \bbR$  with certain regularity properties (see above \eqref{eq:StatementProfileThm51}), there exists a positive constant $c_2(\Delta, \alpha, \eta, F)$ (depending besides $\Delta, \alpha, \eta, F$ also on $A$, $M$ and $d$) such that
\begin{equation}
\label{eq:ProfileDescription}
\begin{split}
\limsup_{N \rightarrow \infty} \frac{1}{N^{d-2}} \log \bbP &\Big[\big|\langle \bbY_N -  \De x \otimes \bbP^{\cH^\alpha_{\mathring{A}}(x) \mathbbm{1} } , \eta \otimes F\rangle\big| \geq \Delta; \cD^\alpha_N \Big] \\
& \leq - \frac{1}{2d}(\overline{h}-\alpha)^2\capa(\mathring{A}) - c_2(\Delta,\alpha,\eta,F),
\end{split}
\end{equation}
where $\bbP^{\cH^\alpha_{\mathring{A}}(x) \mathbbm{1} }$ is the probability measure on $\bbR^{\bbZ^d}$ such that 
\begin{equation}
\text{$\varphi$ under $\bbP^{\cH^\alpha_{\mathring{A}}(x) \mathbbm{1}}$ has the same
law as $\varphi + \cH^\alpha_{\mathring{A}}(x) \mathbbm{1}$ under $\bbP$,
} 
\end{equation}
and  $\mathbbm{1}(y) = 1$ for all $y\in \bbZ^d$.
Again, by making use of~\eqref{eq:Equality_crit_par}, if we assume $A$ to be regular, for any $\alpha<h_*$ it holds that
\begin{equation}
\lim_{N \rightarrow \infty} \bbE\Big[ \big|\langle \bbY_N -  \De x \otimes \bbP^{\cH^\alpha_{\mathring{A}}(x) \mathbbm{1} },  \eta \otimes F \rangle\big|\wedge 1 \big| \cD^\alpha_N \Big] = 0,
\end{equation}
see Corollary \ref{thm:Corollary52}.
\vspace{\baselineskip}

The organization of this article is as follows: In Section 2 we introduce further notation and recall some useful results concerning random walks, the Gaussian free field and the solidification estimates from~\cite{nitzschner2017solidification}. In Section 3, we state and prove Theorem \ref{thm:pushdown}, which corresponds to the entropic push-down result~\eqref{eq:PushDown}. In Section 4, we proceed to our main result, Theorem \ref{thm:MainUpperBound}, corresponding to the pinning result~\eqref{eq:Localization}. Finally, in Section 5, we investigate the profile of the field conditioned on disconnection and prove in Theorem \ref{thm:ProfileDescription} the claim \eqref{eq:ProfileDescription}. In the Appendix, we provide asymptotic comparisons between the hitting probability of arbitrary finite unions of large discrete boxes by the simple random walk and the hitting probability of the solid filling of these boxes by Brownian motion.
\vspace{\baselineskip}

Finally, we give the convention we use concerning constants. By $c,c',\ldots$ we denote generic positive constants changing from place to place, that depend only on the dimension $d$ as well as on the compact set $A$ and the parameter $M > 0$ (see above \eqref{eq:BlowUp_Boundary}), which will be fixed quantities in Sections 3, 4 and 5. Numbered constants $c_1,c_2,\ldots$ will refer to the value assigned to them when they first appear in the text and dependence on additional parameters is indicated in the notation.
\section{Notation and useful results}
In this section we introduce some notation and collect useful results concerning random walks, potential theory, the discrete Gaussian free field and the solidification estimates for porous interfaces from~\cite{nitzschner2017solidification}. These solidification estimates, together with a related capacity lower bound will be instrumental in the following sections to derive the large deviation upper bounds~\eqref{eq:Localization},~\eqref{eq:PushDown} and~\eqref{eq:ProfileDescription}. We will assume that $d \geq 3$ throughout the article.
\vspace{\baselineskip}

We start by introducing some notation. For real numbers $s,t$, we denote by $s \vee t$ and $s \wedge t$ the maximum and minimum of $s$ and $t$, respectively, and we denote the integer part of $s$ by $\lfloor s \rfloor$. We consider on $\mathbb{R}^d$ the Euclidean and $\ell^\infty$-norms $|\cdot|$ and $|\cdot|_\infty$ and  the corresponding closed balls $B_2(x,r)$ and $B_\infty(x,r)$ of radius $r\geq 0$ and center $x \in \bbR^d$. Also,  we denote by $B(x,r) = \lbrace y \in \mathbb{Z}^d; |x-y|_\infty \leq r \rbrace \subseteq \bbZ^d$ the closed $\ell^\infty$-ball of radius $r\geq 0$ and center $x\in \bbZ^d$.  
For subsets $G,H \subseteq \mathbb{R}^d$, we denote by $d(G,H)$ their mutual $\ell^\infty$-distance, i.e.\ $d(G,H) = \inf \lbrace |x-y|_\infty ; x \in G , y \in H \rbrace$ and write for simplicity $d(x,G)$ instead of $d(\lbrace x \rbrace, G)$ for $x \in \mathbb{R}^d$. We also define the diameter of $G$ as $\mathrm{diam}(G) = \sup_{x,y \in G}|x-y|$. For $\delta > 0$, we denote by $G^\delta = \{ x \in \bbR^d; \inf_{y \in G}|x-y| \leq \delta\}$ the closed $\delta$-neighborhood of $G$. For $K \subseteq \mathbb{Z}^d$, we let $|K|$ denote the cardinality of $K$.
If $x,y \in \mathbb{Z}^d$ fulfill $|x - y| = 1$, we call them neighbors and write $x \sim y$. We call $\pi : \lbrace 0,\ldots, N \rbrace \rightarrow \mathbb{Z}^d$ a nearest neighbor path (of length $N \geq 1$) if $\pi(i) \sim \pi(i+1)$ for all $0 \leq i \leq N-1$. Given two measurable functions $f,g$ on $\bbR^d$ such that $|fg|$ is Lebesgue-integrable we define $\langle f,g\rangle = \int f(y)g(y)dy$. For functions $f:\bbR^d \rightarrow \bbR$ and $h: \bbZ^d \rightarrow \bbR$, we denote by $\|f\|_\infty$ and $\|h \|_\infty$ the respective supremum norms over $\bbR^d$ and $\bbZ^d$, and we denote by $f^+ = f \vee 0$ and $f^- = (-f) \vee 0$ the positive and negative part of $f$, respectively. If $f:\bbR^d\to \bbR$ is continuous and compactly supported and $\nu$ is a Radon measure on $\bbR^d$ we write $\langle \nu, f \rangle = \int f\De \nu$.
\vspace{\baselineskip}

Let us now introduce the discrete time simple random walk on $\mathbb{Z}^d$. We denote by $(X_n)_{n \geq 0}$ the canonical process on $(\mathbb{Z}^d)^{\mathbb{N}}$ and by $P_x$ the canonical law of a simple random walk on $\mathbb{Z}^d$ started at $x \in \mathbb{Z}^d$. For a subset $K \subseteq \mathbb{Z}^d$, we introduce stopping times (with respect to the canonical filtration generated by $(X_n)_{n \geq 0})$ $H_K = \inf \lbrace n \geq 0; X_n \in K \rbrace$, $\widetilde{H}_K = \inf \lbrace n \geq 1; X_n \in K \rbrace$, and $T_K = \inf \lbrace n \geq 0; X_n \notin K \rbrace$, the entrance, hitting and exit times of $K$. The Green function of the random walk $g(\cdot,\cdot)$ is then defined by
\begin{equation}
\label{eq:GreenFunction}
g(x,y) = \sum_{n \geq 0} P_x[X_n = y], \text{ for } x,y \in \mathbb{Z}^d,
\end{equation}
and since $d \geq 3$, it is finite. Moreover, one has $g(x,y) = g(x-y,0) \stackrel{\text{def}}{=} g(x-y)$ and the following asymptotic behavior (see e.g. Theorem 5.4, p.31 of~\cite{lawler2013intersections}):
\begin{equation}
\label{eq:AsymptoticBehaviourGreen}
g(x) \sim \frac{C_d}{|x|^{d-2}}, \qquad \text{as } |x| \rightarrow \infty, \text{ with }C_d = \frac{d}{2\pi^{\frac{d}{2}}}\Gamma\left(\frac{d}{2}-1 \right).
\end{equation}
 The equilibrium measure of a finite subset $K \subseteq \mathbb{Z}^d$ is defined by
\begin{equation}
\label{eq:EqMeasure}
e_K(x) = P_x[\widetilde{H}_K = \infty] \mathbbm{1}_K(x), \text{ for } x \in \mathbb{Z}^d,
\end{equation}
and its total mass
\begin{equation}
\text{cap}_{\mathbb{Z}^d}(K) = \sum_{x \in K}e_K(x)
\end{equation}
is called the (discrete) capacity of $K$. Recall that for finite $K \subseteq \mathbb{Z}^d$, one has
\begin{equation}
\label{eq:EquilibriumPotential}
P_x[H_K < \infty] = \sum_{x' \in K} g(x,x') e_K(x'), \text{ for }x \in \mathbb{Z}^d,
\end{equation} 
see e.g.\ Theorem 25.1, p.300 of \cite{spitzer2013principles}. For a box of size $L$, which we denote by $B_L = B(0,L)$, it can be shown that (see e.g.\ (2.16), p.53 of~\cite{lawler2013intersections})
\begin{equation}
\label{eq:BoxCapBound}
cL^{d-2} \leq \text{cap}_{\mathbb{Z}^d}(B_L) \leq c' L^{d-2}, \text{ for } L \geq 1.
\end{equation}
We will now discuss the Gaussian free field on $\mathbb{Z}^d$, $d \geq 3$. We recall the definitions of $(\varphi_x)_{x \in \mathbb{Z}^d}$ and $\bbP$ from~\eqref{eq:GFFDef}. For $U \subseteq \mathbb{Z}^d$, one can define the harmonic average $h^U$ of $\varphi$ in $U$ and the local field $\psi^U$, via 
\begin{align}
\label{eq:HarmonicAverage}
h^U_x = E_x[\varphi_{X_{T_U}}, T_U < \infty] &= \sum_{y \in \mathbb{Z}^d} P_x[X_{T_U} = y, T_U < \infty]\varphi_y, \text{ for } x \in \mathbb{Z}^d,
\\
\label{eq:LocalField}
\psi^U_x &= \varphi_x - h^U_x, \text{ for } x\in \mathbb{Z}^d.
\end{align}
Obviously, $\varphi_x = h^U_x + \psi^U_x$ and $\psi_x^U = 0$ if $x \in \mathbb{Z}^d \setminus U$, whereas $h^U_x = \varphi_x$ in that case. Furthermore, one has the `domain Markov property' of the Gaussian free field, which asserts that
\begin{equation}
\begin{minipage}{0.9\linewidth}
  $(\psi^U_x)_{x \in \mathbb{Z}^d}$ is independent of $\sigma(\varphi_y; y \in U^c)$ (in particular of  $(h^U_x)_{x \in \mathbb{Z}^d}$),
  and is distributed as a centered Gaussian field with covariance $g_U(\cdot,\cdot)$,
\end{minipage}
\end{equation}
where, $g_U(\cdot,\cdot)$ is the Green function of the random walk killed upon exiting $U$ (see (1.3) of~\cite{sznitman2015disconnection}). 

We will also need in Sections 4 and 5 a general second moment estimate which states that for any centered Gaussian vector $(Y_1,\ldots,Y_m)$ with values in $\bbR^m$ (governed by some probability $\bbQ$) and covariance matrix $G \in \bbR^{m \times m}$, one has for $t \geq 0$ (see Lemma A.1, p.1913 of \cite{bolthausen1993critical}), 
\begin{equation}
\label{eq:GaussianEstimate}
\bbQ\bigg[\sum_{i = 1}^m Y_i^2 \geq \tr(G) + t \bigg] \leq \exp \bigg\{-\frac{1}{8} \min\bigg( \frac{t}{\overline{G}}, \frac{t^2}{\tr(G^2)} \bigg) \bigg\},
\end{equation}
where $\overline{G} = \max_{1 \leq i \leq m} \sum_{j = 1}^m |G_{ij}|$ and $\tr(\cdot)$ denotes the trace of a matrix. This inequality will be applied for the Gaussian free field in the proof of Theorem~\ref{thm:MainUpperBound}. One can see that for $G(x,y) = \bbE[\varphi_x\varphi_y]$, $x,y \in J_N = (NJ) \cap \bbZ^d$ (with $J$ a non-empty, compact set), one has the asymptotics 
\begin{equation}
\label{eq:BoundsCovMatrixGFF}
\begin{split}
\tr(G) &\sim c(J) N^d, \\
\tr(G^2) &= O(N^{d+1}) \text{ as }N \rightarrow \infty, \\ 
\overline{G} & = O(N^2) \text{ as }N \rightarrow \infty,
\end{split}
\end{equation}
see e.g.~\cite{bolthausen1993critical}, p. 1899.
\vspace{\baselineskip}

We now introduce Brownian motion on $\bbR^d$ and present some aspects of its potential theory, in a similar fashion as it was done for the simple random walk above. Let $(Z_t)_{t \geq 0}$ be the canonical process on $C(\bbR_+, \bbR^d)$ and denote by $W_z$ the Wiener measure starting from $z \in \bbR^d$ such that under $W_z$, $(Z_t)_{t \geq 0}$ is a Brownian motion starting from $z \in \bbR^d$. For any open or closed set $B\subseteq \bbR^d$, we introduce stopping times (with respect to the canonical filtration generated by $(Z_t)_{t \geq 0}$) $H_B = \inf \lbrace s \geq 0; Z_s \in B \rbrace$ and $\widetilde{H}_B = \inf \lbrace s > 0; Z_s \in B \rbrace$, the entrance and hitting times of $B$ for Brownian motion, and $T_B = \inf \lbrace s \geq 0; Z_s \notin B \rbrace ( = H_{B^c})$, the exit time of Brownian motion from $B$. For later use we also define the first time when $Z$ moves at $|\cdot|_\infty$-distance $ r\geq 0$ from its starting point, 
\begin{equation}
\tau_r = \inf \lbrace s \geq 0; |Z_s - Z_0|_\infty \geq r \rbrace.
\end{equation}
For an open or closed set $B \subseteq \bbR^d$, one introduces the harmonic potential of $B$,
\begin{equation}\label{eq:HarmonicPot}
h_B(z) = W_z[\widetilde{H}_B < \infty], \qquad z \in \bbR^d.
\end{equation}
If $B$ is open and bounded and $(B_n)_{n \geq 1}$ is a sequence of compact sets such that $B_n \uparrow B$, then (see Proposition 1.13, p.60 of~\cite{port2012brownian})
\begin{equation}
\label{eq:hAconvergence}
h_{B_n} \uparrow h_B \qquad \text{on } \bbR^d.
\end{equation}

For $f,g  \in H^1(\bbR^d)$, the usual Sobolev space of square-integrable functions on $\bbR^d$ with square-integrable weak derivatives, one defines the Dirichlet form attached to Brownian motion
\begin{equation}
\cE(f,f) = \frac{1}{2} \int \left|\nabla f(x)\right|^2\De x,
\end{equation}
and by polarization one defines furthermore
\begin{equation}
\cE(f,g) = \frac{1}{4}\left(\cE(f+g,f+g) - \cE(f-g,f-g) \right).
\end{equation}
Note that $\cE(\cdot,\cdot)$ defined in this way is bilinear and its definition can be extended to the space of all weakly differentiable functions with finite Dirichlet energy.  
Combining Theorem 4.3.3, p.\ 171 of~\cite{fukushima2010dirichlet} with Theorem 2.1.5, p.\ 72 of the same reference, one knows that for any bounded and either open or closed set $B \subseteq \bbR^d$, $h_B$ is in this extended Dirichlet space of $(\cE,H^1(\bbR^d))$ (see Example 1.5.3 in \cite{fukushima2010dirichlet} for a characterization of this space) and it holds that
\begin{equation}
\label{eq:EqualityCapacityDirichlet}
\capa(B) = \cE(h_B, h_B).
\end{equation}
We also note here, that if $f,g$ are in the extended Dirichlet space of $(\cE,H^1(\bbR^d))$ and $f$ has compact support, one has
\begin{equation}
\label{eq:CSEnergies}
|\langle f,g \rangle|^2 \leq E(f) \cE(g,g),
\end{equation}
see below \eqref{eq:PushDown} for the definition of the energy $E(f)$. To see this inequality, one can for instance show it first in the case where $f,g$ are smooth and compactly supported, and then use an approximation argument (compare also with Lemma 1.5.3, p.\ 39 of~\cite{fukushima2010dirichlet}).  \vspace{\baselineskip} 

We now recall an asymptotic lower bound from~\cite{nitzschner2017solidification} on the trapping probability of Brownian motion starting in a non-empty compact set $A \subseteq \mathbb{R}^d$ by surrounding `porous interfaces', together with a corresponding asymptotic lower bound on the Brownian capacity of such porous interfaces. These estimates will be pivotal in the derivation of the bounds \eqref{eq:Localization}, \eqref{eq:PushDown} and \eqref{eq:ProfileDescription} of the following chapters. Let $U_0$ be a non-empty Borel subset of $\bbR^d$ with complement $U_1 = \mathbb{R}^d \setminus U_0$ and boundary $ \partial U_0 = \partial U_1$. One measures the local density of $U_1$ at $x \in \bbR^d$ in dyadic scales
\begin{equation}
\widehat{\sigma}_{\ell}(x) = \frac{|B_{\infty}(x,2^{-\ell})\cap U_1|}{|B_{\infty}(x,2^{-\ell})|}, \qquad \ell \in \bbZ,
\end{equation}
 where $| \cdot |$ stands for the Lebesgue measure on $\mathbb{R}^d$. We furthermore introduce for $\ell_\ast$ non-negative integer and for a non-empty compact subset $A$ of $\mathbb{R}^d$
 \begin{equation}
 \label{eq:SegmentationClass}
 \mathcal{U}_{\ell_\ast, A} = \, \begin{minipage}{0.6\linewidth}
  the collection of bounded Borel subsets $U_0 \subseteq \mathbb{R}^d$ with $\widehat{\sigma}_{\ell}(x) \leq \tfrac{1}{2}$  for all $x\in A$  and $\ell \geq \ell_\ast$.
 \end{minipage} 
\end{equation}  
For a given non-empty Borel subset $U_0 \subseteq \bbR^d$, $\varepsilon > 0$ and $\eta \in (0,1)$ we consider the following class of `porous interfaces' 
\begin{equation}
\label{eq:ClassofporousInterf}
\begin{split}
\mathscr{S}_{U_0,\varepsilon,\eta} &= \ \text{the class of }\Sigma \subseteq \mathbb{R}^d \text{ compact with } W_z[H_\Sigma < \tau_\varepsilon] \geq \eta, \text{ for all } z \in \partial U_0.
\end{split}
\end{equation}
Essentially, $\varepsilon$ controls the distance of the porous interface $\Sigma$ from $\partial U_0$ and $\eta$ corresponds to the strength with which it is `felt'. With this, we can quote the solidification estimate from (3.3) of Theorem 3.1 in~\cite{nitzschner2017solidification}, which provides for $\eta \in (0,1)$ in the limit $\varepsilon/ 2^{-\ell_\ast}$ going to zero uniform controls on the probability that Brownian motion starting in $A$ hits $\Sigma \in \cS_{U_0,\varepsilon,\eta}$, when $U_0 \in \cU_{\ell_\ast,A}$,
\begin{equation}
\label{eq:SolidificationProb}
\lim_{u \rightarrow 0} \sup_{\varepsilon \leq u2^{-\ell_\ast}} \sup_{U_0 \in \mathcal{U}_{\ell_\ast,A}} \sup_{\Sigma \in \mathscr{S}_{U_0,\varepsilon,\eta}} \sup_{x\in A} W_x[H_\Sigma = \infty] = 0.
\end{equation}
This roughly means that a Brownian motion starting in $A$ cannot escape any surrounding porous interface. In addition to~\eqref{eq:SolidificationProb}, we need the related capacity lower bound
\begin{equation}
\label{eq:SolidificationEstimate}
\lim_{u \rightarrow 0} \inf_{\varepsilon \leq u2^{-{\ell_\ast}}} \inf_A \inf_{U_0 \in \mathcal{U}_{\ell_\ast,A}} \inf_{\Sigma \in \mathscr{S}_{U_0,\varepsilon,\eta}} \frac{\text{cap}(\Sigma)}{\text{cap}(A)} = 1,
\end{equation}
where $A$ varies in the class of non-empty compact subsets of $\bbR^d$ with positive capacity, see (3.15) of Corollary 3.4 in~\cite{nitzschner2017solidification}. \vspace{\baselineskip}

Finally, we state and prove two lemmas that arise from the solidification estimate~\eqref{eq:SolidificationProb}. We start with Lemma~\ref{prop:harmonic_bound} below, which can be seen as a generalization of~\eqref{eq:SolidificationProb} to $\bbR^d$ and will be employed in Proposition~\ref{prop:3.3} (cf.\ \eqref{eq:pointwisesolid}), to obtain an upper bound on the variance of a certain Gaussian field which is used to encode the event under the probability in \eqref{eq:Prop3_3_claim}. 
\begin{lemma}\label{prop:harmonic_bound} Consider a non-empty compact $A\subseteq \bbR^d$ and $\eta \in (0,1)$. Then,
\begin{equation}\label{eq:solidif_escaping}
\limsup_{u\to 0}  \sup_{\varepsilon \leq u 2^{-\ell_*}} \sup_{U_0 \in \cU_{\ell_*,A}} \sup_{\Sigma \in \cS_{U_0, \varepsilon, \eta}}
\sup_{x\in\bbR^d}\Big( W_x[ H_\Sigma  = \infty] - W_x[H_A = \infty]\Big) = 0.
\end{equation} 
\end{lemma}
\begin{proof}
We start by showing that the left hand side of~\eqref{eq:solidif_escaping} is less or equal than $0$.
For $x\in \bbR^d$, one has
\begin{equation}\label{eq:step1}
\begin{aligned}
  W_x[H_\Sigma &= \infty] \leq W_x[H_\Sigma = \infty, H_A<\infty] + W_x[H_A = \infty] \\
  & =E^{W_x}\big[W_x[H_\Sigma = \infty, H_A<\infty | Z_{H_A}]\big] + W_x[H_A = \infty],
\end{aligned}
\end{equation}
where $E^{W_x}$ is the expectation with respect to $W_x$.
We observe that on the event $\{H_A< H_\Sigma, H_A<\infty\}$ we have $H_\Sigma = H_\Sigma\circ\theta_{H_A} + H_A$, where $\theta_{H_A}$ is the canonical shift by $H_A$. Thus, an application of the strong Markov property yields
\begin{equation}
\label{eq:step2}
\begin{split}
E^{W_x}\big[W_x[H_\Sigma = &\infty, H_A<\infty | Z_{H_A}]\big]  \\
& \leq E^{W_x}\big[\IND_{\{ H_A<\infty \}} W_{Z_{H_A}}[H_\Sigma=\infty]\big] \leq \sup_{y\in A}W_{y}[H_\Sigma=\infty].
\end{split}
\end{equation}
Combining~\eqref{eq:step1} and~\eqref{eq:step2}, we obtain that
\begin{equation}\label{eq:step3}
\sup_{x\in \bbR^d}\Big(W_x[H_\Sigma = \infty] - W_x[H_A = \infty]\Big)\leq  \sup_{y\in A}W_{y}[H_\Sigma=\infty].
\end{equation}
By taking limits in~\eqref{eq:step3} and by using~\eqref{eq:SolidificationProb} we arrive at
\begin{equation*}
\limsup_{u\to 0}  \sup_{\varepsilon \leq u 2^{-\ell_*}} \sup_{U_0 \in \cU_{\ell_*,A}} \sup_{\Sigma \in \cS_{U_0, \varepsilon, \eta}}
\sup_{x\in\bbR^d}\Big( W_x[ H_\Sigma  = \infty] - W_x[H_A = \infty]\Big) \leq 0.
\end{equation*}
Finally, we show that the limit equals $0$ by providing a lower bound with a particular choice. Fix any $z\in A$, $\ell_*>0$ and let $A_{\ell_*}$ be the set of points at sup-distance at most $2^{-\ell_*}$ from $A$.
Then for $\varepsilon \leq u 2^{-\ell_*}$, choosing $U_0 = A_{\ell_*} = \Sigma$, we see that $U_0\in \cU_{\ell_*,A}$  and $\Sigma \in \cS_{U_0, \varepsilon, \eta}$.
In addition, $W_z[H_{A_{\ell_*}} = \infty] = 0 = W_z[H_{A} = \infty]$.
Letting $\varepsilon$ go to $0$ shows that the limit must be non-negative.
\end{proof}
In the remainder of this section we will show in Lemma~\ref{thm:SolidifDirichlet} below that in the limit $\varepsilon / 2^{-\ell_\ast} \rightarrow 0$, the Dirichlet energy of $h_A - h_\Sigma$ is bounded from above by the capacity difference $\capa(\Sigma) - \capa(A)$, uniformly over all compacts $A \subseteq \bbR^d$ and all porous interfaces $\Sigma$. This result will be needed in the proof of Theorem~\ref{thm:MainUpperBound} (see~\eqref{eq:BoundSigma_LargeDirichletform}--\eqref{eq:BoundSplit_1}), to rule out, with high probability, the existence of atypical interfaces of bad boxes.
\begin{lemma}
\label{thm:SolidifDirichlet} Let $\eta\in (0,1)$ be fixed. Then,
\begin{equation}
\label{eq:BoundDirichletForm}
  \limsup_{u\to 0}  \sup_{\varepsilon \leq u 2^{-\ell_*}} \sup_{A} \sup_{U_0 \in \cU_{\ell_*,A}} \sup_{\Sigma \in \cS_{U_0, \varepsilon, \eta}}
  \bigg [\cE(h_A - h_\Sigma,h_A - h_\Sigma) - (\capa(\Sigma) - \capa(A))\bigg] = 0,
\end{equation}
where $A$ varies in the class of non-empty compact subsets of $\bbR^d$. 
\end{lemma}
\begin{proof} First we notice that by~\eqref{eq:EqualityCapacityDirichlet} and the bilinearity of the Dirichlet form, one has
\begin{equation}
0 \leq \cE(h_A - h_\Sigma,h_A - h_\Sigma)  = \capa(A) + \capa(\Sigma) - 2 \cE(h_A,h_\Sigma),
\end{equation}
so to conclude it suffices to show that
\begin{equation}
\label{eq:BoundDirichletForm2}
    \varliminf_{u\to 0} \inf_{\varepsilon \leq u 2^{-\ell_*}} \inf_{A}  \inf_{U_0 \in \cU_{\ell_*,A}} \inf_{\Sigma \in \cS_{U_0, \varepsilon, \eta}} \Big(\cE(h_A,h_\Sigma) -\capa(A)\Big) = 0.
\end{equation}
Using Theorem 2.2.5, p.\ 86 of~\cite{fukushima2010dirichlet}, we have that $\cE(h_\Sigma,h_A) = \int h_\Sigma(x) e_A(dx)$, where for a compact set $B \subseteq \bbR^d$, $e_B$ denotes the equilibrium measure of $B$, which has total mass $\capa(B)$ and is supported on $B$, see for instance p.57 of~\cite{port2012brownian}. From this, we immediately get
\begin{equation}
\label{eq:Dirichlet_A_Sigma}
\cE(h_A,h_\Sigma) = \int W_x[H_\Sigma < \infty] e_A(dx) \geq \inf_{x \in A} W_x[H_\Sigma < \infty] \capa(A),
\end{equation}
where we have used that the set of points $x \in \bbR^d$ with $W_x[\widetilde{H}_\Sigma < \infty] \neq  W_x[H_\Sigma < \infty]$ has null $e_A$-measure (see also (3.18) in~\cite{nitzschner2017solidification} and the argument following it). We therefore obtain
\begin{equation}
\varliminf_{u\to 0} \inf_{\varepsilon \leq u 2^{-\ell_*}}  \inf_{U_0 \in \cU_{\ell_*,A}} \inf_{\Sigma \in \cS_{U_0, \varepsilon, \eta}} \Big(\cE(h_A,h_\Sigma) -\capa(A)\Big) \geq 0,
\end{equation}
having used~\eqref{eq:SolidificationProb}. To show that the limit is equal to $0$, let $A_{\ell_\ast}$ be the set of points at sup-distance at most $2^{-\ell_\ast}$ from $A$. For $\varepsilon \leq u 2^{-\ell_\ast}$, choose $U_0 = A_{\ell_\ast} = \Sigma$ and use that $\cE(h_A,h_{A_{\ell_\ast}}) \rightarrow \capa(A)$ as $\ell_{\ast} \rightarrow \infty$, by~\eqref{eq:Dirichlet_A_Sigma} and dominated convergence. The claim then follows by letting $\varepsilon \rightarrow 0$ and $\ell_\ast \rightarrow \infty$ such that $\varepsilon 2^{\ell_\ast}$ tends to $0$. 
\end{proof}

\section{Entropic push-down by disconnection}

In this section we derive in Theorem~\ref{thm:pushdown} an asymptotic upper bound on the probability of the event that the level-set below $\alpha$ disconnects $A_N$ from $S_N$ \emph{and} that  $\langle \bbX_N,\eta\rangle$, that is,
the average of the Gaussian free field over some continuous and compactly supported function $\eta : \bbR^d \rightarrow [0,\infty)$, becomes bigger than $\langle  \cH^\alpha_{\mathring{A}},\eta\rangle + \Delta$, for some $\Delta>0$ (recall that $\cH^\alpha_{\mathring{A}} = -(\overline{h}-\alpha) h_{\mathring{A}}$ with $h_{\mathring{A}}$ the harmonic potential of $\mathring{A}$, see \eqref{eq:HarmonicPot}). 

As a consequence of Theorem~\ref{thm:pushdown}, when  $A$ is such that $\capa(A)=\capa(\mathring{A})$, one readily derives Corollary~\ref{thm:Corollary32} using~\eqref{eq:Equality_crit_par}.  From a qualitative perspective, Corollary~\ref{thm:Corollary32} should be understood as follows: the occurrence of the event $\cD^\alpha_N$ pushes the Gaussian free field down to a level smaller or equal than $\cH^\alpha_{\mathring{A}}(x/N) = -(h_*-\alpha) h_{\mathring{A}}(x/N)$, $x \in \bbZ^d$.
In particular, the effects of the disconnection event (which depends only on the values of the field in $B(0, MN)$) are felt globally on $\bbZ^d$. This is due to the slow decay of correlations for the Gaussian free field in dimensions $d\geq 3$.
\vspace{\baselineskip}

Throughout this and the next sections, we assume that $A\subseteq \bbR^d$ is a compact set with non-empty interior such that $A\subseteq \mathring{B}_\infty(0,M)$ for some fixed $M>0$. Also, recall the definition~\eqref{eq:measureGFF} of $\bbX_N$ and our convention on constants, given at the end of Section 1. We are ready to state the main result of this section. 
\begin{theorem}\label{thm:pushdown}
Consider $\Delta > 0$, $\alpha < \overline{h}$ and a continuous, compactly supported function $\eta : \bbR^d \rightarrow [0,\infty)$. Then one has the asymptotic upper bound
\begin{equation}\label{eq:thm31}
\begin{aligned}
\limsup_{N\to \infty} \frac{1}{N^{d-2}} \log  \bbP&\Big[ \langle  \bbX_N,\eta \rangle \geq\langle \cH^\alpha_{\mathring{A}},\eta \rangle +\Delta; \cD^{\alpha}_N\Big] \\
& \leq -\frac{1}{2d} (\overline{h}-\alpha)^2 \capa(\mathring{A}) - \frac{\Delta^2}{2 d} \frac{1}{E(\eta)},
\end{aligned}
\end{equation}
where $E(\eta) = \int \eta(x)g_{BM}(x,y)\eta(y)
\De x \De y$ is the energy associated to the function $\eta$ (we adopt the convention that the right-hand side is $-\infty$ if $\eta=0$). 
\end{theorem}
Before we move towards the proof of the above Theorem, we state a corollary that gives insight into the conditional measure $\bbP[ \ \cdot \ | \cD^\alpha_N]$ and makes use of~\cite{duminil2019equality}.
\begin{corollary}
\label{thm:Corollary32}
Consider $\Delta$, $\alpha$, $\eta$ as in Theorem~\ref{thm:pushdown}
 and assume that $A$ is regular in the sense that $\capa(A) = \capa(\mathring{A})$. Then, one has
\begin{equation}
\limsup_{N \rightarrow \infty} \frac{1}{N^{d-2}}\log \bbP\left[\langle \bbX_N,\eta \rangle \geq \langle \cH^\alpha_A,\eta \rangle + \Delta | \cD^\alpha_N \right] \leq - \frac{\Delta^2}{2 d} \frac{1}{E(\eta)}.
\end{equation} 
\end{corollary} 
\begin{proof}
It holds that
\begin{equation}
  \begin{aligned}
    \varlimsup_{N \rightarrow \infty} &\frac{1}{N^{d-2}}  \log \bbP\left[\langle \bbX_N,\eta \rangle \geq \langle \cH^\alpha_A,\eta \rangle + \Delta| \cD^\alpha_N \right]  \\ &\leq \varlimsup_{N \rightarrow \infty} \frac{1}{N^{d-2}} \log \bbP\left[\langle \bbX_N,\eta \rangle \geq \langle \cH^\alpha_A,\eta \rangle + \Delta; \cD^\alpha_N \right]  - \varliminf_{N \rightarrow \infty}  \frac{1}{N^{d-2}}\log \bbP[\cD^\alpha_N].    
  \end{aligned}
\end{equation}
The result now follows directly by the equality of the critical parameters~\eqref{eq:Equality_crit_par}, see~\cite{duminil2019equality}, combining~\eqref{eq:thm31} with the lower bound~\eqref{eq:DiscLowerBound} and noting that  $\capa(A) = \capa(\mathring{A})$ implies $h_A = h_{\mathring{A}}$ Lebesgue-a.e. since
\begin{equation}
\begin{split}
  \cE(h_A - h_{\mathring{A}},h_A - h_{\mathring{A}}) & = \capa(A) + \capa(\mathring{A}) - 2\cE(h_A,h_{\mathring{A}}) \\
  & \leq \capa(A) - \capa(\mathring{A}) = 0. \qedhere
  \end{split}
\end{equation}
\end{proof}

In order to prove Theorem \ref{thm:pushdown}, we shall devise a coarse-graining procedure similar to the one developed in Section 3 of~\cite{nitzschner2018entropic} (see also Section 4 of~\cite{nitzschner2017solidification}). Let us therefore introduce further notation and recall the construction in the above references. For the convenience of the reader, we will reproduce here the main steps in some detail, and cite the references for further explanations. We stress that the coarse-graining performed below is going to be used also in the proof of Theorem~\ref{thm:MainUpperBound}. 
\vspace{\baselineskip} 

Let $\delta<\gamma<\overline{h}$, and select a sequence $(\gamma_N)_{N \geq 1}$ of numbers in $(0,1]$ fulfilling the conditions (4.18) of~\cite{nitzschner2017solidification}, in particular,
\begin{equation}\label{eq:gamma_N}
\gamma_N \rightarrow 0,\quad \gamma_N^{\frac{d+1}{2}}/(N^{2-d}\log N) \to \infty,\quad \text{as }N \rightarrow \infty.
\end{equation}
Moreover, define the two scales
\begin{equation}
\label{eq:ScalesDefinition}
L_0 = \Big\lfloor (\gamma_N^{-1} N \log N)^{\tfrac{1}{d-1}} \Big\rfloor, \qquad \widehat{L}_0 = 100 d \left\lfloor \sqrt{\gamma_N} N \right\rfloor,
\end{equation}
as well as the lattices
\begin{equation}
\bbL_0  = L_0 \bbZ^d,  \qquad \widehat{\bbL}_0 = \tfrac{1}{100d} \widehat{L}_0\bbZ^d = \left\lfloor \sqrt{\gamma_N} N \right\rfloor \bbZ^d.
\end{equation}
Furthermore, we introduce for $K \geq 100$ and $z \in \bbL_0$ the boxes
\begin{equation}
\label{eq:BoxesDefinition}
\begin{split}
B_z &= z + [0,L_0)^d \cap \bbZ^d \subseteq D_z = z + [-3L_0,4L_0)^d \cap \bbZ^d \\
& \subseteq U_z = z + [-KL_0+1,KL_0-1)^d \cap \bbZ^d.
\end{split}\end{equation}
The collection of boxes $U_z$, $z \in \bbL_0$ is used to decompose the Gaussian free field according to \eqref{eq:HarmonicAverage} and \eqref{eq:LocalField}. In particular, we write $\varphi = \psi^z + h^z$ where $h^z = h^{U_z}$ and $\psi^z = \psi^{U_z}$.  

For $\delta < \gamma$ in $(\alpha,\overline{h})$, there is a notion of a box $B_z$, $z \in \bbL_0$ being $\psi$-good at levels $\delta < \gamma$, see (5.7), (5.8) of \cite{sznitman2015disconnection}, which in essence means that the set $B_z \cap \lbrace \psi^z \geq \gamma \rbrace$ contains a connected component of $|\cdot |_\infty$-diameter at least $L_0/10$, and for any neighboring box $B_{z'}$, any two connected components of $B_z \cap \lbrace \psi^z \geq \gamma \rbrace$ and $B_{z'} \cap \lbrace \psi^{z'} \geq \gamma \rbrace$ with $|\cdot |_\infty$-diameter at least $L_0/10$ are connected in $D_z \cap \lbrace \psi^z \geq \delta \rbrace$.  A box that is not $\psi$-good at levels $\delta < \gamma$, is called $\psi$-bad at these levels. 

We will also need the notion of a box $B_z$, $z \in \bbL_0$ being $h$-good at level $a > 0$, which means that $\inf_{D_z} h^z > - a$ (for us $a = \delta-\alpha>0$ will be a natural choice, and eventually, we will send $\delta$ and $\gamma$ to $\overline{h}$). Again, a box that is not $h$-good at a level $a$ is called $h$-bad at this level. \vspace{\baselineskip}

\paragraph{\emph{Outline of the proof}} The proof of Theorem \ref{thm:pushdown} will be given in a multi-step procedure. We will give a detailed outline and explain in an informal fashion some results from~\cite{sznitman2015disconnection} and~\cite{nitzschner2018entropic} that enter the proof. 
\vspace{0.3\baselineskip}

\noindent \textbf{1.} \emph{Effective disconnection event}:

For large $N$ on the disconnection event $\cD^\alpha_N$, we aim at extracting an interface of `blocking' $L_0$-boxes, all either $h$-bad at level $a = \delta - \alpha$ or $\psi$-bad at levels $\delta < \gamma$, located between $A_N$ and the complement of $B(0,(M+1)N)$. Informally, such an interface exists since a sequence of neighboring boxes, that are both $\psi$-good and $h$-good and that connect $A_N$ to the complement of $B(0,(M+1)N)$, would provide a path in $E^{\geq \alpha}$ joining $A_N$ and $S_N$ (see Figure~\ref{fig:path} below).
  \begin{figure}[h]
  \centering
  \includegraphics[width=0.8\textwidth]{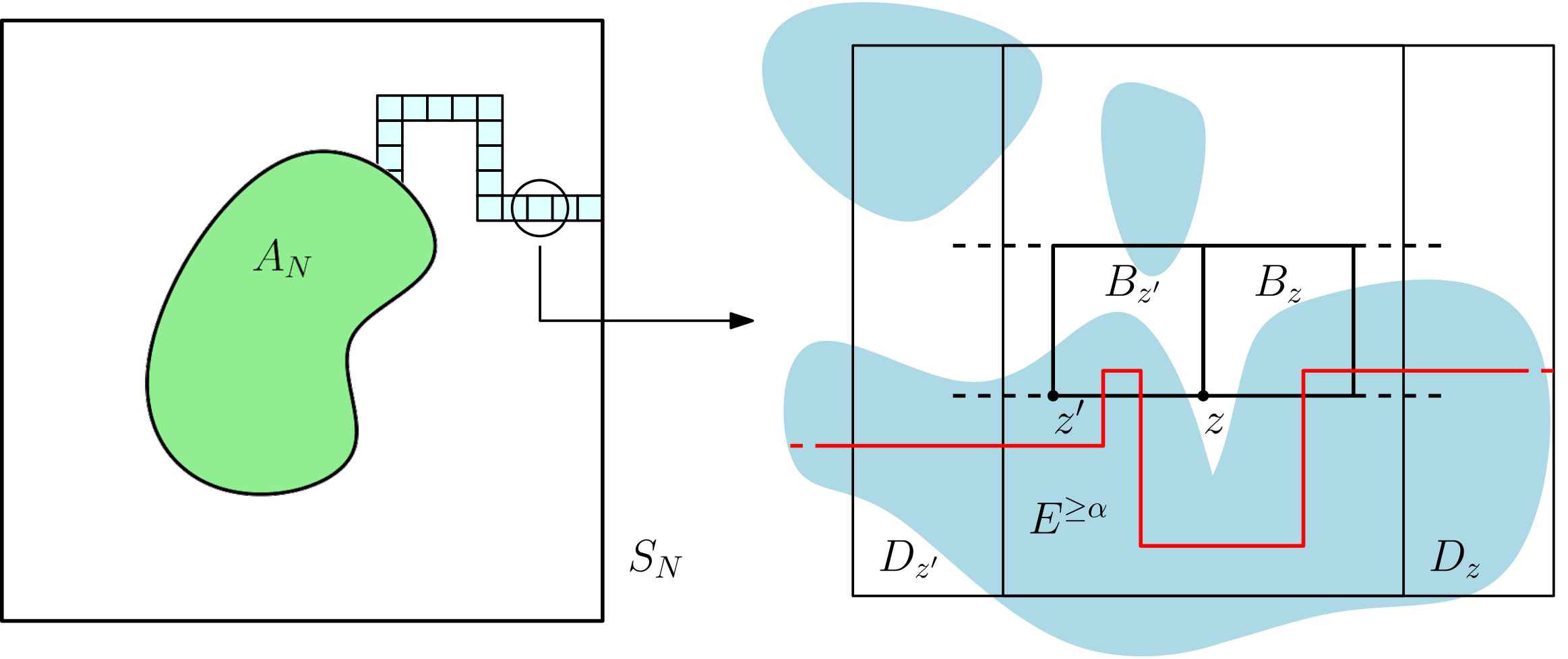}
  \caption{Informal picture for a realization of $(\cD^\alpha_N)^c$: If two neighboring boxes $B_z$ and $B_{z'}$ with $|z - z'| = L_0$ are both $\psi$-good at levels $\delta < \gamma$ and $h$-good at level $a$, one can find a path in $E^{\geq \delta - a} \cap D_z$ starting in $B_z$ and ending in $B_{z'}$ (where $\delta - a = \alpha$). A sequence of such neighboring boxes, connecting $A_N$ to the complement of $B(0,(M+1)N)$, provides a path in $E^{\geq \alpha}$ joining $A_N$ and $S_N$ (on the left-hand side).} 
  \label{fig:path}
 \end{figure}

However, by the independence of the local fields associated to disjoint boxes $U_z$, $U_{z'}$ (for $|z - z'|_\infty$ large enough), one can essentially rule out (using standard bounds on sums of independent Bernoulli random variables, see Section 5 of~\cite{sznitman2015disconnection}) that many boxes in the blocking interface are $\psi$-bad at levels $\delta < \gamma$. 

More precisely, one can define a `bad' event $\cB_N$ (corresponding to the existence of many $\psi$-bad boxes at levels $\delta < \gamma$ within a box $B(0,10(M+1)N)$) having negligible probability for our purposes, see~\eqref{eq:BadEvent} and~\eqref{eq:SuperExponentialBound} and work on the \textit{effective disconnection event} $ \widetilde{\cD}^\alpha_N = \cD^\alpha_N \setminus \cB_N$ henceforth, essentially only keeping track of $h$-bad boxes. 
\vspace{0.3\baselineskip}

\noindent \textbf{2.} \emph{Coarse Graining}:

The position of these $h$-bad boxes is encoded in a set $\cK_N$ of values of a random variable $\kappa_N$ defined on $\widetilde{\cD}^\alpha_N$ giving rise to the coarse-graining of the event $\widetilde{\cD}^\alpha_N$ into sub-events $\cD_{N,\kappa}$, $\kappa \in \cK_N$, see~\eqref{eq:kappaN_Definition} and \eqref{eq:coarseGraining} below. 
This coarse-graining, which we take from~\cite{nitzschner2018entropic} (see also~\cite{nitzschner2017solidification}) is of low combinatorial complexity $\exp\{o(N^{d-2}) \}$, and thus allows us to reduce the problem of finding an asymptotic large deviation upper bound on the probability of the event under the probability in \eqref{eq:thm31} to finding an upper bound on the probability of the event $\cD_{N,\kappa} \cap \{ \langle \bbX_N,\eta \rangle \geq \langle  
\cH^\alpha_{\mathring{A}},\eta \rangle + \Delta \}$, uniformly in $\kappa \in \cK_N$, see~\eqref{eq:CoarseGrainingBound}. 

The coarse-graining procedure can be described informally as follows: We use boxes of the intermediate scale $\widehat{L}_0$ (see~\eqref{eq:ScalesDefinition}) and select a region where the `interface' of $h$-bad $L_0$-boxes has a non-degenerate density within those larger $\widehat{L}_0$-boxes. 
This selection is recorded in a set $\widehat{\cS}_N$ of points of $\widehat{\bbL}_0$.  A scaled $\bbR^d$-filling of these $\widehat{L}_0$-boxes is used to define a set $U_0$  (see \eqref{eq:UnionBound}) that acts as segmentation of the porous interface we construct.
Within a sparse subset of these $\widehat{L}_0$-boxes (corresponding to $\widetilde{\cS}_N\subseteq \widehat{\cS}_N$), one can then extract a `porous interface' with substantial presence of $h$-bad $L_0$-boxes and mutual distance bigger or equal to $4KL_0$ ($C$ is the set of all selected boxes). 
We refer to~\eqref{eq:UnionBound} for the concrete definition of the sets in question. The picture is illustrated in Figure~\ref{fig:CG} below. 
\begin{figure}[h]\label{fig:CG}
  \centering
  \includegraphics[width=0.8\textwidth]{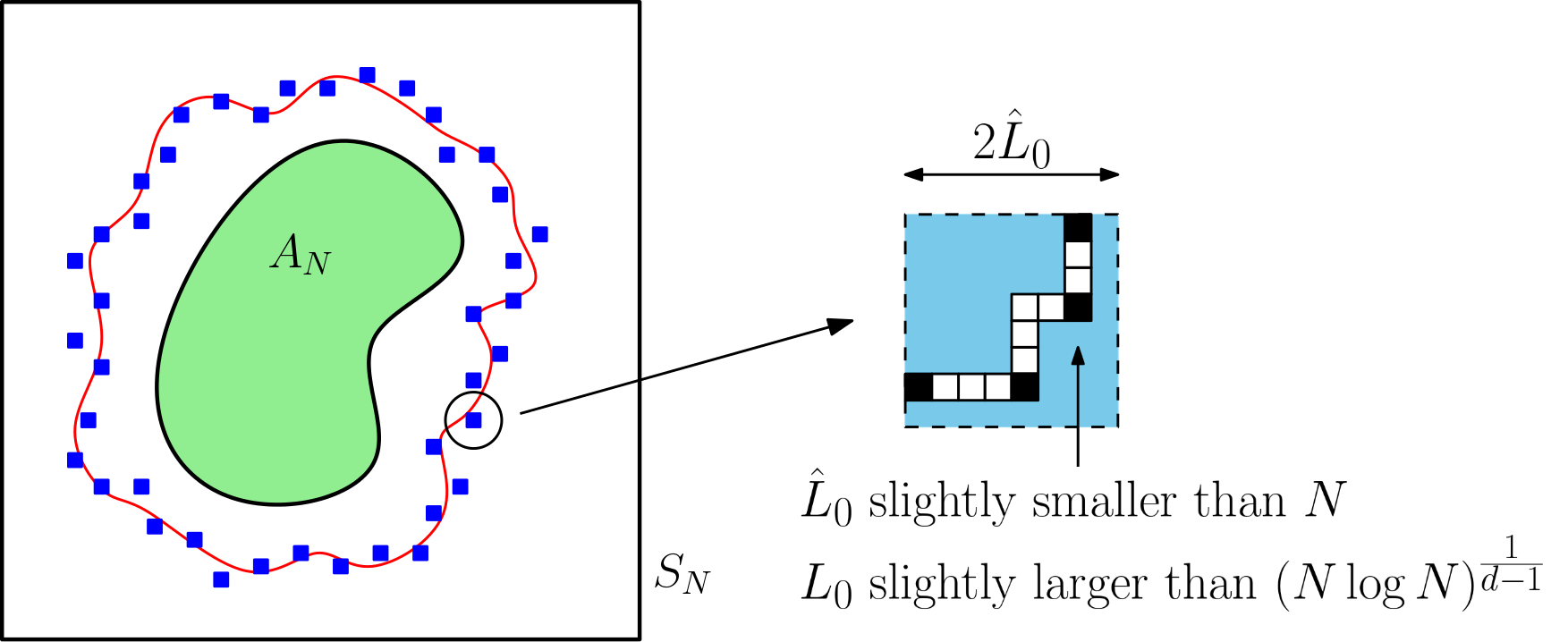}
  \caption{Informal description of the geometric picture corresponding to a choice of $\kappa \in \cK_N$: Within each of the selected boxes of size $2\widehat{L}_0$ (on the left-hand side), a non-degenerate number of boxes of size $L_0$ is selected (in black, in the right-hand side).}
 \end{figure}
\vspace{0.3\baselineskip}

\noindent \textbf{3.} \emph{Solidification estimates}:

In this step, we provide the solidification estimates for the `porous interfaces' which are associated to $\kappa \in \cK_N$. We rely here on the fact that the scaled $\bbR^d$-filling $\Sigma$ of the set of blocking `bad' boxes $C$ associated to any $\kappa \in \cK_N$ can be treated as a porous deformation of the boundary of the `segmentation', $ \partial U_0$. The application of solidification estimates is twofold:
\begin{itemize}
\item We use the pointwise solidification result, Lemma~\ref{prop:harmonic_bound} in order to bound from below the macroscopic averages of the functions $x \mapsto P_x[H_C < \infty]$ by the macroscopic averages of $h_{\mathring{A}}$, uniformly in $\kappa \in \cK_N$, see~\eqref{eq:pointwisesolid}. 
\item We also argue that the capacity lower bound \eqref{eq:SolidificationEstimate} can be applied to the scaled $\mathbb{R}^d$-filling $\Sigma$ of the collection of $L_0$-boxes $C$, giving the uniform asymptotic lower bound~\eqref{eq:capacitysol}.
\end{itemize}
\vspace{0.3\baselineskip}

\noindent \textbf{4.} \emph{Borell--TIS inequality}:

In this step, we develop the pivotal Proposition \ref{prop:3.3}, which gives a uniform bound on the probability that in the collection $C$ of $L_0$-boxes (associated to $\kappa \in \cK_N$) all are $h$-bad (at level $s$) \textit{and} $\langle \bbX_N, \eta \rangle \geq - s\langle \eta, h_{\mathring{A}}\rangle + \Delta$, where $s > 0$.  
This result, which extends Corollary 4.4 of~\cite{sznitman2015disconnection} and Lemma 4.2 in~\cite{nitzschner2018entropic}, is based on the Borell--TIS inequality for a certain Gaussian functional $\widehat{Z}_f$ and brings into play the capacity of the set $C$ of $h$-bad boxes. 

In the construction of $\widehat{Z}_f$, see~\eqref{eq:DefZ_hat}, we introduce a parameter $\beta > 0$ that informally interpolates between a Gaussian functional $Z_f$ which captures the effect of the $h$-bad boxes on one hand and the macroscopic average $\langle \bbX_N, \eta \rangle$ on the other hand. Later, we optimize in this parameter in order to obtain the \textit{additional cost} that comes with enforcing an absence of the entropic push present under disconnection. Roughly speaking, the application of the Borell--TIS inequality naturally brings into play the variance of $\widehat{Z}_f$, which has three parts (see~\eqref{eq:expansion_variance}):
\begin{itemize}
\item the variance of $Z_f$, which can be bounded by relying on the fact that the boxes in $C$ have a distance bigger or equal to $4KL_0$, see~\eqref{eq:UpperBoundZf}.
\item the variance of $\langle \bbX_N, \eta \rangle$, whose asymptotics as $N \rightarrow \infty$ can be calculated explicitly, see~\eqref{eq:HNbound}. 
\item a covariance term between $Z_f$ and $\langle \bbX_N, \eta \rangle$, which measures the `interaction of the macroscopic average' of the Gaussian free field, with the potential generated by the interface $C$, see~\eqref{eq:EstimationGN} and below. To find an adequate lower bound, we rely on the fact that $C$ has a `small volume' by construction, to exclude boundary effects, but has a significant presence around $A_N$, for which the pointwise solidification result of the previous step is crucial.\end{itemize}
 With Proposition~ \ref{prop:3.3} at hand, we finish the rest of the proof by using the capacity lower bounds for the sets $C$ associated to $\kappa \in \cK_N$ and optimizing in $\beta > 0$.

\vspace{0.3\baselineskip}

\begin{proof}[Proof of Theorem \ref{thm:pushdown}] As anticipated in the outline above we will perform a four-step procedure to derive our result.

\vspace{\baselineskip}
\noindent\textbf{Step 1: Effective disconnection event.}
In what follows we will assume
\begin{equation}
  \alpha + a = \delta(<\gamma<\overline{h}).
\end{equation}
 We proceed as in \cite{nitzschner2018entropic}, Section 3 and define the random subset 
\begin{equation}
\cU^1 = \ \begin{minipage}{0.8\linewidth}
  the union of all $L_0$-boxes $B_z$ that are either contained in\\  $B(0,(M+1)N)^c$ or connected to an $L_0$-box in $B(0,(M+1)N)^c$ by
a path of $L_0$-boxes $B_{z_i}$, $0 \leq i \leq n$, all (except possibly the last one) $\psi$-good at levels $\delta < \gamma$ and $h$-good at level $a = \delta - \alpha$.
\end{minipage}
\end{equation}
One introduces the function
\begin{equation}
\widehat{\sigma}(x) = |\cU^1 \cap B(x,\widehat{L}_0) | / |B(x,\widehat{L}_0)|, \quad x \in \bbZ^d,
\end{equation}
which tracks the presence of $\cU^1$ within boxes $B(x,\widehat{L}_0)$, and the set $\widehat{\cS}_N$, that provides a `segmentation' of the interface of blocking $L_0$-boxes, namely 
\begin{equation}
\widehat{\cS}_N = \left\{x \in \widehat{\bbL}_0; \widehat{\sigma}(x) \in \left[ \tfrac{1}{4}, \tfrac{3}{4} \right] \right\}.
\end{equation} 
One then extracts from $\widehat{\cS}_N$ another random subset $\widetilde{\cS}_N$ such that
\begin{equation} 
\begin{minipage}{0.8\linewidth}$\widetilde{\cS}_N$ is a maximal subset of $\widehat{\cS}_N$ with the property that the $B(x,2\widehat{L}_0)$, $x \in \widetilde{\cS}_N$, are pairwise disjoint.
\end{minipage}
\end{equation}
(Informally, the boxes in the left hand side of Figure~\ref{fig:CG} correspond to $B(x,\widehat{L}_0)$, with $x\in\widetilde{\cS}_N$).
We also recall the `bad' event $\cB_N$ from (3.20) of \cite{nitzschner2018entropic}, which is defined as 
\begin{equation}
\label{eq:BadEvent}
\cB_N = \bigcup_{e \in \{e_1,...,e_d \} } \left\{ \begin{minipage}{0.6\textwidth}
  there are at least $\rho(L_0)(N_{L_0}/L_0)^{d-1}$ columns of $L_0$-boxes in the direction $e$ in $B(0,10(M+1)N)$ that contain a $\psi$-bad $L_0$-box at levels $\delta < \gamma$
\end{minipage}\right\},
\end{equation}
where $\rho(L)$ a suitable function depending on $\gamma, \delta$ and $K$ tending to $0$ as $L \rightarrow \infty$, $N_{L_0} = L_0^{d-1}/\log L_0$, and $\{e_1,...,e_d \}$ is the canonical basis of $\bbR^d$. It can be argued that one has the super-exponential bound
\begin{equation}
\label{eq:SuperExponentialBound}
\lim_{N \rightarrow \infty} \frac{1}{N^{d-2}} \log \bbP[\cB_N] = -\infty,
\end{equation}
see Proposition 5.2 of \cite{sznitman2015disconnection} and the proof of Lemma 4.2 of \cite{nitzschner2017solidification}. Therefore, the event $\cB_N$ is irrelevant at the order we are interested in. This allows us to introduce the \emph{effective} disconnection event 
\begin{equation}
\widetilde{\cD}^\alpha_N = \cD^\alpha_N \setminus \cB_N.
\end{equation}

\noindent\textbf{Step 2: Coarse graining.}
We set $\overline{K} = 4K$ and as in (4.39)--(4.41) of \cite{nitzschner2017solidification} (see also below (3.20) in \cite{nitzschner2018entropic}), we have that
\begin{equation}\label{effectiveevent}
  \begin{minipage}{0.85\textwidth}
    for large $N$, on $\widetilde{\cD}^\alpha_N$, for each $x \in \widetilde{\cS}_N$, one can find a collection $\widetilde{\cC}_x$ of points in $\bbL_0$ and $\widetilde{i}_x \in \{ 1,...,d \}$, such that the $L_0$-boxes $B_z$, $z\in \widetilde{\cC}_x$, intersect $B(x,\widehat{L}_0)$ and have $\widetilde{\pi}_x$-projection at mutual distance at least $\overline{K}L_0$, where $\widetilde{\pi}_x$ is the orthogonal projection on the set of points in $\bbZ^d$ with vanishing $\widetilde{i}_x$-coordinate. Moreover, $\widetilde{\cC}_x$ has cardinality $\big\lfloor\big( \tfrac{c'}{K} \tfrac{\widehat{L}_0}{L_0} \big)^{d-1} \big\rfloor$ and for each $z \in \widetilde{\cC}_x$, $B_z$ is $\psi$-good at level $\gamma < \delta$ and $h$-bad at level $a = \delta- \alpha$.\end{minipage}
\end{equation}
We remark that the existence of $\widetilde{\pi}_x$ and $\widetilde{\cC}_x$ for each $x \in \widetilde{\cS}_N$ on $\widetilde{\cD}^\alpha_N$ follows from the isoperimetric controls (A.3) -- (A.6), p. 480--481 of~\cite{deuschel1996surface}. Note that we choose here $\overline{K} = 4K$ (instead of $\overline{K} = 2K + 3$ as in \cite{nitzschner2017solidification} and \cite{nitzschner2018entropic}), since it will be necessary in the proof of the following proposition that the boxes $U_z$, $z \in \widetilde{\cC}_x$, $x \in \widetilde{\cS}_N$ are all at a large enough distance. This modification is only minor and does not change the validity of the coarse graining procedure. 

We now introduce the random variable $\kappa_N$ defined on $\widetilde{\cD}^\alpha_N$ with range $\cK_N$, 
\begin{equation}
\label{eq:kappaN_Definition}
\kappa_N = \big(\widehat{\cS}_N, \widetilde{\cS}_N, (\widetilde{\pi}_x, \widetilde{\cC}_x)_{x \in \widetilde{\cS}_N} \big),
\end{equation}
see below (4.41) of \cite{nitzschner2017solidification} or (3.19) of \cite{nitzschner2018entropic}. Moreover, we have based on a counting argument involving the choice of the scales $\widehat{L}_0$ and $L_0$ together with \eqref{effectiveevent}, that
\begin{equation}
\label{eq:smallCombComplexity}
|\cK_N| = \exp\{ o(N^{d-2}) \},
\end{equation} 
cf. (4.43) of \cite{nitzschner2017solidification}, which is the `small combinatorial complexity' we need. 
For later use, we define the following sets associated to a choice of $\kappa = (\widehat{\cS},\widetilde{\cS}, (\widetilde{\pi}_x,\widetilde{\cC}_x)_{x \in \widetilde{\cS}}) \in \cK_N$:
\begin{equation}
\label{eq:UnionBound}
\begin{cases}
\cC & = \bigcup_{x \in \widetilde{\cS}} \widetilde{\cC}_x \\
C & = \bigcup_{z \in \cC} B_z \subseteq \bbZ^d \\
\Sigma & = \frac{1}{N} \bigcup_{z \in \cC} \Big (z + [0,L_0]^d \Big ) \subseteq \bbR^d \\
U_1 & = \text{ the unbounded component of }\bbR^d \setminus \tfrac{1}{N}  \bigcup_{x \in \widehat{\cS}} B_\infty\Big(x, \tfrac{1}{50d} \widehat{L}_0 \Big)  \\
U_0 & = \bbR^d \setminus U_1. 
\end{cases}
\end{equation}
Essentially, $\Sigma$ and $U_0$ will later play the role of a `segmentation' and `porous interface' in the sense of \eqref{eq:SegmentationClass} and \eqref{eq:ClassofporousInterf} (with the choice $\varepsilon = 10 \tfrac{\widehat{L}_0}{N}$).
With this preparation, one has the coarse-graining
\begin{equation}\label{eq:coarseGraining}
\widetilde{\cD}^\alpha_N = \bigcup_{\kappa \in \cK_N} \cD_{N,\kappa}, \ \text{ where } \cD_{N,\kappa} = \widetilde{\cD}^\alpha_N \cap \{ \kappa_N = \kappa \}.
\end{equation}
What is crucial is that on the event $\cD_{N,\kappa}$, in view of~\eqref{effectiveevent}, all $B_z$ with $z\in \cC$ are $h$-bad at level $a$ and at mutual distance $\geq \overline{K}L_0$, for large $N$. In particular for $\alpha + a = \delta$,
\begin{equation}\label{eq:badboxes}
  \cD_{N,\kappa} \subseteq \bigcap_{z \in \cC} \{\inf_{D_z} h^z \leq -a \}.
\end{equation}
Thus, applying a union bound and using the super-exponential bound \eqref{eq:SuperExponentialBound}, one finds that the left-hand side of \eqref{eq:thm31} can be bounded as follows:
\begin{equation}
    \label{eq:CoarseGrainingBound}
    \limsup_{N \rightarrow \infty} \frac{1}{N^{d-2}} \log \bbP[\cA^{\overline{h}-\alpha,\Delta}_N \cap \cD^\alpha_N] \leq \limsup_{N \rightarrow \infty} \sup_{\kappa \in \cK_N} \frac{1}{N^{d-2}} \log \bbP[\cA^{\overline{h}-\alpha,\Delta}_N \cap \cD_{N,\kappa}],
\end{equation}
where we have used the notation
\begin{equation}
\label{eq:DefAsDelta_N}
\cA^{s,\Delta'}_N = \big\{\langle  \bbX_N,\eta \rangle \geq -s \langle  h_{\mathring{A}},\eta \rangle + \Delta'\big\}, \quad \text{for } s, \Delta' \in \bbR.  
\end{equation}

We will provide an asymptotic upper bound on the probability on the right hand side of~\eqref{eq:CoarseGrainingBound} in Step 4. This bound will bring into play the capacity of the set $C$ of $h$-bad boxes attached to $\kappa\in \cK_N$, which we study in the next step.

\vspace{\baselineskip}
\noindent\textbf{Step 3: Solidification estimates.} In this step we derive some bounds on the volume and capacity of the interface $C$ attached to $\kappa\in \cK_N$.  
We will show that, in the limit as $N,K\to \infty$, their (scaled) $\mathbb{R}^d$-fillings act as blocking interfaces for the set $A$. Despite the fact that they have small volume compared to $A_N$, we will show that their capacity is of the same order as the capacity of $A_N$.

We start by studying the capacity and macroscopic averages of the discrete harmonic potential of $C$ attached to $\kappa\in \cK_N$. Recall the set $\Sigma$ for any $\kappa\in \cK_N$ from~\eqref{eq:UnionBound}, also for $\kappa \in \cK_N$ set
\begin{equation}
\widetilde{\Gamma} = \bigcup_{z \in \cC} (z + [\tfrac{L_0}{4}, \tfrac{3L_0}{4}]),\qquad \widetilde{\Sigma} = \frac{1}{N}\widetilde{\Gamma} ( \subseteq \Sigma).
\end{equation}
Let $\eta : \mathbb{R}^d \rightarrow [0,\infty)$ be a continuous, compactly supported function. We can employ the strong coupling result~\eqref{eq:limsup_sc} in Proposition~\ref{prop:strong_coupling} and Proposition A.1 of~\cite{nitzschner2017solidification} respectively to infer the lower bounds
\begin{align}\label{eq:LiminfBound} \varliminf_{N \rightarrow \infty}  \inf_{\kappa \in \cK_N} \frac{1}{N^d} \sum_{x \in \bbZ^d} \eta\left(\tfrac{x}{N} \right)P_x[H_C < \infty] &\geq \varliminf_{N \rightarrow \infty} \inf_{\kappa \in \cK_N}\frac{1}{N^d} \sum_{x \in \bbZ^d} \eta\left( \tfrac{x}{N} \right)W_x[H_{\widetilde{\Gamma}} < \infty]\\
  & = \varliminf_{N \rightarrow \infty} \inf_{\kappa \in \cK_N}\frac{1}{N^d} \sum_{x \in \bbZ^d} \eta\left( \tfrac{x}{N} \right)W_{\tfrac{x}{N}}[H_{\widetilde{\Sigma}} < \infty], \notag
  \\ \label{eq:liminf}
      \varliminf_{K\to\infty}\varliminf_{N \rightarrow \infty} \inf_{\kappa \in \cK_N} \, \frac{d}{N^{d-2}} \capa_{\bbZ^d}(C) & \geq\varliminf_{K\to\infty}\varliminf_{N \rightarrow \infty} \inf_{\kappa \in \cK_N} \capa(\Sigma).
\end{align}
Next, we want to apply Lemma~\ref{prop:harmonic_bound} and the capacity lower bound~\eqref{eq:SolidificationEstimate} to the interfaces $\widetilde{\Sigma}$ and $\Sigma$ respectively. To this end, we consider a compact set $A' \subseteq \mathring{A}$ and some $\ell^\ast \geq 0$ (depending on $A,A'$), such that for large $N$ and all $\kappa \in \cK_N$, $d(A',U_1)\geq 2^{-\ell*}$ (recall the definition of $U_1$ from \eqref{eq:UnionBound}). In particular $W_x[H_{\widetilde{\Sigma}} < \tau_{10\widehat{L}_0/N}] \geq c(K)$ for all $x \in \partial U_0$ (we refer to (4.48)--(4.54) in~\cite{nitzschner2017solidification} for details of this calculation). Thus, $\widetilde{\Sigma}$ and even more so $\Sigma$ are `porous interfaces' for $A'$ for large $N$, meaning that ${\widetilde{\Sigma}},\Sigma \in \cS_{U_0, 10\widehat{L}_0/N, c(K)}$ for $U_0$ chosen as above. This allows us to apply Lemma~\ref{prop:harmonic_bound} in~\eqref{eq:LiminfBound} and the capacity lower bound~\eqref{eq:SolidificationEstimate} in~\eqref{eq:liminf} to obtain
\begin{align}\label{eq:pointwisesolid1} \varliminf_{N \rightarrow \infty}  \inf_{\kappa \in \cK_N} \frac{1}{N^d} \sum_{x \in \bbZ^d} \eta(\tfrac{x}{N})P_x[H_C < \infty] &\geq \varliminf_{N\to\infty} \frac{1}{N^d} \sum_{x\in\bbZ^d} \eta(\tfrac{x}{N}) h_{A'}(\tfrac{x}{N}) = \langle h_{A'},\eta\rangle, \\ \label{eq:capacitysol1}
  \varliminf_{K\to\infty}\varliminf_{N \rightarrow \infty} \inf_{\kappa \in \cK_N} \, \frac{d}{N^{d-2}} \capa_{\bbZ^d}(C) & \geq \capa(A').
\end{align}
In the last step of~\eqref{eq:pointwisesolid}, we used that the set of points where $h_{A'}$ is not continuous has zero Lebesgue measure. From~\eqref{eq:pointwisesolid1} and~\eqref{eq:capacitysol1}, it follows by taking $A' \uparrow \mathring{A}$ that
\begin{align}\label{eq:pointwisesolid} \varliminf_{N \rightarrow \infty}  \inf_{\kappa \in \cK_N} \frac{1}{N^d} \sum_{x \in \bbZ^d} \eta(\tfrac{x}{N})P_x[H_C < \infty] &\geq \varliminf_{N\to\infty} \frac{1}{N^d} \sum_{x\in\bbZ^d} \eta(\tfrac{x}{N}) h_{\mathring{A}}(\tfrac{x}{N}) = \langle h_{\mathring{A}},\eta\rangle, \\ \label{eq:capacitysol}
  \varliminf_{K\to\infty}\varliminf_{N \rightarrow \infty} \inf_{\kappa \in \cK_N} \, \frac{d}{N^{d-2}} \capa_{\bbZ^d}(C) & \geq \capa(\mathring{A}),
\end{align}
using~\eqref{eq:hAconvergence} and Proposition 1.13, p.60 of~\cite{port2012brownian}, respectively.

We proceed by showing that, uniformly on $\kappa\in \cK_N$, the volume of $C$ is much smaller than $N^d$. To do so we first observe that
\begin{equation}\label{eq:estimate_numpoints}
|\cC| \leq \sum_{x\in \widetilde{\cS}} |\widetilde{\cC}_x| \leq c  \bigg(\frac{N}{\widehat{L}_0}\bigg) \bigg(\frac{\widehat{L}_0}{L_0}\bigg)^{d-1} \stackrel{\eqref{eq:ScalesDefinition}}{\leq} c''\sqrt{\gamma_N}\frac{N^{d-2}}{\log N},
\end{equation}
which implies, recalling~\eqref{eq:gamma_N} and~\eqref{eq:ScalesDefinition}, 
\begin{equation}
\label{eq:smallVolumeOfC}
 \sup_{\kappa\in \cK_N} \frac{|C|}{N^d} \leq c'' \frac{\sqrt{\gamma_N} L_0^d}{N^2 \log N} \leq c''' \bigg(\gamma_N^{-\frac{d+1}{2}} N^{2-d}\log N \bigg)^{\frac{1}{d-1}} \to 0\quad \text{as $N\to\infty$}.
\end{equation}
As a further consequence of~\eqref{eq:estimate_numpoints}
together with~\eqref{eq:capacitysol} and~\eqref{eq:gamma_N} one has for $K$ large enough
\begin{equation}\label{eq:smallerror}
  \lim_{N\to\infty}\sup_{\kappa\in \cK_N}\frac{|\cC|}{\capa_{\bbZ^d}(C)} =0.
\end{equation}

\vspace{\baselineskip}
\noindent\textbf{Step 4: Borell--TIS inequality.}
In this step we prove the pivotal Proposition~\ref{prop:3.3}.
Roughly speaking, we will derive an asymptotic upper bound on the probability that in an interface $C$, consisting of $L_0$-boxes, all the boxes are $h$-bad at level $s$ and simultaneously the macroscopic averages of the Gaussian free field are above macroscopic averages of the function $-sh_{\mathring{A}}$. More precisely, we will show:
\begin{prop}\label{prop:3.3} Let $\Delta, s, \beta>0$, $\eta : \bbR^d \rightarrow [0,\infty)$ a continuous compactly supported function. Then there exists a function $\widetilde{\alpha}(\cdot)$ (possibly dependent on $\beta$ and $\eta$) with $\lim_K \widetilde{\alpha}(K) = 1$ such that for large enough $K$, one has
  \begin{equation}
  \begin{split}
  \label{eq:Prop3_3_claim}
  \limsup_{N \rightarrow \infty} & \sup_{\kappa \in \cK_N} \frac{1}{N^{d-2}}\log  \bbP\bigg[ \cA^{s,\Delta}_N \cap \bigcap_{z \in \cC} \{\inf_{D_z} h^z \leq -s \} \bigg] \\
  \leq 
  & -  \frac{1}{2}( s + \beta \Delta)^2 \varliminf_{N\to\infty} \inf_{\kappa\in \cK_N} \frac{\capa_{\bbZ^d}(C) N^{2-d}}{\widetilde{\alpha}(K) + \beta^2 d\, \capa_{\bbZ^d}(C)N^{2-d}  E(\eta) }
  \end{split}
  \end{equation}
  (recall the definition of the event $\cA^{s,\Delta}_N$ from~\eqref{eq:DefAsDelta_N}).
  \end{prop}
  \begin{proof}

  Similarly to the proofs of Lemma 4.2 of \cite{sznitman2015disconnection} and Lemma 4.2 of \cite{nitzschner2018entropic}, we use a Gaussian field as a tool to bound the probability of the event under the probability in \eqref{eq:Prop3_3_claim}. We attach to $\kappa \in \cK_N$  a collection of functions 
  \begin{equation}\label{eq:FunctionClassDef}
  \cF = \{ f \in (\bbZ^d)^{\cC} ; f(z) \in D_z \text{ for each }z \in \cC \},
  \end{equation}
  and define for $\beta > 0$, $f \in \cF$ the random variables 
  \begin{equation}\label{eq:DefZ_hat}
  Z_f  = \sum_{z \in \cC} \lambda(z) h^z(f(z)), \qquad
  \widehat{Z}_f  = Z_f(1 + \beta\langle h_{\mathring{A}},\eta \rangle) - \beta \langle  \bbX_N,\eta \rangle,
  \end{equation}
  where we set $\lambda(z) = e_C(B_z)/\capa_{\bbZ^d}(C)$ (recall the definition of $e_C$ from \eqref{eq:EqMeasure}). The crucial observation for our purposes is that
  \begin{equation}\label{eq:crucial}
    \cA^{s,\Delta}_N \cap \bigcap_{z \in \cC} \{\inf_{D_z} h^z \leq -s \} \subseteq \Big\{\inf_{f\in \cF} \widehat{Z}_f \leq -s -\beta \Delta \Big\}.
  \end{equation}
  Note that $\widehat{Z}_f$ is a zero-average Gaussian field, thus, to get an exponential upper bound on the probability of the event on the right-hand side of~\eqref{eq:crucial} it suffices to use the Borell--TIS inequality  (see Theorem 2.1.1, p. 50 of~\cite{adler2009random}), which yields
  \begin{equation}
  \label{eq:IntermediateStepBTIS}
      \bbP\bigg[\cA^{s,\Delta}_N \cap \bigcap_{z \in \cC} \{\inf_{D_z} h^z \leq -s \} \bigg]  \leq \exp \Big\{ - \frac{1}{2\sigma^2} \Big(s + \beta\Delta - \Big|\bbE\Big[\inf_{f \in \cF}\widehat{Z}_f\Big]\Big| \Big)^2_+ \Big\},
  \end{equation}
  where $\sigma^2 = \sup_{f\in \cF} \var[\widehat{Z}_f]$. 
  Note that $\bbE[\inf_{f \in \cF} \widehat{Z}_f] = \bbE[\inf_{f \in \cF} Z_f]$. From (4.17) of Theorem 4.2 in~\cite{sznitman2015disconnection} (which uses Dudley's Theorem on the supremum of a Gaussian process), we obtain that 
\begin{equation}
\label{eq:boundinfZ_f}
\sup_{\kappa \in \cK_N} \Big|\bbE\Big[\inf_{f \in \cF}\widehat{Z}_f\Big]\Big| \bigg( \frac{|\cC|}{\capa_{\bbZ^d}(C)} \bigg)^{-\frac{1}{2}} \leq \frac{c_3}{K}.
\end{equation}
   Upon combining this bound with the capacity lower bound~\eqref{eq:capacitysol} and the estimate on the cardinality of $\cC$~\eqref{eq:estimate_numpoints} we infer, for $K$ large enough
\begin{equation}
\label{eq:BoundCoverCap}
  \limsup_{N \rightarrow \infty} \sup_{\kappa \in \cK_N}\Big|\bbE\Big[\inf_{f \in \cF}\widehat{Z}_f\Big]\Big| \stackrel{\eqref{eq:boundinfZ_f}}{\leq}  \limsup_{N \rightarrow \infty} \sup_{\kappa \in \cK_N} \frac{c_3}{K}\bigg(\frac{|\cC|}{\capa_{\bbZ^d}(C)}\bigg)^{1/2} = 0,
\end{equation} 
see also (3.33)--(3.34) of~\cite{nitzschner2018entropic} for a similar argument. After taking logarithms in \eqref{eq:IntermediateStepBTIS}, dividing by $N^{d-2}$ and taking limits we thus obtain, for $K$ large enough,
  \begin{equation}\label{eq:firststep}
      \varlimsup_{N \rightarrow \infty} \sup_{\kappa \in \cK_N} \frac{1}{N^{d-2}}\log  \bbP\bigg[ \cA^{s,\Delta}_N \cap \bigcap_{z \in \cC} \{\inf_{D_z} h^z \leq -s \} \bigg] 
      \leq  -   \varliminf_{N\to\infty} \inf_{\kappa\in \cK_N} \frac{1}{2}\frac{( s + \beta \Delta)^2}{\sup_{f\in \cF} \var[\widehat{Z}_f] N^{d-2}}.
  \end{equation}
  We are left with providing an upper bound  for the variance of $\widehat{Z}_f$. To this end, we write
  \begin{equation}\label{eq:expansion_variance}
  \var(\widehat{Z}_f) = (1 + \beta \langle h_{\mathring{A}}, \eta\rangle)^2 \var(Z_f) - 2\beta(1 + \beta\langle  h_{\mathring{A}},\eta \rangle) \cG_N + \beta^2 \cH_N,
  \end{equation}
  where we defined 
  \begin{equation}
  \label{eq:GNHNDefinition}
  \cG_N = \bbE[Z_f \langle \bbX_N,\eta \rangle], \qquad \cH_N = \bbE[\langle \bbX_N,\eta \rangle^2].
  \end{equation} 
  From the proof of Theorem 4.2 of \cite{sznitman2015disconnection}, we have
\begin{equation}
\label{eq:UpperBoundZf}
\var(Z_f) \leq \frac{1}{\capa_{\bbZ^d}(C)} \left(\frac{c}{K^{d-2}} + \gamma(K,L_0) \right),
\end{equation}
where 
\begin{equation}\label{eq:gammaKL}
\gamma(K,L_0) = \sup_{\substack{z,z' \in \bbZ^d \\ |z-z'|_\infty \geq K L_0}} \sup_{\substack{y \in D_z, y' \in D_{z'} \\ x \in B_z, x' \in B_{z'} }}  \frac{g(y,y')}{g(x,x')} (\geq 1),
\end{equation}
and one has $\lim_K \limsup_N \gamma(K,L_0) = 1$. We also define for later use 
\begin{align}
\label{eq:DefGammaTilde}
\widetilde{\gamma}(K,L_0) & = \inf_{\substack{z,z' \in \bbZ^d \\ |z-z'|_\infty \geq K L_0}} \inf_{\substack{y \in D_z, y' \in D_{z'} \\ x \in B_z, x' \in B_{z'} }}  \frac{g(y,y')}{g(x,x')} (\leq 1), \\
\label{eq:Alpha1}
\alpha_1(K) &  = \frac{c}{K^{d-2}} +  \limsup_{N \rightarrow \infty}  \gamma(K,L_0), \\
\label{eq:Alpha2}
\alpha_2(K) & = \liminf_{N \rightarrow \infty} \widetilde{\gamma}(K,L_0),
\end{align}
and note that $\lim_K \alpha_1(K) = \lim_K \alpha_2(K) = 1$. 
We continue by looking for an asymptotic lower bound for $\cG_N$. We claim that
\begin{equation}\label{eq:claimGN}
  \liminf_{N\to \infty} \inf_{\kappa\in \cK_N} \inf_{f\in \cF} (\capa_{\bbZ^d}(C) \cG_N)\geq \alpha_2(K) \langle h_{\mathring{A}},\eta \rangle.
\end{equation}
In order to prove~\eqref{eq:claimGN}, first we introduce $\cU = \bigcup_{z\in \cC} U_z$. For every $x\notin \cU$ we have $d(x,C)\geq K L_0$ and $\bbE[h^z(f(z))\varphi_x] = g(f(z),x)$. Thus,
\begin{equation}
\label{eq:EstimationGN}
\begin{split}
  \cG_N & = \frac{1}{N^d} \sum_{x\in \bbZ^d} \eta\big(\tfrac{x}{N}\big) \sum_{z\in \cC} \bbE[h^z(f(z))\varphi_x] \lambda(z)\\
  & \geq \tfrac{1}{N^d \capa_{\bbZ^d}(C)} \sum_{x\in \cU^c} \eta\big(\tfrac{x}{N}\big) \sum_{z\in \cC} g(f(z),x) e_C(B_z)\\
  & \geq \tfrac{\widetilde{\gamma}(K,L_0)}{N^d \capa_{\bbZ^d}(C)} \sum_{x\in \cU^c} \eta\big(\tfrac{x}{N}\big) \sum_{x'\in C} g(x,x') e_C(x')\\
  & \stackrel{\eqref{eq:EquilibriumPotential}}{\geq} \tfrac{\widetilde{\gamma}(K,L_0)}{\capa_{\bbZ^d}(C)}\bigg(\tfrac{1}{N^d} \sum_{x \in \bbZ^d}  \eta\big(\tfrac{x}{N}\big) P_x[H_C<\infty]
  - \tfrac{1}{N^d}\sum_{x\in \cU} \eta\big(\tfrac{x}{N}\big)\bigg).
\end{split}
\end{equation}
Given, $\kappa \in \cK_N$, one has the following upper bound on the number of elements of $\cU$:
\begin{equation}\label{eq:BoundU}
|\cU|\leq  (10K)^d |C|,
\end{equation}
from which we infer using the `small volume' of the interface $C$~\eqref{eq:smallVolumeOfC} that
\begin{equation}
\label{eq:UTermVanishes}
\frac{1}{N^d} \sup_{\kappa \in \cK_N} \sum_{x \in \cU}\eta\left( \frac{x}{N} \right) \leq \| \eta \|_\infty   \sup_{\kappa \in \cK_N} \frac{|\cU|}{N^d} \rightarrow 0, \qquad \text{ as }N \rightarrow \infty. 
\end{equation}
We can therefore discard the last part of~\eqref{eq:EstimationGN} in our estimation. Thus, combining \eqref{eq:Alpha2},~\eqref{eq:EstimationGN} and~\eqref{eq:UTermVanishes} with the `pointwise solidification' lower bound \eqref{eq:pointwisesolid} yields~\eqref{eq:claimGN}.

We proceed by giving an asymptotic bound for $\cH_N = \bbE[\langle \bbX_N, \eta \rangle^2]$. 
By Lemma 2.2 of \cite{bolthausen1993critical}, we have
\begin{equation}\label{eq:HNbound}
  \lim_{N\to \infty} N^{d-2} \cH_N = \lim_{N\to \infty}  \frac{1}{N^{d+2}} \sum_{x,y \in \bbZ^d} \eta\left( \frac{x}{N}\right) g(x,y)\eta \left( \frac{y}{N}\right) = d\, E(\eta).
\end{equation}
We can now put~\eqref{eq:UpperBoundZf},~\eqref{eq:claimGN} and~\eqref{eq:HNbound} together into~\eqref{eq:expansion_variance} to conclude that
\begin{equation}\label{eq:finally}
  \limsup_{N\to\infty} \sup_{\kappa\in \cK_N} \sup_{f\in\cF} N^{d-2} \var[\widehat{Z}_f] \leq \limsup_{N\to\infty} \sup_{\kappa\in \cK_N} \frac{N^{d-2}}{\capa_{\bbZ^d}(C)} \widetilde{\alpha}(K)  + \beta^2 d E(\eta),
\end{equation}
where we defined
\[
\widetilde{\alpha}(K) = \alpha_1(K) + \beta^2 \langle h_{\mathring{A}},\eta\rangle^2\Big[\alpha_1(K)-\alpha_2(K)\Big] + 2 \beta \langle h_{\mathring{A}},\eta\rangle\Big[\alpha_1(K)-\alpha_2(K)\Big].
\]
Note that $\widetilde{\alpha}(K) \to 1$ as $K\to\infty$.
This concludes the proof of the proposition by plugging~\eqref{eq:finally} into~\eqref{eq:firststep}. 
  \end{proof}

We are now ready to finish the proof of Theorem~\ref{thm:pushdown} with the help of Proposition~\ref{prop:3.3}. Recall that $\alpha + a = \delta< \overline{h}$ and choose $\overline{h}-\delta$ small enough so that $\widetilde{\Delta} = \Delta -(\overline{h}-\delta) \langle h_{\mathring{A}},\eta \rangle > 0$. For a fixed $\beta>0$, we define the auxiliary function 
\begin{equation}
\label{eq:DefinitionAuxiliaryFunction}
F(t,u;\beta) = \frac{u}{t + \beta^2 d E(\eta) u},\qquad t > 0,\, u > 0,
\end{equation}
which is increasing in $u > 0$ and jointly continuous in $t, u$. 

Combining the bound~\eqref{eq:CoarseGrainingBound} with Proposition~\ref{prop:3.3} (with the choice $s = a ( = \delta - \alpha)$), we obtain for all $K$ large enough
\begin{equation}\label{eq:finaltouch}
  \begin{aligned}
    \varlimsup_{N \rightarrow \infty}\frac{1}{N^{d-2}} & \log \bbP[\cA^{\overline{h}-\alpha,\Delta}_N \cap \cD^\alpha_N] = \varlimsup_{N \rightarrow \infty}\frac{1}{N^{d-2}} \log \bbP[\cA^{a,\widetilde{\Delta}}_N \cap \cD^\alpha_N] \\ & \leq - \varliminf_{N \rightarrow \infty} \inf_{\kappa \in \cK_N} \frac{1}{2}( a + \beta \widetilde{\Delta})^2 F\Big(\widetilde{\alpha}(K), \tfrac{1}{N^{d-2}} \capa_{\bbZ^d}(C) ;\beta\Big).
  \end{aligned}
\end{equation}

We obtain upon applying $\liminf_K$ on both sides of~\eqref{eq:finaltouch}, using~\eqref{eq:capacitysol} and $\widetilde{\alpha}(K) \rightarrow 1$: 
\begin{equation}
\limsup_{N \rightarrow \infty} \frac{1}{N^{d-2}} \log \ \bbP[\cA^{\overline{h}-\alpha,\Delta}_N \cap \cD^\alpha_N] \leq - \frac{1}{2}(a + \beta \widetilde{\Delta})^2 F\Big(1,\tfrac{1}{d}\capa(\mathring{A}); \beta\Big),
\end{equation}
in view of the monotonicity in $u$ of $F$. Taking $\delta \rightarrow \overline{h}$ (which implies that $\widetilde{\Delta} \rightarrow \Delta$), we see that
\begin{equation}
\limsup_{N \rightarrow \infty} \frac{1}{N^{d-2}} \log \ \bbP[\cA^{\overline{h}-\alpha,\Delta}_N \cap \cD^\alpha_N] \leq - \frac{1}{2}(\overline{h} - \alpha + \beta \Delta)^2 F\Big(1, \tfrac{1}{d}\capa(\mathring{A});\beta\Big),
\end{equation}
and~\eqref{eq:thm31} now follows by reinserting the definition of the auxiliary function $F$ from \eqref{eq:DefinitionAuxiliaryFunction} and optimizing in $\beta > 0$. 
\end{proof}

\section{Pinning of the entropic push-down under disconnection}
In this section, we state and prove in Theorem \ref{thm:MainUpperBound} our main result, namely an asymptotic upper bound on the probability of the intersection between the disconnection event and the event that for a fixed compact set $J \subseteq \bbR^d$, the $d_J$-distance (which we introduce in \eqref{eq:DistanceMeasures}) between the measures $\bbX_N$ and $\cH^\alpha_{\mathring{A}}(x)\De x$ is larger than a given $\Delta>0$. In essence, $d_J(\bbX_N,\cH^\alpha_{\mathring{A}})$ encodes the maximal separation between $\langle \bbX_N,\eta\rangle$ and $\langle \cH^\alpha_{\mathring{A}},\eta\rangle$, when $\eta$ runs over a class of functions with support in $J$ and uniformly bounded sup-norm and Lipschitz constant, see \eqref{eq:DefLip_1}. Note that the large deviation upper bound comes with a non-explicit exponential rate compared to the `push-down' result of Theorem \ref{thm:pushdown} in the previous section. The extension of the result from Theorem \ref{thm:pushdown} is twofold: On one hand, we have to deal with the aforementioned uniformity, on the other hand the entropic `push-down' result must be accompanied with a corresponding bound that ensures that, with high probability, the local average $\langle \bbX_N, \eta \rangle$ of a \emph{fixed} continuous, compactly supported function $\eta : \bbR^d \rightarrow [0,\infty)$ can in fact not be below $\langle\cH^\alpha_{\mathring{A}}, \eta \rangle$ (whereas \eqref{eq:thm31} essentially only showed that, with high probability, it cannot be above $\langle\cH^\alpha_{\mathring{A}}, \eta \rangle$). 

Let us mention that Theorem~\ref{thm:MainUpperBound} will find an application in the proof of Proposition~\ref{prop:XclosetoTarget} (cf.\ \eqref{eq:LastBoundProp55}), which is an intermediate step towards the proof of Theorem~\ref{thm:ProfileDescription}, that provides a profile description of the field.
\vspace{\baselineskip} 

We will need some preparations before we can state the main Theorem. First, we introduce for any continuous, compactly supported function $\eta : \bbR^d \rightarrow  \bbR$ the notation
\begin{equation}
\| \eta \|_{BL} = \| \eta \|_\infty + \sup_{x,y \in \bbR^d; x \neq y} \frac{|\eta(x) - \eta(y)|}{|x-y|} \in [0,\infty],
\end{equation}
for the sum of the sup-norm and Lipschitz constant of $\eta$. For a non-empty set $J \subseteq \bbR^d$, we define the function space
\begin{equation}
\label{eq:DefLip_1}
\lip_1(J) = \Big\{\text{$\eta : \bbR^d \rightarrow \bbR$, with $\supp \eta \subseteq J$ and $\| \eta \|_{BL} \leq 1$}\Big\}.
\end{equation}
With this definition at hand, we introduce for two signed Radon measures $\mu$ and $\nu$ on $\bbR^d$ the metric 
\begin{equation}
\label{eq:DistanceMeasures}
d_J(\mu,\nu) = \sup \{|\langle \mu - \nu,\eta \rangle|; \eta \in \lip_1(J) \}.
\end{equation}
If both $\mu$ and $\nu$ are probability measures and if $J = \bbR^d$, $d_J(\mu,\nu)$ is recognized as the Kantorovich-Rubinstein distance of the probability measures $\mu$ and $\nu$, see e.g.~p.191 of \cite{bogachev2007measure}. If $\nu$ has a density $f$ with respect to the Lebesgue measure on $\bbR^d$, i.e. $\nu(\De x) = f(x)\De x$, we write for simplicity $d_J(\mu,f)$ for $d_J(\mu,\nu)$. 
 We now come to the main asymptotic upper bound.
\begin{theorem}
\label{thm:MainUpperBound}
Consider $\Delta > 0$, $\alpha < \overline{h}$ and a compact, non-empty set $J \subseteq \bbR^d$. Then one has the asymptotic upper bound
\begin{equation}\begin{split}
\label{eq:Theorem4_1_claim}
\limsup_{N\to \infty} \frac{1}{N^{d-2}} \log \ & \bbP\Big[ d_J(\bbX_N, \cH^\alpha_{\mathring{A}}) \geq \Delta ; \cD^{\alpha}_N\Big]  \leq -\frac{1}{2d} (\overline{h}-\alpha)^2 \capa(\mathring{A}) - c_1(\Delta, \alpha, J).
\end{split}
\end{equation}
\end{theorem}
Let us explain shortly, why $\{d_J(\bbX_N, \cH^\alpha_{\mathring{A}}) \geq \Delta \}$ is in fact measurable. First, we notice that since $J \subseteq \bbR^d$ is compact, the set $C^0(J)$ of continuous functions from $J$ to $\bbR$, equipped with $\| \cdot \|_\infty$ is a separable metric space. Thus, $\lip_1(J) \subseteq C^0(J)$ is separable itself with respect to $\| \cdot \|_\infty$. Let $\{\eta_n \}_{n \in \bbN}$ be a countable, dense subset of $\lip_1(J)$. Since for any $\varphi \in \bbR^{\bbZ^d}$, $N \geq 1$, both $\bbX_N$ and $\cH^\alpha_{\mathring{A}}(x) \De x$ are finite measures on compact sets, one has
\begin{equation}
\lbrace d_J(\bbX_N, \cH^\alpha_{\mathring{A}}) \geq \Delta \rbrace = \Big\lbrace \sup_{n \in \bbN} |\langle \bbX_N - \cH^\alpha_{\mathring{A}}, \eta_n\rangle| \geq \Delta \Big\rbrace,
\end{equation}
and therefore the set on the left-hand side of the equation is measurable. 

The following consequence is immediate from combining Theorem \ref{thm:MainUpperBound} with the lower bound \eqref{eq:DiscLowerBound} and follows in the same way as Corollary \ref{thm:Corollary32}, using~\eqref{eq:Equality_crit_par}, see~\cite{duminil2019equality}.
\begin{corollary}
\label{thm:Corollary42}
Consider $\Delta, \alpha, J$ as in Theorem \ref{thm:MainUpperBound} and assume that $\capa(A) = \capa(\mathring{A})$. Then, one has 
\begin{equation}
\limsup_{N\to\infty} \frac{1}{N^{d-2}} \log \bbP\Big[ d_J(\bbX_N, \cH^\alpha_{\mathring{A}}) \geq \Delta \big\vert \cD^\alpha_N \Big] \leq - c_1(\Delta, \alpha, J).
\end{equation}
\end{corollary}
It might be helpful at this point to give a short outline of the proof of Theorem \ref{thm:MainUpperBound} and explain where the coarse-graining procedure from the proof of Theorem \ref{thm:pushdown} comes into play. We will first introduce a mollifier $\chi_\epsilon$ and reduce the problem of controlling the supremum of $|\langle  \bbX_N - \cH^\alpha_{\mathring{A}},\eta \rangle |$ over the class $\lip_1(J)$ to a class with much smaller complexity, namely the `location family' $\{ \chi_\epsilon(\cdot - x) \}_{x \in J}$. This replacement involves the use of certain Gaussian estimates that were recalled in Section 2. As a next step, we use the same coarse-graining procedure as in Section 3 to decompose the disconnection event into a union over $\kappa \in \cK_N$, however we also distinguish cases where the Dirichlet energy of $h_{\mathring{A}} - h_{\widehat{\Sigma}}$ ($\widehat{\Sigma}$ is a slight enlargement of $\Sigma$ associated to $\kappa \in \cK_N$) is either larger or smaller than a given value $\mu > 0$ (corresponding to $\kappa \in \widetilde{\cK}^\mu_N$ and $\kappa \in \cK^\mu_N$, respectively), see \eqref{eq:PartitionK_N}.
We will show that, in light of Lemma~\ref{thm:SolidifDirichlet}, the probability of the former case happening decays faster than the probability of disconnection for every choice of $\mu>0$ (cf.\ \eqref{eq:BoundD_LargeDirichletform}--\eqref{eq:BoundSplit_1}).  In the latter case, we develop a bound on the probability of the coarse-grained event in Proposition \ref{thm:Prop43}, where we  introduce a centered Gaussian field $\bar{Z}_{f,x}$ similar to $\widehat{Z}_f$ from \eqref{eq:DefZ_hat}, where $f$ varies in the class $\cF$ of \eqref{eq:FunctionClassDef} and $\eta$ is replaced by the location family. Like in the proof of Theorem \ref{thm:pushdown}, we will apply the Borell--TIS inequality, and we thus need bounds on the variance of $\bar{Z}_{f,x}$ and the expectation of $\bar{Z} = \inf_{x \in J} \inf_{f \in \cF} \bar{Z}_{f,x}$, that are derived in Lemma \ref{thm:Lemma44} and Lemma \ref{thm:Lemma45}, respectively. 

The distinction into $\kappa \in \widetilde{\cK}^\mu_N$ and $\kappa \in \cK^\mu_N$, which was not required in the `push-down' result of Theorem \ref{thm:pushdown}, is needed for the following reason:
In the course of proving the entropic push-down, a pointwise lower bound of $h_{\widetilde{\Sigma}}$ in terms of $h_{A'}$, $A' \subseteq \mathring{A}$ compact, was used in the bound for the variance of $\widehat{Z}_f$, relying on the solidification estimate of Lemma \ref{prop:harmonic_bound}. A similar approach using $\bar{Z}_{f,x}$ to show that, with high probability, $\langle \bbX_N, \eta\rangle \geq \langle \cH^\alpha_{\mathring{A}}, \eta\rangle - \Delta$ for large $N$ involves a variance bound where $h_{\widehat{\Sigma}}$ must be bounded above in terms of  $h_{A'}$. Since such a bound is not obvious, we have to put in `by hand' the constraint that $h_{\widehat{\Sigma}}$ and $h_{A'}$ are close (which corresponds to $\kappa \in \cK^\mu_N$), see also Remark \ref{thm:Remark46}, 2).

\begin{proof}[Proof of Theorem \ref{thm:MainUpperBound}] We start by performing a reduction of the problem to a location family. Let $\chi : \bbR^d \rightarrow [0,\infty)$ be a symmetric, smooth probability density supported in $B_2(0,1)$, and set for $\epsilon >0$, $\chi_\epsilon(x) = \epsilon^{-d}\chi(x/\epsilon)$. We define for $\eta \in \lip_1(J)$ the mollification
\begin{equation}
\eta_\epsilon = \eta \ast \chi_\epsilon.
\end{equation}
Our preliminary goal is to replace $\eta$ in $\langle \bbX_N - \cH^\alpha_{\mathring{A}}, \eta \rangle$ by $\eta_\epsilon$ and to derive \eqref{eq:FirstBoundProof}. We first bound the deterministic contribution of the difference between $\langle \bbX_N - \cH^\alpha_{\mathring{A}}, \eta \rangle$ and $\langle \bbX_N - \cH^\alpha_{\mathring{A}}, \eta_\epsilon \rangle$ as follows: If $0 < \vartheta < \Delta$, one has for some $\epsilon \leq \epsilon_1(\vartheta,\alpha)$ the inequality
\begin{equation}
\label{eq:SmallnessConvolution}
\sup_{\eta \in \lip_1(J)} |\langle \cH^\alpha_{\mathring{A}},\eta - \eta_\epsilon \rangle | \leq \| \cH^\alpha_{\mathring{A}} \ast \chi_\epsilon - \cH^\alpha_{\mathring{A}}\|_{L^1(J)} < \tfrac{\vartheta}{2},
\end{equation}
where we used $\|\eta \|_\infty \leq 1$, $\supp \eta \subseteq J$ for all $\eta \in \lip_1(J)$ and the convergence of  $\cH^\alpha_{\mathring{A}}\ast \chi_\epsilon$ to $\cH^\alpha_{\mathring{A}}$ in $L^1(J)$ as $\epsilon \rightarrow 0$. 

 Consequently, we obtain for $\epsilon \leq \epsilon_1(\vartheta,\alpha)$ the inequality 
\begin{equation}
\label{eq:FirstBoundProof}
\begin{split}
\bbP\Big[d_J(\bbX_N,\cH^\alpha_{\mathring{A}}) \geq \Delta; \cD^\alpha_N \Big] &\leq \bbP\bigg[\sup_{\eta \in \lip_1(J)}  | \langle \bbX_N - \cH^\alpha_{\mathring{A}},\eta_\epsilon\rangle| \geq \Delta - \vartheta; \cD^\alpha_N \bigg] \\
& + \bbP\bigg[\sup_{\eta \in \lip_1(J)}  | \langle \bbX_N, \eta_\epsilon - \eta\rangle| \geq \tfrac{\vartheta}{2} \bigg].
\end{split}      
\end{equation}
We will now derive separate large deviation upper bounds on the two summands on the right-hand side of \eqref{eq:FirstBoundProof}. The second summand in which the disconnection event was neglected will be treated by applying the Gaussian  inequality \eqref{eq:GaussianEstimate}, leading to \eqref{eq:GaussianBoundResult}. For the first summand, one performs a coarse-graining of the disconnection event, decomposing the event $\widetilde{\cD}^\alpha_{\mathring{A}} = \cD^\alpha_N \setminus \cB_N$ (where $\cB_N$ is the `bad' event from \eqref{eq:BadEvent}) into a union over $\cD_{N,\kappa}, \kappa \in \cK_N$, which leads to \eqref{eq:CoarseGrainingSection4}.

We start with an upper bound on the second summand in \eqref{eq:FirstBoundProof}. For $z \in \bbR^d$, one has the bound
\begin{equation}
\sup_{\eta \in \lip_1(J)} |\eta_\epsilon(z) - \eta(z)| \leq \sup_{\eta \in \lip_1(J)} \| \eta \|_{BL} \int |y| \chi_\epsilon(y) \De y \leq c\epsilon. 
\end{equation} 
This enables us to bound the second term on the right-hand side in \eqref{eq:FirstBoundProof} as follows:
\begin{equation}
\bbP\bigg[\sup_{\eta \in \lip_1(J)}  | \langle \bbX_N,\eta_\epsilon - \eta\rangle| \geq \tfrac{\vartheta}{2} \bigg] \leq \bbP\bigg[ \frac{c\epsilon}{N^d} \sum_{y \in (J^\epsilon)_N} |\varphi_y| \geq \tfrac{\vartheta}{2}  \bigg],
\end{equation} 
where $(J^\epsilon)_N = (NJ^\epsilon) \cap \bbZ^d$ (and $J^\epsilon = \{x\in \bbR^d: d(x,J)\leq\epsilon \}$ denotes the closed $\epsilon$-neighborhood of $J$). Using the bounds \eqref{eq:BoundsCovMatrixGFF}, one obtains from \eqref{eq:GaussianEstimate} 
\begin{equation}
\label{eq:GaussianBoundResult}
\bbP\bigg[\sup_{\eta \in \lip_1(J)}  | \langle \bbX_N, \eta_\epsilon - \eta \rangle| \geq \tfrac{\vartheta}{2} \bigg] \leq \exp\bigg\{-\frac{1}{8}\min \Big( N^{d-2} \frac{c_4\vartheta^2}{\epsilon^2}, N^{d-1} \frac{c_5\vartheta^4}{\epsilon^4} \Big) \bigg\}.
\end{equation}
For later purposes, we choose $\epsilon \leq \epsilon_2(\vartheta,\alpha) (\leq \epsilon_1(\vartheta,\alpha))$ such that \eqref{eq:SmallnessConvolution} holds and moreover
\begin{equation}
\label{eq:RelationForCapA}
\frac{c_4\vartheta^2}{8\epsilon^2} > \frac{1}{2d}(\overline{h} - \alpha)^2 \capa(\mathring{A}).
\end{equation} 
This choice will be crucial when comparing the two different exponential rates of the two summands in \eqref{eq:FirstBoundProof} towards the end of the proof (cf.\ \eqref{eq:finalstep}).

Let us now deal with the first summand on the right-hand side of \eqref{eq:FirstBoundProof}. In view of the inequality
\begin{equation}
|\langle  \bbX_N - \cH^\alpha_{\mathring{A}},\eta_\epsilon \rangle| = \left\vert \int_J \eta(y) \langle \bbX_N - \cH^\alpha_{\mathring{A}}, \chi_\epsilon(\cdot - y) \rangle \De y \right\vert \stackrel{\eqref{eq:DefLip_1}}{\leq} |J|\sup_{y \in J} |\langle \bbX_N - \cH^\alpha_{\mathring{A}},\chi_\epsilon(\cdot-y) \rangle |,
\end{equation}
we are reduced to finding a large deviation upper bound on the probability of the event $\widetilde{\cA}^{\overline{h}-\alpha, ( \Delta-\vartheta)/|J|,\chi_\epsilon}_N \cap \cD^\alpha_N$, where we defined
\begin{equation}
\widetilde{\cA}^{s,\Delta',\zeta }_N = \left\lbrace \sup_{y \in J} |\langle \bbX_N + sh^\alpha_{\mathring{A}}, \zeta(\cdot-y) \rangle | \geq \Delta'\right\rbrace, 
\end{equation}
for any $\zeta :\bbR^d \rightarrow \bbR $ smooth, compactly supported function and $s, \Delta' > 0$ (which should be compared with \eqref{eq:DefAsDelta_N}). We will use the notation from \eqref{eq:ScalesDefinition}--\eqref{eq:BoxesDefinition}. By performing the same coarse-graining of the event $\widetilde{\cD}^\alpha_N = \cD^\alpha_N \setminus \cB_N$ leading up to \eqref{eq:CoarseGrainingBound}, we find that the first summand in the right hand side of~\eqref{eq:FirstBoundProof} can be estimated as
\begin{equation}
\label{eq:CoarseGrainingSection4}
  \begin{split}
    \limsup_{N \rightarrow \infty} \frac{1}{N^{d-2}} &\log\bbP\bigg[\sup_{\eta \in \lip_1(J)}  | \langle \bbX_N - \cH^\alpha_{\mathring{A}},\eta_\epsilon\rangle| \geq \Delta - \vartheta; \cD^\alpha_N \bigg] \\
&\leq\limsup_{N \rightarrow \infty} \frac{1}{N^{d-2}} \log \bbP\big[\widetilde{\cA}^{\overline{h}-\alpha,( \Delta-\vartheta)/|J|,\chi_\epsilon}_N \cap \cD^\alpha_N\big] 
\\&\leq \limsup_{N \rightarrow \infty} \sup_{\kappa \in \cK_N} \frac{1}{N^{d-2}} \log \bbP\big[\widetilde{\cA}^{\overline{h}-\alpha,( \Delta-\vartheta)/|J|,\chi_\epsilon}_N \cap \cD_{N,\kappa}\big],
  \end{split}
\end{equation}
relying on the fact that $\cB_N$ is negligible, cf.~\eqref{eq:SuperExponentialBound} and that $|\cK_N| = \exp(o(N^{d-2}))$, see~\eqref{eq:smallCombComplexity}. The term on the right-hand side of \eqref{eq:CoarseGrainingSection4} will be bounded in two different ways, depending on the nature of the `porous interface' generated by $\kappa \in \cK_N$ (recall \eqref{eq:UnionBound}): If the Dirichlet energy of the difference  of $h_{\widehat{\Sigma}}$ (with $\widehat{\Sigma}$ an enlarged version of $\Sigma$, see \eqref{eq:GammaHatDef} below) and $h_{\mathring{A}}$ is smaller than some $\mu > 0$, we will use Proposition \ref{thm:Prop43} below to bound the probability under consideration. For this case, we will use the notation $\kappa \in \cK^\mu_N$. In the situation where the Dirichlet energy of the difference of $h_{\widehat{\Sigma}}$ and $h_{\mathring{A}}$ is larger than $\mu$, namely when $\kappa \in \widetilde{\cK}^\mu_N := \cK_N \setminus \cK^\mu_N$, we will rely on the solidification result for Dirichlet forms (cf.\ \eqref{eq:BoundDirichletForm} in Lemma~\ref{thm:SolidifDirichlet}), to show that the porous interfaces attached to $\kappa\in \widetilde{\cK}^\mu_N$  are unlikely, see~\eqref{eq:BoundSplit_1}. 

 We now introduce some more notation to make this dichotomy for $\kappa \in \cK_N$ explicit. To $\kappa \in \cK_N$, we associate sets of boxes in $\bbR^d$
 \begin{equation}
 \label{eq:GammaHatDef}
 \widehat{\Gamma} = \bigcup_{z \in \cC} (z + [-\tfrac{L_0}{4},\tfrac{5L_0}{4}]^d), \qquad \widehat{\Sigma} = \frac{1}{N}\widehat{\Gamma},
 \end{equation}
 $\widehat{\Gamma}$ being a slightly enlarged version of the $\bbR^d$-filling of $C$ (see \eqref{eq:UnionBound}), and
for $\mu > 0$, we introduce the following partition of $\cK_N$:
\begin{equation}
\label{eq:PartitionK_N}
\begin{split}
\cK_N^\mu & = \lbrace \kappa \in \cK_N;  \cE(h_{\mathring{A}} - h_{\widehat{\Sigma}},h_{\mathring{A}} - h_{\widehat{\Sigma}}) \leq \mu 
\rbrace, \\
\widetilde{\cK}_N^\mu & = \lbrace \kappa \in \cK_N;  \cE(h_{\mathring{A}} - h_{\widehat{\Sigma}},h_{\mathring{A}} - h_{\widehat{\Sigma}}) > \mu 
\rbrace \  ( = \cK_N \setminus \cK^\mu_N).
\end{split}
\end{equation}

The next proposition, which can be seen as a generalization of Proposition \ref{prop:3.3}, will be needed to bound the probability on the right-hand side of \eqref{eq:CoarseGrainingSection4} uniformly over $\kappa \in \cK^\mu_N$. The primary idea is to use a Gaussian field $(\bar{Z}_{f,x})_{f \in \cF, x\in J}$ to encode the event under the probability on the right-hand side of \eqref{eq:CoarseGrainingSection4}. This allows us to use the Borell--TIS inequality to obtain an upper bound, granted that we have bounds on the variance (provided in Lemma~\ref{thm:Lemma44} below) and on the expectation (derived in Lemma~\ref{thm:Lemma45} below) of the infimum (over $f \in \cF$ and $x \in J$) of this field. It is worth stressing that the need of the splitting $\cK_N = \cK_N^\mu\cup \widetilde{\cK}^\mu_N$ enters the proof only in the variance bound (that is, in Lemma~\ref{thm:Lemma44}).
\begin{prop}
\label{thm:Prop43}
Let $\Delta, s > 0$, $\zeta : \bbR^d \rightarrow \bbR $ a smooth compactly supported function and $0 < \beta < \left( \sup_{x \in J} |\langle h_{\mathring{A}}, \zeta(\cdot - x) \rangle| \right)^{-1}$. Then, there exists a $K_0 \geq 100$ and a function $\widetilde{\alpha}'(\cdot)$ dependent on $\beta$ and $\zeta$ with $\lim_K \widetilde{\alpha}'(K) = 1$ such that for $K \geq K_0$ 
\begin{equation}
\begin{split}
\label{eq:Prop4_3_claim}&
\varlimsup_{N \rightarrow \infty} \sup_{\kappa \in \cK_N^\mu} \frac{1}{N^{d-2}}\log  \bbP\Big[ \widetilde{\cA}^{s,\Delta,\zeta}_N \cap \bigcap_{z \in \cC} \{\inf_{D_z} h^z \leq -s \} \Big] \\
& \leq - \varliminf_{N \rightarrow \infty} \inf_{\kappa \in \cK_N^\mu}\frac{( s + \beta \Delta  )^2}{2N^{d-2}} \frac{\capa_{\bbZ^d}(C)}{\widetilde{\alpha}'(K) + \beta^2 c_6(M)E(\zeta) + 2\sqrt{\mu}\beta\sqrt{E(|\zeta|)} \alpha_1(K)}.
\end{split}
\end{equation}
\end{prop}
\begin{proof}
We will essentially perform a modification of the proof of Proposition \ref{prop:3.3}, and use a Gaussian field $(\bar{Z}_{f,x})_{f \in \cF, x \in J}$ as a tool to capture the event under the probability on the right-hand side of \eqref{eq:Prop4_3_claim}. Recall the definitions of $\cF$, $Z_f$ and $\lambda(z)$ from \eqref{eq:FunctionClassDef}, \eqref{eq:DefZ_hat} and below. Similarly to $\widehat{Z}_f$, we introduce, for  $x \in J$ and $f \in \cF$, the zero-average Gaussians
\begin{equation}
\begin{split}
\bar{Z}_{f,x} & = Z_f(1 - \beta \langle  h_{\mathring{A}} ,\zeta(\cdot - x)\rangle) + \beta \langle  \bbX_N, \zeta(\cdot - x) \rangle. 
\end{split}
\end{equation}
We furthermore define 
\begin{equation}
\begin{split}
\bar{Z} & = \inf_{x \in J} \inf_{f \in \cF} \bar{Z}_{f,x}.
\end{split}
\end{equation}
By replacing $\zeta$ by $-\zeta$, the claim will follow once we show that for every $\zeta : \bbR^d \rightarrow \bbR$ smooth and compactly supported the upper bound \eqref{eq:Prop4_3_claim} holds with $\widetilde{\cA}^{s,\Delta,\zeta}_N$ replaced by the event
\begin{equation}
(\widetilde{\cA}')^{s,\Delta,\zeta }_N = \left\lbrace \inf_{y \in J} \langle sh_{\mathring{A}} + \bbX_N, \zeta(\cdot-y) \rangle  \leq -\Delta\right\rbrace. 
\end{equation}
The central observation is that similar as in \eqref{eq:crucial}, one has
\begin{equation}
(\widetilde{\cA}')^{s,\Delta,\zeta }_N \cap \cD_{N,\kappa} \subseteq \{ \bar{Z} \leq -s - \beta \Delta \},
\end{equation}
and thus, we are reduced to bound the probability of the event on the right-hand side, for which we will again employ the Borell--TIS inequality. To this end, bounds on $\bbE[\bar{Z}]$ as well as on $\sup_{x\in J}\sup_{f\in \cF} \var(\bar{Z}_{f,x})$ will be given in the next two lemmas. At this point, the Reader may first read the statement of the next two lemmas, skip their proofs, and directly proceed above \eqref{eq:ConclusionProof43} to see how the proof of Proposition \ref{thm:Prop43} (and thus, of Theorem \ref{thm:MainUpperBound}) is completed. 
\begin{lemma} 
\label{thm:Lemma44}
\begin{equation}
\label{eq:MainVarianceBounds}
\begin{split}
\limsup_{N \rightarrow \infty} \sup_{\kappa \in \cK_N^\mu} &\sup_{x\in J} \sup_{f \in \cF} ( \capa_{\bbZ^d}(C) \var(\bar{Z}_{f,x})) \\& \leq \widetilde{\alpha}'(K) + \beta^2 c_6(M)E(\zeta) + 2\beta\sqrt{\mu}\sqrt{E(|\zeta|)}\alpha_1(K),
\end{split}
\end{equation}
where $\widetilde{\alpha}'(K) \rightarrow 1$ as $K \rightarrow \infty$. 
\end{lemma}
\begin{proof}
The variance of $\bar{Z}_{f,x}$ is given by
\begin{equation}
\begin{split}
  \var(\bar{Z}_{f,x}) &= (1 - \beta \langle h_{\mathring{A}},\zeta(\cdot-x) \rangle)^2 \var(Z_f) \\ &+ 2\beta ( 1 - \beta \langle h_{\mathring{A}}, \zeta(\cdot-x) \rangle) \cG_{x,N} + \beta^2 \cH_{x,N},
\end{split}
\end{equation}
where $\cG_{x,N}$ and $\cH_{x,N}$ are defined as in \eqref{eq:GNHNDefinition} with $\eta$ replaced by $\zeta(\cdot -x)$. In view of \eqref{eq:UpperBoundZf} and \eqref{eq:HNbound}, we only need an \textit{upper} bound on $\capa_{\bbZ^d}(C)\cG_{x,N}$ (with an error term which is uniformly small in $\kappa \in \cK^\mu_N$ and $x \in J$). We decompose $\zeta = \zeta^+ - \zeta^-$ and define
\begin{equation}
\cG_{x,N} = \bbE[Z_f \langle \bbX_N, \zeta^+(\cdot-x) \rangle] - \bbE[Z_f \langle \bbX_N, \zeta^-(\cdot-x) \rangle] \stackrel{\text{def}}{=} \cG_{x,N}^+ - \cG_{x,N}^-.
\end{equation}
Let us first treat $\cG_{x,N}^+$. In a first step, we argue similarly as in \eqref{eq:EstimationGN}--\eqref{eq:UTermVanishes} in order to recover a sum involving the hitting probability $P_y[H_C < \infty]$, apart from a negligible error. We have for $\cU = \bigcup_{z\in \cC} U_z$
\begin{equation}
\label{eq:UpperBoundG}
\begin{split}
  &\cG_{x,N}^+  = \frac{1}{N^d} \sum_{y\in \bbZ^d} \zeta^+\left(\tfrac{y-x}{N}\right) \sum_{z\in \cC} \bbE[h^z(f(z))\varphi_x] \lambda(z)\\
  & \leq \tfrac{1}{\capa_{\bbZ^d}(C) N^d} \sum_{y\in \bbZ^d} \zeta^+\left(\tfrac{y-x}{N}\right) \sum_{z\in \cC} e_C(B_z) g(x,f(z))\\
  & \stackrel{\eqref{eq:HarmonicPot}}{\leq} \tfrac{\gamma(K,L_0)}{\capa_{\bbZ^d}(C) N^d}
   \sum_{y\in \bbZ^d} \zeta^+\left(\tfrac{y-x}{N}\right)
   P_y[H_C <\infty] + \tfrac{1}{\capa_{\bbZ^d}(C) N^d} \sum_{y\in \cU} \zeta^+\left(\tfrac{y-x}{N}\right) e_C(B_{z_y}) g(y,f(z_y)),
   \end{split}
\end{equation}
where $z_y$ is the only element in $\cC$ such that $y\in U_{z_y}$ and where $\gamma(K, L_0)$ was defined in~\eqref{eq:gammaKL}.
We first treat the product of the second summand of the above inequality and $\capa_{\bbZ^d}(C)$ and see that
\begin{equation}
\label{eq:BoundSecondTermG}
\begin{aligned}
 \sup_{\kappa \in \cK_N^\mu} & \sup_{x \in J}\sup_{f\in \cF} \frac{1}{N^d} \sum_{y\in \cU}  \zeta^+\left(\tfrac{y-x}{N}\right) e_C(B_{z_y}) g(y,f(z_y)) \\ &\leq  \sup_{\kappa \in \cK_N^\mu} c \max_{y,y'\in U_0} g(y,y') \|\zeta \|_\infty \left(\frac{L_0}{N}\right)^d \capa_{\bbZ^d}(C)
   \leq c \left(\frac{L_0}{N}\right)^d N^{d-2},
\end{aligned}
\end{equation}
which converges to zero as $N\to \infty$ in view of \eqref{eq:ScalesDefinition}. This shows that the second summand in the last line of \eqref{eq:BoundSecondTermG} yields no contribution in the limit $N \rightarrow \infty$. 

 We now consider the first summand in the last inequality of \eqref{eq:UpperBoundG}. An upper bound will be obtained in three steps, which we shortly explain. In a first (technical) step, we replace $P_y[H_C < \infty]$ by $W_y[H_{\widehat{\Gamma}} < \infty]$ with the help of a strong coupling result (Proposition \ref{prop:strong_coupling}). In a second step, we replace the sum $\frac{1}{N^d} \sum_{y} \zeta^{+}(\tfrac{y-x}{N}) h_{\widehat{\Sigma}}(\tfrac{y}{N})$ uniformly over $\kappa \in \cK_N^\mu$ by an integral, see \eqref{eq:Replace_sum_by_integral}, where we use the harmonicity of $h_{\widehat{\Sigma}}$ outside a small blow-up of $\widehat{\Sigma}$. Finally, we replace $h_{\widehat{\Sigma}}$ by $h_{\mathring{A}}$, using that $\kappa \in \cK_N^\mu$. Note that since we are essentially looking for an \textit{upper} bound of $h_{\widehat{\Sigma}}$ in terms of $h_{\mathring{A}}$ (as opposed to the lower bound in Section 3), we cannot use a pointwise solidification result, see also Remark 4.6, 2). 

Combining \eqref{eq:UpperBoundG}, \eqref{eq:BoundSecondTermG} with \eqref{eq:liminf_sc} and using that $W_y[H_{\widehat{\Gamma}}<\infty] = h_{\widehat{\Sigma}}(\tfrac{y}{N})$ (by scaling and the fact that all points in $\widehat{\Gamma}$ are regular), we find that
 \begin{equation}
 \label{eq:BoundGN_Step1}
 \begin{split}
\sup_{\kappa \in \cK_N^\mu} \sup_{f\in \cF} &(\capa_{\bbZ^d}(C) \cG_{x,N}^+)  \leq \alpha_1(K)   \sup_{\kappa \in \cK_N^\mu} \frac{1}{N^d} \sum_{y \in \bbZ^d} \zeta^+\left(\tfrac{y-x}{N}\right)  h_{\widehat{\Sigma}}\left(\tfrac{y}{N}\right) + \cR^{(1)}_{x,N},
 \end{split}
 \end{equation}
where
 \begin{equation}
 \begin{split}
 \varlimsup_{N \rightarrow \infty} &  \sup_{x \in J} \cR^{(1)}_{x,N} \leq  \alpha_1(K) \| \zeta \|_\infty |J| \varlimsup_{N \rightarrow \infty} \sup_{\kappa \in \cK_N^\mu} \sup_{y\in \bbZ^d} \Big(P_y[ H_{C}<\infty] - W_y[H_{\widehat{\Gamma}}<\infty]\Big) \stackrel{\eqref{eq:liminf_sc}}{\leq} 0,
 \end{split}
 \end{equation}
 with $\alpha_1$ defined in \eqref{eq:Alpha1} (actually, we could also take $\alpha_1(K) =\varlimsup_{N} \gamma(K,L_0)$ as a definition here, but this is not important). 
 As a next step, we aim at replacing the sum in the right member of \eqref{eq:BoundGN_Step1} by an integral, making use of the fact that $h_{\widehat{\Sigma}}$ is harmonic outside $\widehat{\Sigma}$. Let $\widehat{\Sigma}^\ast$ and $\widehat{\Sigma}_2^\ast$ be the enlargements of $\widehat{\Sigma}$ by $\tfrac{L_0}{N}$ and $\tfrac{L_0}{2N}$ respectively, that is
 \begin{equation}
\widehat{\Sigma}^\ast = \{ x \in \bbR^d; d_\infty(x,\widehat{\Sigma}) \leq \tfrac{L_0}{N} \},\quad 
\widehat{\Sigma}_2^\ast = \{ x \in \bbR^d; d_\infty(x,\widehat{\Sigma}) \leq \tfrac{L_0}{2N} \}.
 \end{equation}
By harmonicity of $h_{\widehat{\Sigma}}$, we can make use of an elementary gradient bound (see Theorem 2.10 of \cite{gilbarg2015elliptic}), stating that
\begin{equation}
\label{eq:GradientEstimate}
\sup_{z \in (\widehat{\Sigma}^\ast_2)^c} |\nabla h_{\widehat{\Sigma}}(z)| \leq c \frac{N}{L_0}.
\end{equation}
We proceed by approximating a sum with an integral. Writing
\begin{equation}
\label{eq:Replace_sum_by_integral}
\frac{1}{N^d} \sum_{y \in \bbZ^d} \zeta^+\left( \tfrac{y-x}{N} \right) h_{\widehat{\Sigma}}\left(\tfrac{y}{N} \right) \leq \int \zeta^+(z-x)h_{\widehat{\Sigma}}(z)dz + \cR^{(2)}_{x,N,\kappa},
\end{equation}
where 
\begin{equation}
\begin{split}
\cR^{(2)}_{x,N,\kappa} &= \frac{1}{N^d} \sum_{\frac{y}{N} \in \widehat{\Sigma}^\ast} \zeta^+\left( \tfrac{y-x}{N} \right) h_{\widehat{\Sigma}}\left(\tfrac{y}{N} \right)  + \sum_{\frac{y}{N} \notin \widehat{\Sigma}^\ast}  \int_{\big[\tfrac{y}{N}, \tfrac{y+1}{N}\big)^d}  h_{\widehat{\Sigma}}(z) \left|\zeta^+\left( \tfrac{y-x}{N} \right)  - \zeta^+(z-x) \right| \De z \\
& + \sum_{\frac{y}{N} \notin \widehat{\Sigma}^\ast}  \zeta^+\left( \tfrac{y-x}{N} \right) \int_{\big[\tfrac{y}{N}, \tfrac{y+1}{N}\big)^d}| h_{\widehat{\Sigma}}(z) - h_{\widehat{\Sigma}}\left( \tfrac{y}{N} \right) | \De z,
\end{split}
\end{equation}
we obtain the following bound for $\cR^{(2)}_{x,N} =\sup_{\kappa\in \cK_N^\mu}\cR^{(2)}_{x,N,\kappa}$
\begin{equation}
\label{eq:BoundRN}
 \varlimsup_{N \rightarrow \infty} \sup_{x \in J} \cR^{(2)}_{x,N} \leq \varlimsup_{N \rightarrow \infty} \sup_{\kappa \in \cK_N^\mu} \left( \frac{c|C|}{N^d} \| \zeta \|_{\infty} + \frac{c}{N}\|\nabla \zeta\|_\infty  + \frac{cN}{L_0} \| \zeta \|_\infty \frac{1}{N}\right) = 0,
\end{equation}
where we used  $\bigcup_{y/N\notin \widehat{\Sigma}^\ast} {\big[\tfrac{y}{N}, \tfrac{y+1}{N}\big)^d} \subseteq (\widehat{\Sigma}^\ast_2)^c$ and the gradient estimate \eqref{eq:GradientEstimate} to bound from above the third term together with the fact that $\sup_{\kappa \in \cK_N^\mu} |C| = o(N^d)$, which follows in the same way as the bound \eqref{eq:smallVolumeOfC}. This concludes the second step of the aforementioned procedure, and combining \eqref{eq:BoundGN_Step1}, \eqref{eq:Replace_sum_by_integral} and \eqref{eq:BoundRN} leads to
 \begin{equation}
 \label{eq:BoundGN_Step2}
 \begin{split}
 \sup_{\kappa \in \cK_N^\mu} \sup_{f\in \cF} &(\capa_{\bbZ^d}(C) \cG_{x,N}^+)\leq \alpha_1(K)  \sup_{\kappa \in \cK_N^\mu}  \int \zeta^+(y-x)h_{\widehat{\Sigma}}(y)\De y + \cR^{(1)}_{x,N}+ \cR^{(2)}_{x,N}.
 \end{split}
 \end{equation}
We now proceed to the final step of the variance bound. To this end, we bound 
\begin{equation}\label{eq:energyest}
\left\vert \int \zeta^+(y-x)(h_{\widehat{\Sigma}}(y) - h_{\mathring{A}}(y))\De y \right\vert \leq E(\zeta^+)^{\frac{1}{2}} \cE(h_{\mathring{A}} - h_{\widehat{\Sigma}}, h_{\mathring{A}} - h_{\widehat{\Sigma}})^{\frac{1}{2}},
\end{equation}
using \eqref{eq:CSEnergies} and the fact that $E(\zeta^+(\cdot -x)) = E(\zeta^+)$ for any $x \in \bbR^d$. Combining~\eqref{eq:energyest} and~\eqref{eq:BoundGN_Step2} and noting $E(\zeta^+)\leq E(|\zeta|)$, we obtain
 \begin{equation}
 \label{eq:BoundGN_Step3}
 \begin{split}
  \sup_{\kappa \in \cK_N^\mu} \sup_{f\in \cF} &(\capa_{\bbZ^d}(C) \cG_{x,N}^+)  \leq \alpha_1(K) \left( \langle h_{\mathring{A}}, \zeta^+(\cdot-x) \rangle + \mu^{\frac{1}{2}} E(|\zeta|)^{\frac{1}{2}}\right) + \cR_{x,N}, 
 \end{split}
 \end{equation}
 with $\cR_{x,N} = \cR_{x,N}^{(1)} + \cR_{x,N}^{(2)}$ and $\varlimsup_{N} \sup_{x \in J} \cR_{x,N} \leq 0$. Let us shortly discuss how to proceed for $\cG^-_{x,N}$. By \eqref{eq:claimGN} we see that 
\begin{equation}
\label{eq:BoundGNNegative}
\inf_{\kappa \in \cK^\mu}\inf_{f \in \cF} (\capa_{\bbZ^d}(C) \cG^-_{x,N}) \geq \alpha_2(K) \langle h_{\mathring{A}},\zeta^-(\cdot -x) \rangle + \widetilde{\cR}_{x,N},
\end{equation}
again with $\varliminf_{N} \inf_{x \in J} \widetilde{\cR}_{x,N} \geq 0$.
By combining \eqref{eq:BoundGN_Step3} and \eqref{eq:BoundGNNegative} with the bounds on $\var(Z_{f,x})$ and $\cH_{x,N}$ from \eqref{eq:UpperBoundZf} and \eqref{eq:HNbound}, we finally obtain
\begin{equation}
\begin{split}
\varlimsup_{N \rightarrow \infty} \sup_{\kappa \in \cK_N^\mu}& \sup_{x\in J} \sup_{f \in \cF} ( \capa_{\bbZ^d}(C) \var(\bar{Z}_{f,x})) \\
 & \leq \widetilde{\alpha}'(K) + \beta^2 c_6(M)dE(\zeta) + 2 \beta \alpha_1(K) \sqrt{\mu}\sqrt{E(|\zeta|)}, 
\end{split}
\end{equation}
where 
\begin{equation*}
  \widetilde{\alpha}'(K) = \alpha_1(K) + 2\beta \left(\sup_{x\in J} \langle h_{\mathring{A}},\zeta^+(\cdot-x) \rangle + \sup_{x\in J} \langle h_{\mathring{A}},\zeta^-(\cdot-x) \rangle \right) |\alpha_1(K) - \alpha_2(K)|
\end{equation*}
and the claim of the lemma follows. 
\end{proof}
In the following lemma, we will show that as $N \rightarrow \infty$, the expectation of $\bar{Z}$ vanishes uniformly in $\kappa \in \cK_N^\mu$. This behavior is the same as for $\bbE[\inf_{f\in \cF}\widehat{Z}_f]$ from the previous section, see \eqref{eq:BoundCoverCap}, however the argument is more involved in the present context, since we take another infimum over $x \in J$. 
\begin{lemma}
\label{thm:Lemma45}
For all large enough $K$, one has
\begin{equation}
\label{eq:BoundEZ}
\lim_{N \rightarrow \infty} \sup_{\kappa \in \cK_N^\mu} |\bbE[\bar{Z}]| = 0.
\end{equation}
\end{lemma}
\begin{proof}
Given $\kappa \in \cK_N^\mu$, we fix some $\overline{f}\in \cF$. Then, one has 
\begin{equation}
    \inf_{f\in \cF} \inf_{x\in J} \bar{Z}_{f,x}  \geq 2 \inf_{f\in \cF} \Big(Z_f - Z_{\overline{f}}\Big) + 2 Z_{\overline{f}} \wedge 0
    - \beta \sup_{x\in J} \langle \bbX_N, \zeta(\cdot-x) \rangle.  
\end{equation}
Taking expectations on both sides and rearranging $Z_{\overline{f}}$ leads to 
\begin{equation}\label{eq:exp_bound}
  \bbE\Big[\inf_{f\in \cF} \inf_{x\in J } Z_{f,x}\Big]  \geq 2 \bbE\Big[\inf_{f\in \cF} Z_f \Big] - 2 \bbE\Big[ Z_{\overline{f}} \vee 0 \Big ]
    -  \beta \bbE\Big[ \sup_{x\in J} \langle\bbX_N, \zeta(\cdot-x)\rangle\Big].  
\end{equation}
We can bound the first summand in~\eqref{eq:exp_bound} as in Theorem 4.2 of~\cite{sznitman2015disconnection},
\begin{equation}
 \sup_{\kappa \in \cK_N^\mu}\left\vert \bbE\Big[\inf_{f\in \cF} Z_f \Big] \right\vert \leq \frac{c}{K} \sup_{\kappa \in \cK_N^\mu} \sqrt{ \frac{|\cC|}{\capa_{\bbZ^d}(C)}} \rightarrow 0, \ N \rightarrow \infty,
\end{equation}
using \eqref{eq:BoundCoverCap}. Moreover we notice that the variance bound (4.25) of~\cite{sznitman2015disconnection} yields
\begin{equation}
   \sup_{\kappa \in \cK_N^\mu} \bbE\Big[ Z_{\overline{f}} \vee 0 \Big ] \leq  \sup_{\kappa \in \cK_N^\mu} \Big(\bbP[ Z_{\overline{f}} \geq 0] \var[Z_{\overline{f}} ]\Big)^{1/2} \leq \sup_{\kappa \in \cK_N^\mu} \sqrt{\frac{\alpha_1(K) }{2\capa_{\bbZ^d}(C)}} \rightarrow 0, N \rightarrow \infty,
\end{equation}
by \eqref{eq:BoundCoverCap} and the fact that $|\cC| \rightarrow \infty$ as $N \rightarrow \infty$.
What is left to do is to find an upper bound for the last summand in~\eqref{eq:exp_bound}. This will come as an application of the metric entropy method, outlined in \cite{adler2009random}. We consider the canonical metric induced on $\bbR^d$ by the $L^2(\bbP)$-distance of $\langle \bbX_N, \zeta(\cdot - x) \rangle$, $x \in \bbR^d$, namely $d_{ME} : \bbR^d\times \bbR^d \to [0,\infty)$ defined by
\begin{equation}
\label{eq:CanonicalMetric}
  \begin{split}
    d_{ME}^2(x,y) & =  \bbE\Big[  \langle\bbX_N, \zeta(\cdot-x) -  \zeta(\cdot-y)\rangle^2\Big] \leq \Big(  |x-y| \rho(N)\Big)^2, \ \text{where} \\
    \rho(N) & = \frac{c}{N^d}\bigg(  \sum_{x,y \in (\supp\zeta + J)_N} g(x,y) \bigg)^{\frac{1}{2}},
      \end{split}
\end{equation}
where we defined $(\supp \zeta + J)_N = N(\supp \zeta + J) \cap \bbZ^d$. Notice that $\rho(N) \to 0$ as $N\to\infty$. Moreover, it is easy to see that $J$ can be covered by at most $c_7 \big(\tfrac{\mathrm{diam}(J)}{\varepsilon} \big)^d$ Euclidean balls of radius $\varepsilon > 0$. In view of \eqref{eq:CanonicalMetric}, this implies 
\begin{equation}
  \begin{minipage}{0.7\linewidth}
    $J$ is covered by at most $N(\varepsilon) = c_7 \big(\tfrac{ \rho(N)\mathrm{diam}(J)}{\varepsilon}\big)^d$ balls
    in the canonical metric $d_{ME}$.    
  \end{minipage}
\end{equation} 
By Theorem 1.3.3, p.14 of~\cite{adler2009random}, we obtain that
\begin{equation}
\begin{split}
  \bbE\Big[ \sup_{x\in J} &\langle \bbX_N, \zeta(\cdot-x) \rangle\Big]  \leq c  \int_0^{\mathrm{diam}(J)\rho(N)} \log \bigg(\frac{\rho(N)}{ s }\bigg)^{\frac{1}{2} }\,\De s \\
  & = c\rho(N) \int_0^{\mathrm{diam}(J)} \log(t^{-1})^{\frac{1}{2} }\,\De t \to 0,\quad N\to \infty. \qedhere
  \end{split}
\end{equation}

\end{proof}
We can now conclude the proof of Proposition \ref{thm:Prop43}.
By the Borell--TIS inequality, (see 2.1.1, p.\ 50 of~\cite{adler2009random}) one has - using the bounds \eqref{eq:MainVarianceBounds} and \eqref{eq:BoundEZ} - that 
\begin{equation}
\label{eq:ConclusionProof43}
\begin{split}
&\varlimsup_N \sup_{\kappa \in \cK_N^\mu}  \frac{1}{N^{d-2}} \log \bbP\left[(\widetilde{\cA}')^{s,\Delta,\zeta}_N \cap \cD_{N,\kappa} \right] \leq \varlimsup_N  \sup_{\kappa \in \cK_N^\mu} \frac{1}{N^{d-2}} \log \bbP[\bar{Z}  \leq -s - \beta \Delta] \\
&\leq - \varliminf_N \inf_{\kappa \in \cK_N^\mu} \frac{\left(s + \beta\Delta  \right)^2}{2 N^{d-2}} \frac{\capa_{\bbZ^d}(C)}{\widetilde{\alpha}'(K) + \beta^2 c_6(M) dE(\zeta) + 2 \alpha_1(K) \beta \sqrt{\mu} \sqrt{E(|\zeta|)}},
\end{split}
\end{equation}
where we could omit the term containing $\bbE[\bar{Z}]$ in the last term due to \eqref{eq:BoundEZ}. \end{proof}

We will now carry out the last steps of the proof of Theorem \ref{thm:MainUpperBound}. Set $\alpha + a = \delta< \overline{h}$ and choose $\delta$ so that $\widetilde{\Delta} - \vartheta > 0$, where $\widetilde{\Delta} = \Delta - \sup_{x \in J}(\overline{h}-\delta) \langle \chi_\epsilon(\cdot - x),h_{\mathring{A}} \rangle$. Since $\cK_N = \cK^\mu_N \cup \widetilde{\cK}^\mu_N$, it holds that
\begin{equation}
\label{eq:SplittedVersionBound}
\begin{split}
&\varlimsup_N \frac{1}{N^{d-2}} \log \bbP\left[\widetilde{\cA}^{\overline{h}-\alpha,(\Delta-\vartheta)/|J|,\chi_\epsilon}_N \cap \cD^\alpha_N \right] \leq \\
&\bigg( \varlimsup_{N} \sup_{\kappa \in \cK_N^\mu} \frac{1}{N^{d-2}} \log \bbP[\widetilde{\cA}^{a,(\widetilde{\Delta}-\vartheta)/|J|,\chi_\epsilon}_N \cap \cD_{N,\kappa}] \bigg) \vee \bigg( \varlimsup_N \sup_{\kappa \in \widetilde{\cK}_N^\mu }\frac{1}{N^{d-2}} \log \bbP[\cD_{N,\kappa}] \bigg).
\end{split}
\end{equation}
We will now bound the two terms contributing to the maximum on the right-hand side of \eqref{eq:SplittedVersionBound} separately. Since a solidification result is needed in both cases, the following observation is important: We fix a compact set $A' \subseteq \mathring{A}$ and some $\ell^\ast \geq 0$ (depending on $A,A'$), such that for large $N$ and all $\kappa \in \cK_N$, $d(A',U_1)\geq 2^{-\ell*}$. In particular $W_x[H_{\widehat{\Sigma}} < \tau_{10\widehat{L}_0/N}] \geq c(K)$ for all $x \in \partial U_0$ (again, we refer to (4.48)--(4.54) in~\cite{nitzschner2017solidification}).  Thus, $\widehat{\Sigma}$, and consequently $\Sigma$, is a porous interface for $A'$ and large $N$. 

To find a bound on the second term of~\eqref{eq:SplittedVersionBound}, we will make use of the solidification result for Dirichlet forms, Lemma \ref{thm:SolidifDirichlet}. Indeed, 
\begin{equation}
\label{eq:BoundD_LargeDirichletform}
\begin{split}
\varlimsup_N  \sup_{\kappa \in \widetilde{\cK}^\mu_N} \frac{1}{N^{d-2}} & \log \bbP[\cD_{N,\kappa}]  \leq \varlimsup_N  \sup_{\kappa \in \widetilde{\cK}_N^\mu}\frac{1}{N^{d-2}} \log \bbP\left[\bigcap_{z\in \cC} \{ \inf_{D_z} h^z \leq -a \} \right] \\
&\leq  -\varliminf_N \inf_{\kappa \in \widetilde{\cK}_N^\mu} \frac{1}{2} \bigg(a - \frac{c_3}{K}\sqrt{\tfrac{|\cC|}{\capa_{\bbZ^d}(\widehat{C})}}\bigg)_+^2 \frac{\capa_{\bbZ^d}(\widehat{C})}{\alpha(K)N^{d-2}},
\end{split}
\end{equation}
where $\alpha(K) \rightarrow 1$ as $K \rightarrow \infty$, see (3.26) of \cite{nitzschner2018entropic}, with the difference that we choose $\widehat{C}$ slightly larger than $C$ such that  the $\bbR^d$-filling of $\widehat{C}$ contains $N\widehat{\Sigma}$ (for concreteness take $\widehat{C} = \bigcup_{z \in \cC} (z + [-\tfrac{L_0}{4},\tfrac{5L_0}{4})^d\cap \bbZ^d$). Taking $\liminf_K$ on both sides and using the first inequality of \eqref{eq:liminf}, we obtain that
\begin{equation}
\varliminf_K \varliminf_N \frac{1}{N^{d-2}} \capa_{\bbZ^d}(\widehat{C}) \geq \varliminf_K \varliminf_N \frac{1}{d} \capa(\widehat{\Sigma}).
\end{equation}
Moreover, one also has, with the help of Lemma~\ref{thm:SolidifDirichlet} on the third line, 
\begin{equation}
\label{eq:BoundSigma_LargeDirichletform}
\begin{split}
&\varliminf_N  \inf_{\kappa \in \cK_N^\mu} \capa(\widehat{\Sigma})  = \capa(A') + \varliminf_N \inf_{\kappa \in \cK_N^\mu} (\capa(\widehat{\Sigma}) - \capa(A')) \\
&\ = \capa(A') + \varliminf_N \inf_{\kappa \in \cK_N^\mu} \left( \cE(h_{A'} - h_{\widehat{\Sigma}}) + (\capa(\widehat{\Sigma}) - \capa(A') - \cE(h_{A'} - h_{\widehat{\Sigma}}) \right) \\
&\ \stackrel{\eqref{eq:BoundDirichletForm}}{\geq} \capa(A') +  \varliminf_N \inf_{\kappa \in \cK_N^\mu} \cE(h_{A'} - h_{\widehat{\Sigma}}) \\
&\ \geq \capa(A') + (\cE(h_{\mathring{A}}-h_{A'})^{\frac{1}{2}} - \mu^{\frac{1}{2}})^2,
\end{split}
\end{equation}
where we introduced the shorthand notation $\cE(f) = \cE(f,f)$ and used Lemma \ref{thm:SolidifDirichlet} for the first inequality. Collecting \eqref{eq:BoundD_LargeDirichletform}--\eqref{eq:BoundSigma_LargeDirichletform} and using \eqref{eq:BoundCoverCap}, we have obtained that 
\begin{equation}
\label{eq:BoundSplit_1}
\varlimsup_{K \rightarrow \infty}\varlimsup_{N \rightarrow \infty}   \sup_{\kappa\in\widetilde{\cK}_N^\mu} \frac{1}{N^{d-2}} \log \bbP[\cD_{N,\kappa}] \leq -\frac{a^2}{2d}\Big(\capa(A') + (\cE(h_{\mathring{A}}-h_{A'})^{\frac{1}{2}} - \mu^{\frac{1}{2}})^2\Big).
\end{equation}
Going back to \eqref{eq:SplittedVersionBound}, we now need to bound the first term of the maximum as well. 
Taking $\liminf_K$, using Proposition \ref{thm:Prop43} as well as the arguments around \eqref{eq:liminf} and \eqref{eq:capacitysol}, we see that 
\begin{equation}
\label{eq:BoundSplit_2}
\begin{split}
\varlimsup_{K \rightarrow \infty}\varlimsup_{N \rightarrow \infty}  \sup_{\kappa \in \cK_N^\mu} &\frac{1}{N^{d-2}} \log \bbP[\widetilde{\cA}^{a,(\widetilde{\Delta}-\vartheta)/|J|,\chi_\epsilon}_N \cap \cD_{N,\kappa} ] \\ &\leq - \frac{1}{2d}\Big(a + \beta \tfrac{\widetilde{\Delta} - \vartheta}{|J|} \Big)^2\frac{\capa(A')}{1 + \beta^2c_6(M)dE(\chi_\epsilon) + 2\beta \sqrt{\mu}\sqrt{E(\chi_\epsilon)}}.
\end{split}
\end{equation}
Taking $A' \uparrow \mathring{A}$, we set $\mu = \beta^2$ and then choose $\beta > 0$ small enough such that the right-hand sides of \eqref{eq:BoundSplit_1} and \eqref{eq:BoundSplit_2} are both in absolute value bigger than $\frac{a^2}{2d} \capa(\mathring{A})$. Inserting these bounds back into \eqref{eq:SplittedVersionBound}, one finds that
\begin{equation}
\label{eq:PenultimateStep}
\varlimsup_{K \rightarrow \infty}\varlimsup_{N \rightarrow \infty} \sup_{\kappa \in \cK_N} \frac{1}{N^{d-2}} \log \bbP[\widetilde{\cA}^{\overline{h}-\alpha,|J|^{-1}(\widetilde{\Delta}-\vartheta),\chi_\epsilon}_N \cap \cD_{N,\kappa} ] \leq - \frac{1}{2d}a^2\capa(\mathring{A}) + c_8(E(\chi_\epsilon), \widetilde{\Delta},J,\vartheta).
\end{equation}

The proof is now concluded as follows. We recall the upper bound \eqref{eq:FirstBoundProof}, and find for $\vartheta < \Delta$, that 
\begin{equation}\label{eq:finalstep}
\varlimsup_N \frac{1}{N^{d-2}} \log \bbP\Big[d_J(\bbX_N,\cH^\alpha_{\mathring{A}}) \geq \Delta; \cD^\alpha_N \Big] \leq -\bigg(\frac{c_4\vartheta^2}{8\epsilon^2} \bigg) \wedge \bigg( \frac{a^2}{2d} \capa(\mathring{A}) + c_8(E(\chi_\epsilon), \widetilde{\Delta},J,\vartheta) \bigg).
\end{equation}
Letting $\delta \rightarrow \overline{h}$ and using \eqref{eq:RelationForCapA}, together with $\widetilde{\Delta}\rightarrow \Delta$, we obtain the claim. 
\end{proof}
\begin{remark}
\label{thm:Remark46}
1) The uniform bound on the variance of $\bar{Z}_{f,x}$ in Lemma \ref{thm:Lemma44} is actually more general: Using the fact that $\sup_{\eta \in \lip_1(J)} E(|\eta|) < \infty$, one can also prove a modification of Lemma~\ref{thm:Lemma44} where $\bar{Z}_{f,x}$ is replaced by $\bar{Z}_{f,\eta} = \bar{Z}_{f,\eta}  = Z_f(1 - \beta \langle  h_{\mathring{A}},\eta \rangle) + \beta \langle \bbX_N,\eta \rangle$, and $\sup_{x \in J}$ is replaced by $\sup_{\eta \in \lip_1(J)}$. Such a modification is however not possible in Lemma \ref{thm:Lemma45}, which essentially forces the consideration of the location family $\chi_\epsilon(\cdot - x)$ as a remedy, since the index set $x\in J$ has a lower dimension then $\lip_1(J)$.  \\
2) The lack of a more explicit rate in \ref{thm:MainUpperBound} compared to the `entropic push-down' result of Theorem \ref{thm:pushdown} is not only due to taking the supremum over $\eta \in \lip_1(J)$, but also comes with the fact that the pointwise asymptotic lower bound of $P_x[H_C < \infty]$ by $h_{\mathring{A} }(x/N)$ in \eqref{eq:LiminfBound} and \eqref{eq:pointwisesolid}, which is due to the solidification estimates, is not complemented by a corresponding pointwise upper bound. For this reason, the cases $\kappa \in \cK^\mu_N$ and $\kappa \in \widetilde{\cK}^\mu_N$ had to be treated separately, leading to the implicit derivation of the constant $c_8$ in \eqref{eq:PenultimateStep}. \\
3) Incidentally, a result similar to Theorem \ref{thm:MainUpperBound} may be obtained in the case where the disconnection event $\cD^\alpha_N$ is replaced by the event that an adequately thickened component of the boundary of the closed $| \cdot |_\infty$-ball in $\bbZ^d$ with center $0$ and radius $N$ (this is $S_N$ with the choice $M = 1$ in \eqref{eq:BlowUp_Boundary})  in $E^{\geq \alpha}$ leaves in the box a macroscopic volume in its complement when $\alpha < \overline{h}$, a situation discussed in \cite{sznitman2018macroscopic}. Roughly speaking, a combination of~\eqref{eq:Equality_crit_par}, Theorem 4.1 and Theorem 4.2 of this reference imply that, conditionally on such a `macroscopic hole event', the scaled $\bbR^d$-filling of the left-out set by the component of the boundary of $S_N$ in $E^{\geq \alpha}$ is close to a translate of the Euclidean ball $B$ in an $L^1$-sense (see also Remark 4.3 of the same reference). It might be argued that an approach similar to the one developed to prove the pinning under disconnection in Theorem \ref{thm:MainUpperBound} could be applicable to show that conditionally on the existence of a `macroscopic hole', the local averages of the Gaussian free field are pinned at a multiple of the harmonic potential of the corresponding translate of $B$ in $\bbR^d$. 
\end{remark}

\section{Profile description under disconnection}

In this section, we analyze the local picture of the Gaussian free field conditioned on disconnection and derive a certain `profile' description of the measure $\bbP[\ \cdot \ | \cD^\alpha_N ]$. We define the random measure
\begin{equation}
\bbY_N = \frac{1}{N^d} \sum_{x \in \bbZ^d} \delta_{x/N} \otimes \delta_{\tau_x \varphi}, \qquad \text{ on } \bbR^d \times \bbR^{\bbZ^d},
\end{equation}
where $\bbR^d$ is equipped with its Borel $\sigma$-algebra, $\bbR^{\bbZ^d}$ is equipped with the $\sigma$-algebra generated by the canonical coordinates on $\bbZ^d$ and $\tau_x f = f(\cdot + x)$ denotes the shift by $x \in \bbZ^d$ of the function $f: \bbZ^d \rightarrow \bbR$. \vspace{\baselineskip}

The main result of this section, which comes in Theorem \ref{thm:ProfileDescription} below, provides insight into the asymptotic behavior of $\bbY_N$, when $\cD^\alpha_N$ occurs. Roughly speaking, it accompanies the entropic pinning result of the previous section by the following microscopic description: If $A$ is regular in the sense that $\capa(\mathring{A}) = \capa(A)$, then on a local scale, $\varphi$ retains its structure as a Gaussian free field under $\bbP[ \ \cdot \ |\cD^\alpha_N]$, while its average is pinned at $\cH^\alpha_A$ on a global scale.

Before stating the result, we introduce some necessary notation. For a function $f: \bbZ^d \to \bbR$ we define the probability  measure $\bbP^f$ on $\bbR^{\bbZ^d}$  by
\begin{equation}\label{eq:shifted_GFF}
  \text{$\varphi$ under $\bbP^f$ has the same law as $\varphi + f$ under $\bbP$.}
\end{equation}
We also define for any given measurable $u\,:\, \bbR^d\to \bbR$ the measure $\Phi(u)$ on $\bbR^d\times \bbR^{\bbZ^d}$ by
\begin{equation}\label{eq:shifted measure}
    \Phi(u)(\De x, \De f) = \De x \otimes \bbP^{\mathbbm{1} u(x)}(\De f),
\end{equation}
where $\mathbbm{1}$ denotes the function on $\bbZ^d$ which is constant and equal to one.

A function $F\,:\, \bbR^{\bbZ^d} \to \bbR$ is called \emph{local} if there exists a finite set $\Gamma \subseteq \bbZ^d$ such that $F(f) = F(g)$ whenever $f|_\Gamma = g|_\Gamma$. 
Note that if $F$ is a local function, then there exists a function $\overline{F} : \bbR^\Gamma \rightarrow \bbR$ such that for any $g \in \bbR^{\bbZ^d}$, one has
\begin{equation}
\label{eq:Locality}
F(g) = \overline{F}(g|_\Gamma).
\end{equation}
We say that a local function $F: \bbR^{\bbZ^d} \to \bbR$ is \emph{Lipschitz} with constant $K > 0$ (or $K$-Lipschitz) if
\begin{equation}\label{eq:lip}
  |F(f)-F(g)| \leq K \|f - g\|_{\infty},\quad \text{ for any } f,g \in \bbR^{\bbZ^d}.
\end{equation}
We are ready to present the main result of this section.
\begin{theorem}
\label{thm:ProfileDescription}Let $\Delta > 0$ and $\alpha < \overline{h}$ be fixed. Let  $\eta\,:\,\bbR^d \to \bbR$ be a compactly supported Lipschitz function and $F: \bbR^{\bbZ^d} \to \bbR$ a bounded local Lipschitz function. Then,
\begin{equation}
\begin{split}
\label{eq:StatementProfileThm51}
  \limsup_{N\to \infty}\frac{1}{N^{d-2}} \log \, & \bbP\Big[\big|\langle  \Phi(\cH^\alpha_{\mathring{A}}) -  \bbY_N, \eta\otimes F \rangle\big| \geq \Delta;\,\cD^{\alpha}_N \Big]  \\
  & \leq  - \frac{1}{2d} (\overline{h}-\alpha)^2\capa(\mathring{A}) -c_2(\Delta,\alpha,\eta,F).
  \end{split}
\end{equation}
\end{theorem}
One has again an important consequence from Theorem \ref{thm:ProfileDescription}, \eqref{eq:DiscLowerBound} and~\eqref{eq:Equality_crit_par}, see~\cite{duminil2019equality}, which follows in the same way as Corollary \ref{thm:Corollary32} and Corollary \ref{thm:Corollary42}:
\begin{corollary}
\label{thm:Corollary52}
Consider $\Delta, \alpha, \eta$ and $F$ as in Theorem \ref{thm:ProfileDescription} and assume that $\capa(A) = \capa(\mathring{A})$. Then, one has 
\begin{equation}
\limsup_{N\to \infty} \frac{1}{N^{d-2}} \log \bbP\Big[\big|\langle \Phi(\cH^\alpha_{A}) -  \bbY_N, \eta\otimes F \rangle\big| \geq \Delta \big\vert \cD^{\alpha}_N \Big] \leq - c_2(\Delta,\alpha,\eta,F).
\end{equation}
\end{corollary}
We provide a short outline of the proof of Theorem \ref{thm:ProfileDescription}, which is very much inspired by an argument in~\cite{bolthausen1993critical}, where large deviation estimates for the profile of the one-marginal of the Gaussian free field without conditioning were derived. The basic idea of the proof consists in separating the local picture from the macroscopic behavior by conditioning on the values of the Gaussian free field on a `diluted' lattice $L\bbZ^d$ at the mesoscopic scale $L = \lfloor \log N \rfloor$. Using such a decomposition gives rise to a random measure $Z_N$ on $\bbR^d \times \bbR^{\bbZ^d}$, and we show the closeness of $\bbY_N$ and $\Phi(\cH^\alpha_{\mathring{A}})$  essentially in a three-step procedure: In a first step, we will show that the probability of a large deviation of the measures $\bbY_N$ and $Z_N$, tested against $\eta \otimes F$ is exponentially small at a rate $N^{d-\delta}$, for any $\delta > 0$, see Proposition \ref{prop:YclosetoZ}, which is a modification of Proposition 3.10 of \cite{bolthausen1993critical}. In a second step, we show that the probability of a large deviation of the measures $Z_N$ and $\Phi(\bbX_N^\epsilon)$, tested against $\eta \otimes F$, decays with a rate faster than $N^{d-2}$, where $\bbX_N^\epsilon$ is a regularized version of the measure $\bbX_N$, see Proposition \ref{prop:ZclosetoX}. Finally, we show in Proposition \ref{prop:XclosetoTarget}, by means of Theorem~\ref{thm:MainUpperBound}, that the probability of a large deviation of the measures $\Phi(\bbX_N^\epsilon)$ and $\Phi(\cH^\alpha_{\mathring{A}})$, tested against $\eta \otimes F$ \emph{and} disconnection decays with rate $N^{d-2}$. The claim will then follow by choosing $\epsilon > 0$ small enough. Importantly, the disconnection event only plays a role in the derivation of the last step, while the approximations of $\bbY_N$ by $Z_N$ and of $Z_N$ by $\Phi(\bbX_N^\epsilon)$ in Propositions \ref{prop:YclosetoZ} and \ref{prop:ZclosetoX} are generic and only depend on the Gaussian nature of $\varphi$ under $\bbP$.
\vspace{\baselineskip}

\emph{The field at mesoscopic scale $Z_N$.} We condition the field $\varphi_x$, $x\in\bbZ^d$ on its values on the diluted lattice $L\bbZ^d$, where $L = \lfloor \log N \rfloor$. The field decomposes as
\begin{equation}\label{eq:decomposition}
  \varphi_x  = \psi_x + h_x, \quad x\in \bbZ^d,
\end{equation}
(see~\eqref{eq:HarmonicAverage} and \eqref{eq:LocalField}, with $U = \bbZ^d \setminus L \bbZ^d$) where $\psi_x$ under $\bbP$ is a Gaussian free field with zero boundary conditions on $L \bbZ^d$, independent from $h_x = \bbE[\varphi_x | \varphi_z,\, z\in L\bbZ^d]$. We denote by
\begin{equation}\label{eq:covariances}
  g^L(x,y) = \bbE[\psi_x\psi_y],\quad g^{h}(x,y) = g(x,y) -  g^L(x,y),
\end{equation}
the corresponding covariances.
It is known (see Lemma 3.8 of~\cite{bolthausen1993critical}) that for all $x,y\in \bbZ^d$, one has the upper bound
\begin{equation}\label{eq:decay}
  g^L(x,y) \leq c \exp \Big(-c' |x-y|/L^{d-2} \Big).
\end{equation}
Property~\eqref{eq:decay} is crucial to show that on a local level the field $\psi$ `does not feel' long range effects.

We define $Z_N$ to be the random measure on $\bbR^d\times \bbR^{\bbZ^d}$ given by
\begin{equation}\label{eq:defZ}
  Z_N = \frac{1}{N^d} \sum_{x\in \bbZ^d} \delta_{x/N}\otimes \bbP^{\tau_xh}.
\end{equation}
We now show that the probability that $\bbY_N$ deviates from $Z_N$ (in a weak sense) is exponentially small with decay rate $N^{d-\delta}$ for some $\delta>0$. In particular, the decay is much faster than that of the disconnection probability.

\begin{prop}\label{prop:YclosetoZ} For any $\Delta, \delta>0$, any compactly supported Lipschitz function $\eta\,:\,\bbR^d\to \bbR$ and any bounded local Lipschitz function $F\,:\, \bbR^{\bbZ^d}\to \bbR$, we have
\begin{equation}\label{eq:tightness}
  \lim_{N\to \infty} \frac{1}{N^{d-\delta}} \log \bbP\Big[\big|\langle  \bbY_N  - Z_N, \eta\otimes F \rangle\big| \geq \Delta \Big] = -\infty.
\end{equation}
\end{prop}
\begin{proof} The proof is inspired by that of Proposition 3.10 of~\cite{bolthausen1993critical} with the difference that in~\cite{bolthausen1993critical} the authors deal only with the one-point marginal of the Gaussian free field. Without loss of generality we assume that $\|\eta\|_\infty\leq 1$ and that $F$ is $1$-Lipschitz.
First note that
\begin{equation}
  \big|\langle \bbY_N -Z_N,\eta\otimes F \rangle\big| \leq \frac{1}{N^d}\biggl| \sum_{x\in \bbZ^d} \eta\left(\frac{x}{N}\right) \Big(F(\tau_x\varphi) - \bbE^{\tau_x h}[F]\Big)\biggr|.
\end{equation}
Using the decomposition $\varphi = \psi + h$ from \eqref{eq:decomposition} and the independence between $\psi$ and $h$, we get
\begin{equation}
\begin{split}\label{eq:step42}
  &\bbP\bigg[\frac{1}{N^d}\biggl| \sum_{x\in \bbZ^d} \eta\left(\frac{x}{N}\right) \Big(F(\tau_x\varphi) - \bbE^{\tau_x h}[F]\Big)\biggr|\geq \Delta\bigg] \\
  &\qquad \leq \sup_{f\in \bbR^{\bbZ^d}}
  \bbP\bigg[\frac{1}{N^d}\biggl| \sum_{x\in \bbZ^d} \eta\left(\frac{x}{N}\right) \Big(F(f+\tau_x\psi) - \bbE^{f}[F]\Big)\biggr|\geq \Delta\bigg].
\end{split}
\end{equation}
We claim the following
\begin{equation}\label{eq:claim}
  \lim_{N\to \infty} \sup_{f\in \bbR^{\bbZ^d}}\frac{1}{N^d}\bigg|\sum_{x\in (\supp \eta)_N}\Big( \bbE^f[F]- \bbE[F(f+\tau_x\psi)]\Big)\bigg| = 0,
\end{equation}
where we recall the notation $(\supp \eta)_N = (N \supp \eta) \cap \bbZ^d$. The justification of~\eqref{eq:claim} will be delayed to the end of the proof. Given~\eqref{eq:claim}, to conclude it suffices to show that, uniformly in $f\in \bbR^{\bbZ^d}$,
\begin{equation}\label{eq:bound}
  \begin{split}
    \bbP\bigg[\frac{1}{N^d}\biggl| \sum_{x\in \bbZ^d} \eta\left(\frac{x}{N}\right) \Big(F(f+\tau_x\psi) - \bbE[F(f+\tau_x\psi)]\Big)\biggr|\geq \Delta\bigg]
    \leq e^{-c N^{d-\delta}},
  \end{split}
  \end{equation}
  for large enough $N$.
Since $F$ is local there exists a set $\Gamma\subseteq \bbZ^d$ finite and $\overline{F} : \bbR^{\Gamma}\to \bbR$ fulfilling the locality property \eqref{eq:Locality}. Set $\Gamma^N_{\eta} = (\supp\eta)_N + \Gamma$ and notice that the functions
\begin{equation}
  G_f^{\pm} : \bbR^{\Gamma^N_{\eta}}\to \bbR,\quad G_f^{\pm}(g) =  \pm\frac{1}{N^d} \sum_{x\in (\supp \eta)_N} \eta\left(\frac{x}{N}\right) \overline{F}\Big(f|_\Gamma+(\tau_x g)|_\Gamma\Big)
\end{equation}
are Lipschitz for any $f\in \bbR^{\bbZ^d}$.
Then, an application of the Gaussian concentration inequality (A.5) from Lemma A.4 of~\cite{bolthausen1993critical} yields  
\begin{equation}
  \begin{split}
    &\bbE\bigg[\exp\biggl(\pm\frac{1}{N^d} \sum_{x\in \bbZ^d} \eta\left(\frac{x}{N}\right) \Big[F(f+\tau_x\psi) - \bbE[F( f+ \tau_x \psi)]\Big] \biggr)\bigg]\\
    &\qquad =  \bbE\Big[\exp\Bigl(G_f^{\pm}(\psi|_{\Gamma^N_\eta}) - \bbE[G_f^{\pm}(\psi|_{\Gamma^N_\eta})]\Big)\Big]\leq \exp\biggl( \frac{c(\Gamma)}{N^{2d}} \sum_{x,y\in \Gamma^N_\eta} |\bbE[\psi_x \psi_y]| \biggr),
  \end{split}
\end{equation}
where we used that $F$ is $1$-Lipschitz and $\|\eta\|_\infty \leq 1$. Using that $\bbE[\psi_x \psi_y] = g^L(x,y)$,~\eqref{eq:decay} and  Chernoff's inequality,~\eqref{eq:bound} follows in a standard way. This combined  with~\eqref{eq:claim} and~\eqref{eq:step42} yields~\eqref{eq:tightness}. 

We are left with proving the claim~\eqref{eq:claim}. Relying on the fact that $F$ is local and $1$-Lipschitz, we  get 
  \begin{align}
    \sup_{f\in \bbR^{\bbZ^d}}\bigg|\frac{1}{N^d} \sum_{x\in (\supp \eta)_N}& \Big(  \bbE^f[F]- \bbE[F(f+\tau_x\psi)]\Big)\bigg| \notag \\ & \leq c(\Gamma) \max_{y,z\in \Gamma}  \frac{1}{N^d} \sum_{x\in (\supp \eta)_N} \Big|g(y,z) - g^L(x+y,x+z)\Big|^{1/2}, \label{eq:RHS}
  \end{align}
uniformly in $f\in \bbR^{\bbZ^d}$. In order to show that the right-hand side goes to zero, we introduce for a sequence $\widehat{L} = o(L)$ depending on $N$, $\widehat{L}\to \infty$, 
\begin{equation}\label{eq:boxesC_N}
C_N = \bigcup_{x\in L\bbZ^d} \Big(x + [-\widehat{L},\widehat{L}]^d\Big) \cap( \supp\eta)_N,
\end{equation}
and observe that $|C_N| = o(N^d)$. Also, for each $x\in ( \supp\eta)_N\setminus C_N$, $y,z\in \Gamma$, using the strong Markov property, we get for large $N$
\begin{equation}
  \begin{aligned}
  g(y,z) &- g^L(x+y, x+z) = E_{x+y}\bigg[\sum_{n=H_{L\bbZ^d}}^\infty \IND_{\{X_n = x+z\}}\bigg] \\ 
  & = E_{x+y}\big[\IND_{\{H_{L\bbZ^d} <\infty\} } g(X_{H_{L\bbZ^d}},x+z)\big] \stackrel{\eqref{eq:AsymptoticBehaviourGreen}}{\leq} c\widehat{L}^{2-d}.
  \end{aligned}
\end{equation}
We can now estimate~\eqref{eq:RHS} by splitting the sum on the right hand side into two sums over over $C_N$ and $( \supp\eta)_N\setminus C_N$. 
\begin{equation*}
  \begin{aligned}
    \max_{z,y\in \Gamma}\frac{1}{N^d} \sum_{x\in (\supp \eta)_N} \Big|g(y,z) - g^L(x+y,x+z)\Big|^{1/2} \leq g(0,0)^{1/2}\frac{|C_N|}{N^d} + c\widehat{L}^{1-d/2}.
  \end{aligned}
\end{equation*}
Taking $N\to \infty$ proves~\eqref{eq:claim} and hence the Proposition.
\end{proof}

\emph{Identification of the global shift.} Using the mesoscopic scale $L = \lfloor \log N \rfloor$, we were able to show that the local picture associated to $\bbY_N$ has a law asymptotically equal to that of a Gaussian free field shifted by the harmonic extension of the Gaussian free field on $L\bbZ^d$. In what follows we show that this shift can be approximated at each point by averaging the Gaussian free field around a macroscopic ball.

To this end, we introduce a smoothing of the random measure $\bbX_N$ on $\bbR^d$ (cf.~\eqref{eq:measureGFF}). For that, let $\chi : \bbR^d \to [0,\infty)$ be a symmetric smooth probability density with support contained in the Euclidean unit ball and for $\epsilon > 0$ let $\chi_\epsilon(x) = \epsilon^{-d} \chi(x/\epsilon)$. Also, we define
\begin{equation}\label{eq:smoothing}
  \bbX_N^\epsilon(x) = \chi_\epsilon \ast \bbX_N(x) = \frac{1}{N^d}\sum_{z\in \bbZ^d} \chi_\epsilon\left(\frac{z}{N}-x\right) \varphi_z.
\end{equation}

In the next proposition we show that the probability that the random measures $\Phi(\bbX_N^\epsilon)$ and $Z_N$ tested against $\eta\otimes F$ deviate from each other is super-exponentially small at rate $N^{d-2}$. 
\begin{prop}\label{prop:ZclosetoX}
Let $\Delta, \epsilon>0$ be fixed. Then, for any compactly supported Lipschitz function $\eta\,:\,\bbR^d\to \bbR$ and any bounded local Lipschitz function $F\,:\, \bbR^{\bbZ^d}\to \bbR$, we have
\begin{equation}
  \lim_{N\to\infty}\frac{1}{N^{d-2}} \log \bbP\Big[\big|\langle  \Phi(\bbX^\epsilon_N) -  Z_N,\eta\otimes F \rangle\big| \geq \Delta\Big] = -\infty.
\end{equation}
\end{prop}
\begin{proof}Since the arguments are very close to those in the proof of Lemma 3.12 of~\cite{bolthausen1993critical} we only sketch the proof. In fact, we only need to show that we can reduce to the one-marginal case, which is the one treated in \cite{bolthausen1993critical}. Without loss of generality we can assume that $\|\eta\|_{BL}\leq 1$ and that $F$ is $1$-Lipschitz. We introduce the random measure
\begin{equation}
  \widetilde{Z}_N = \frac{1}{N^d}\sum_{x\in \bbZ^d} \delta_{x/N} \otimes \bbP^{\mathbbm{1} h_x},
\end{equation}
which is a modification of $Z_N$ that allows a better comparison to $\Phi(\bbX_N^\epsilon)$, see \eqref{eq:ZTildeCloseToX} below. We will now show that for any $\delta>0$
\begin{equation}\label{eq:tildeZprofile}
  \lim_{N\to\infty}\frac{1}{N^{d-2}} \log \bbP\Big[\big|\langle  \widetilde{Z}_N -  Z_N,\eta\otimes F \rangle\big| \geq \delta\Big] = -\infty,
\end{equation}
implying that $\widetilde{Z}_N$ is indeed a `good' approximation for $Z_N$ for our purposes. 
Let us denote by $\Gamma \subseteq \bbZ^d$ the finite set associated to the local function $F$.
By the boundedness of $\eta$ and the Lipschitz continuity of $F$, to prove~\eqref{eq:tildeZprofile} it suffices to bound
\begin{equation}\label{eq:step23}
  \begin{split}
    \bbP\Big[\frac{1}{N^d}\sum_{x\in (\supp \eta)_N} &\sup_{y\in \Gamma}\big|h_{x} - h_{x+y}\big| \geq \delta\Big] \leq |\Gamma|\sup_{y\in \Gamma}
    \bbP\Big[\frac{1}{N^d}\sum_{x\in (\supp \eta)_N}\big|h_{x} - h_{x+y}\big| \geq \delta  \Big].
  \end{split}
\end{equation}
As a first step, we split the sum under the probability on the right-hand side of~\eqref{eq:step23} into two sums over $C_N$ and $(\supp \eta)_N\setminus C_N$, where $C_N$ was defined in~\eqref{eq:boxesC_N}. 
Let us start with
\begin{equation}\label{eq:eq23}
  \frac{1}{N^d} \sum_{x\in C_N} |h_{x+y} - h_x| \leq  \frac{1}{N^d} \sum_{x\in C_N} |h_{x+y}| +\frac{1}{N^d} \sum_{x\in C_N} |h_x|.
\end{equation}
We shall focus on $\frac{1}{N^d} \sum_{x\in C_N} |h_x|$ as the other term can be treated similarly. Denote the covariance matrix of the field $h$ over $C_N$ by $G_{1,N} = (g^h(x,y))_{x,y\in C_N}$ (see \eqref{eq:covariances}). In view of $g^h(x,y) \leq g(x,y)$ and~\eqref{eq:BoundsCovMatrixGFF}, one can show that
\begin{equation}
  \tr(G_{1,N})\sim c\big(\tfrac{N \widehat{L} }{L}\big)^d,\quad \overline{G}_{1,N} = O\Big(\big(\tfrac{N \widehat{L} }{L}\big)^2\Big),\quad \tr (G_{1,N}^2) = O\Big(\big(\tfrac{N \widehat{L} }{L}\big)^{d+1}\Big).
\end{equation}
An application of the Gaussian estimate \eqref{eq:GaussianEstimate} allows us to conclude that
\begin{equation}
\label{eq:GaussianEstimateG1}
\begin{split}
& \bbP\Big[\frac{1}{N^d}\sum_{x\in C_N}\big|h_{x}\big| \geq \delta/3\Big]  \leq \bbP\Big[\sum_{x\in C_N}  h_{x}^2 \geq \tfrac{\delta^2}{9}\tfrac{N^{2d}}{|C_N|} \Big] \\
& \qquad \leq \exp\Big\{-\frac{c}{8}\min\Big( \tfrac{\delta^2}{9} N^{d-2} \big(\tfrac{L}{\widehat{L}}\big)^{d+2} , \tfrac{\delta^4}{81}  N^{d-1} \big(\tfrac{L}{\widehat{L}}\big)^{3d+1} \Big)\Big\},
\end{split}
\end{equation}
and thus we obtain that
\begin{equation}\label{eq:partone}
  \lim_{N\to \infty}\frac{1}{N^{d-2}}\log\bbP\Big[\frac{1}{N^d}\sum_{x\in C_N}\big|h_{x}\big| \geq \delta/3\Big]  = -\infty.
\end{equation}
The same calculation can be performed for the case where $h_x$ is replaced by $h_{x+y}$, $y\in \Gamma$.
For the sum over $(\supp \eta)_N\setminus C_N$ we need to investigate the covariances of $(h_{x+y}-h_x)_{x\in (\supp \eta)_N\setminus C_N}$, which we denote by $G_{2,N}$. For any $x,z\in (\supp \eta)_N\setminus C_N$, using the random walk representation \eqref{eq:HarmonicAverage} for $h_x$ and combining the Harnack inequality and the gradient estimates in Theorem 1.7.2 and 1.7.1 of \cite{lawler2013intersections} for $P_\cdot[X_{H_{L\bbZ^d}} = kL]$, which is non-negative and harmonic on $\bbZ^d\setminus L\bbZ^d$, one obtains for $N$ large enough
\begin{equation}
  \begin{aligned}
   |(G_{2,N})_{x,z}| = \Big\vert \bbE & \Big[(h_{x+y}-h_x)(h_{z+y}-h_z)\Big] \Big\vert \\
  &\leq \sum_{k,h\in \bbZ^d} \bbE[\varphi_{kL}\varphi_{hL}] \Big| P_{x+y}[X_{H_{L\bbZ^d}} = kL]- P_{x}[X_{H_{L\bbZ^d}} = kL]
  \Big| \\ &\qquad \qquad\cdot\Big|P_{z+y}[X_{H_{L\bbZ^d}} = hL]- P_{z}[X_{H_{L\bbZ^d}} = hL]\Big| \\ &\leq \frac{c}{\widehat{L}^2} \sum_{k,h\in
  \bbZ^d}\bbE[\varphi_{kL}\varphi_{hL}]P_{x}[X_{H_{L\bbZ^d}} = kL]P_{z}[X_{H_{L\bbZ^d}} = hL]\\
  & = \frac{c}{\widehat{L}^2} \bbE[h_x h_z] \stackrel{\eqref{eq:covariances}}{\leq} \frac{c}{\widehat{L}^2} g(x,z).
  \end{aligned}
  \end{equation}
The above estimate combined with~\eqref{eq:BoundsCovMatrixGFF} yields
\begin{equation}
  \tr(G_{2,N}) = O\Big(\tfrac{N^d}{\widehat{L}^2}\Big),\quad \overline{G}_{2,N} = O\Big(\tfrac{N^2}{\widehat{L}^2}\Big),\quad \tr (G_{2,N}^2) = O\Big(\tfrac{N^{d+1}}{\widehat{L}^4}\Big).
\end{equation}
Again, an application of~\eqref{eq:GaussianEstimate} similar to \eqref{eq:GaussianEstimateG1} allows to conclude that
\begin{equation}\label{eq:partwo}
  \lim_{N\to \infty}\frac{1}{N^{d-2}}\log\bbP\Big[\frac{1}{N^d}\sum_{x\in (\supp \eta)_N\setminus C_N}\big|h_{x+y} - h_{x}\big| \geq \delta/3\Big]  = -\infty.
\end{equation}
A combination of~\eqref{eq:step23},~\eqref{eq:partone},~\eqref{eq:partwo} together with the fact that $\Gamma$ is finite shows~\eqref{eq:tildeZprofile}.

To prove the Proposition, it therefore suffices to show that for $0 <\delta<\Delta$, one has
\begin{equation}\label{eq:stepBatman}
  \lim_{N\to\infty}\frac{1}{N^{d-2}} \log \bbP\Big[\big|\langle  \Phi(\bbX^\epsilon_N) -  \widetilde{Z}_N, \eta\otimes F \rangle\big| \geq \Delta-\delta\Big] = -\infty.
\end{equation}
The fact that $F$ is $1$-Lipschitz and bounded and $\|\eta\|_{BL} \leq 1$  yields the bound
\begin{equation}
\label{eq:ZTildeCloseToX}
\begin{split}
  & \big|\langle \Phi(\bbX^\epsilon_N) -  \widetilde{Z}_N ,\eta\otimes F \rangle\big| \leq \frac{1}{N^d} \bigg|  \sum_{x \in (\supp\eta)_N} \eta\big(\tfrac{x}{N} \big) \bigg(\bbE^{\IND \bbX_N^\epsilon(\tfrac{x}{N})}[F] - \bbE^{\IND h_x}[F] \bigg) \bigg| \\
 & + \bigg| \int \eta(y)\bbE^{\IND \bbX_N^\epsilon(y)}[F]\De y - \frac{1}{N^d} \sum_{x \in (\supp\eta)_N} \eta\big(\tfrac{x}{N} \big) \bbE^{\IND \bbX_N^\epsilon(\tfrac{x}{N})}[F] \bigg|
   \\ &  
  \leq \frac{c}{N^d}\bigg(  \sum_{x\in (\supp\eta)_N} |\bbX^\epsilon_N(x/N) - h_x| + \frac{\epsilon^{-d-1}}{N} \sum_{x\in (\supp\eta)^{\epsilon}_N} |\varphi_x| \bigg)+\frac{c'}{N}.
  \end{split}
\end{equation}One can now proceed exactly as in the proof of Lemma 3.12 of~\cite{bolthausen1993critical} below equation $(3.13)$ to estimate the right-hand side of \eqref{eq:ZTildeCloseToX} and obtain that~\eqref{eq:stepBatman} holds true. 
\end{proof}
The next proposition provides the last step in the approximation procedure, and it gives a large deviation upper bound at rate $N^{d-2}$ on the probability that $\Phi(\bbX_N^\epsilon)$ and $\Phi(\cH^\alpha_{\mathring{A}})$, tested against $\eta \otimes F$,  deviate from each other when the disconnection event $\cD^\alpha_N$ happens. This can be achieved by providing an upper bound on the probability that, uniformly in a compact set, $\bbX_N^\epsilon(x)$ is far from $\cH^\alpha_{\mathring{A}}(x)$. To this end, an application of Theorem~\ref{thm:MainUpperBound} (cf.\ \eqref{eq:LastBoundProp55} below) will be crucial.

\begin{prop}\label{prop:XclosetoTarget}
Let $\Delta > 0$, $\alpha < \overline{h}$ be fixed. Let $\eta:\bbR^d\to \bbR$ be a compactly supported Lipschitz function and $F: \bbR^{\bbZ^d}\to \bbR$ a bounded local Lipschitz function. Then, there exists $\epsilon >0$ such that 
\begin{equation}\label{eq:exp}
  \begin{split}
    \limsup_{N\to\infty}\frac{1}{N^{d-2}} \log \bbP\Big[\big|\langle & \Phi(\bbX^\epsilon_N) -  \Phi(\cH^\alpha_{\mathring{A}}), \eta\otimes F  \rangle\big| \geq \Delta;\,\cD^{\alpha}_N \Big] \\ &\leq 
    -\frac{1}{2d} (\overline{h} - \alpha)^2 \capa(\mathring{A}) - c_9(\Delta,\alpha,\eta,F).      
  \end{split}
\end{equation}
\end{prop}
\begin{proof} Let $K> 0$ be the Lipschitz constant of $F$. One has the bound
\begin{equation}\label{eq:step12}
\Big|\langle \Phi(\bbX^\epsilon_N) -  \Phi(\cH^\alpha_{\mathring{A}}), \eta\otimes F \rangle \Big| \leq K \|\eta\|_\infty  \int_{\supp \eta} |\bbX^\epsilon_N(x) - \cH^\alpha_{\mathring{A}}(x)| \,\De x.
\end{equation}
Since $\chi_\epsilon\ast \cH^\alpha_{\mathring{A}}$ converges locally in $L^1$ to $\cH^\alpha_{\mathring{A}}$ as $\epsilon \to 0$, we can find an $\epsilon > 0$ (depending on $\eta$, $F$, $\alpha$ and $\Delta$) small enough such that
\begin{equation}
 K\|\eta\|_\infty \int_{\supp\eta}|\chi_\epsilon\ast \cH^\alpha_{\mathring{A}}(x)- \cH^\alpha_{\mathring{A}}(x)| \De x \leq \tfrac{\Delta}{2}.
\end{equation}
Without loss of generality, assume that $|\supp \eta| > 0$, then we have the bound
\begin{equation}
\begin{split}
  &\bbP\Big[ K\|\eta\|_\infty \int_{\supp \eta} |\bbX^\epsilon_N(x) - \chi_\epsilon\ast\cH^\alpha_{\mathring{A}}(x)|\De x \geq \tfrac{\Delta}{2};\,\cD^{\alpha}_N  \Big] \\ & \leq
  \bbP\Big[\sup_{x\in\supp \eta} |\bbX^\epsilon_N(x) - \chi_\epsilon\ast\cH^\alpha_{\mathring{A}}(x)| \geq \tfrac{\Delta}{2|\supp\eta|  K\|\eta\|_\infty };\,\cD^{\alpha}_N  \Big].
\end{split}
\end{equation}
Rewriting $\bbX_N^\epsilon(x) = \langle \bbX_N,\chi_\epsilon(\cdot-x) \rangle$ (see~\eqref{eq:smoothing}) as well as $\chi_\epsilon \ast \cH^\alpha_{\mathring{A}}(x) = \langle \cH^\alpha_{\mathring{A}},\chi_\epsilon(\cdot - x) \rangle$, one can use Theorem~\ref{thm:MainUpperBound} to infer 
\begin{equation}
\label{eq:LastBoundProp55}
\begin{split}
\limsup_{N \rightarrow \infty} \frac{1}{N^{d-2}} \log \bbP& \Big[\sup_{x\in\supp \eta} |\langle \bbX_N - \cH^\alpha_{\mathring{A}}, \chi_\epsilon(\cdot - x)\rangle | \geq \tfrac{\Delta}{2|\supp\eta| K\|\eta\|_\infty };\,\cD^{\alpha}_N  \Big] \\
& \leq - \frac{1}{2d}(\overline{h}-\alpha)^2\capa(\mathring{A}) -c_1\big(\tfrac{\Delta}{2|\supp \eta| K \|\eta\|_\infty \|\chi_\epsilon\|_{BL}},\alpha,(\supp \eta)^\epsilon \big),
\end{split}
\end{equation}
where we have used that the functions $\tfrac{1}{\|\chi_\epsilon \|_{BL}} \chi_\epsilon(\cdot-x)$, $x \in \supp\eta$ all belong to $\lip_1((\supp \eta)^\epsilon)$. The claim follows by collecting~\eqref{eq:step12}--\eqref{eq:LastBoundProp55} and setting 
\begin{equation}
c_9(\Delta,\alpha,\eta,F)= c_1\big(\tfrac{\Delta}{2|\supp \eta| K \|\eta\|_\infty \|\chi_\epsilon\|_{BL}},\alpha,(\supp \eta)^\epsilon \big). \qedhere
\end{equation}
\end{proof}

Finally, we have all the elements to deduce Theorem~\ref{thm:ProfileDescription}. 

\begin{proof}[Proof of Theorem~\ref{thm:ProfileDescription}]
Observe that from the triangle inequality, one has
\begin{equation}
\begin{split}
\bbP\Big[\big|\langle \Phi(\cH^\alpha_{\mathring{A}}) -  \bbY_N,\eta\otimes F \rangle\big| \geq \Delta;\,\cD^{\alpha}_N \Big] &\leq
\bbP\Big[\big|\langle \Phi(\cH^\alpha_{\mathring{A}}) - \Phi(\bbX_N^\epsilon),\eta\otimes F \rangle\big| \geq \Delta/3;\,\cD^{\alpha}_N \Big] \\
& + \bbP\Big[\big|\langle \Phi(\bbX_N^\epsilon) - Z_N,\eta\otimes F\rangle\big| \geq \Delta/3\Big]\\
&+ \bbP\Big[\big|\langle Z_N - \bbY_N, \eta\otimes F \rangle\big| \geq \Delta/3\Big].
\end{split}
\end{equation}
Fix $\epsilon>0$ small enough to ensure~\eqref{eq:exp} for $\Delta/3$ in place of $\Delta$.
The result now follows as an easy combination of Proposition~\ref{prop:YclosetoZ} (with the choice $\delta = 2$), Proposition~\ref{prop:ZclosetoX} and Proposition~\ref{prop:XclosetoTarget} with $c_2(\Delta,\alpha,\eta,F) = c_9(\Delta/3,\alpha,\eta,F)$. 
\end{proof}
\begin{remark}
The approach taken to show Theorem \ref{thm:ProfileDescription} is quite general in nature and may be applicable to other situations than disconnection: Suppose that $(E_N)$ is a sequence of events with
\begin{equation}
\liminf_{N \rightarrow \infty} \frac{1}{N^{d-2}} \log \bbP[E_N] \geq - c_{10},
\end{equation}
and suppose that a result equivalent to Theorem \ref{thm:MainUpperBound} holds, that is
\begin{equation}
\limsup_{N \rightarrow \infty} \frac{1}{N^{d-2}} \log \bbP\bigg[ \sup_{\eta \in \lip_1(J)} |\langle \bbX_N - f,\eta \rangle \geq \Delta ; E_N \bigg] \leq - c_{11}(\Delta), \qquad c_{11}(\Delta) > c_{10},
\end{equation}
with a sufficiently regular function $f : \bbR^d \rightarrow \bbR$, then a procedure as in the proof of Theorem \ref{thm:ProfileDescription} will provide a `profile description' for the conditional probability $\bbP[ \ \cdot \ | E_N]$ as well, meaning that for any compactly supported Lipschitz function $\eta : \bbR^d \rightarrow \bbR$ and any bounded local Lipschitz function $F : \bbR^{\bbZ^d} \rightarrow \bbR$, one has 
\begin{equation}
\limsup_{N \rightarrow \infty} \frac{1}{N^{d-2}} \log \bbP\Big[\big|\langle \Phi(f) -  \bbY_N,\eta\otimes F \rangle\big| \geq \Delta \big\vert E_N \Big] < 0.
\end{equation}
\end{remark}

\appendix

\section{Strong coupling lemma}

In this appendix we state and prove Proposition~\ref{prop:strong_coupling} which provides a uniform comparison between the discrete harmonic potential of an arbitrary finite union of discrete $L$-boxes and the Brownian potential of their $\bbR^d$-filling when $L$ converges to infinity quick enough.
The proof relies on a strong coupling result of~\cite{einmahl1989extensions} in the spirit of Koml\'os, Major and Tusn\'ady.
\medskip

We first introduce some notation. We consider integers $L = L(N) \geq 1$, $N\in \bbN$.
We define for a non-empty finite set of points $\cC\subseteq\bbZ^d$ the set
\begin{equation}
  C = \bigcup_{z\in \cC} \Big(z + [0,L)^d \Big )\cap \bbZ^d,
\end{equation}
and its $\bbR^d$-filling, that is the compact set
\begin{equation}
  \Gamma = \bigcup_{z\in \cC} \Big(z + [0,L]^d\Big ).
\end{equation}
Moreover we define the compact sets
\begin{equation}
  \widetilde{\Gamma} =\big\{x\in \Gamma\,;\, d_\infty(x,\partial\Gamma)\geq L/4\big\},\quad \widehat{\Gamma} = \big\{x\in \bbR^d\,;\, d_\infty(x,\Gamma)\leq L/4 \big\}.
\end{equation}
Then, we have the following comparison between the hitting probabilities of $C$ for the discrete random walk and the hitting probabilities of $\widetilde{\Gamma}$ and $\widehat{\Gamma}$ for the Brownian motion. 

\begin{prop}\label{prop:strong_coupling} Assume that $c' N^\eta \leq L \leq c'' N$ for some $\eta>0$. Then, for any fixed integer $M \geq 1$
\begin{equation}\label{eq:limsup_sc}
  \liminf_{N\to\infty} \inf_{\cC \subseteq B(0, M N)} \inf_{x\in\bbZ^d} \Big(P_x[ H_{C}<\infty] - W_x[H_{\widetilde{\Gamma}}<\infty]\Big) \geq 0,
\end{equation}
\begin{equation}\label{eq:liminf_sc}
  \limsup_{N\to\infty} \sup_{\cC \subseteq B(0, M N)} \sup_{x\in\bbZ^d}\Big(P_x[ H_{C}<\infty] - W_x[H_{\widehat{\Gamma}}<\infty]\Big) \leq 0.
\end{equation}
\end{prop}
\begin{proof} The argument for this proposition is largely inspired by the proof of Proposition 4.1 of~\cite{nitzschner2018entropic}. We will only show~\eqref{eq:limsup_sc} as~\eqref{eq:liminf_sc} can be proven in a similar way interchanging the roles of the simple random walk and Brownian motion. 
  
Fix $R>M+c''$ so  that $\widetilde{\Gamma}  \subseteq B_\infty(0, R N)$ for all $\cC\subseteq B(0,M N)$. We can write
\begin{equation}\label{eq:sc1}
  \begin{aligned}
    \inf_{\cC \subseteq B(0, M N)}\inf_{x\in\bbZ^d} &\Big( P_x[ H_{C}<\infty] - W_x[H_{\widetilde{\Gamma}}<\infty]\Big) \\
    & \geq \left( \inf_{\cC \subseteq B(0, M N)}\inf_{x\in B(0,R N)} \Big( P_x[ H_{C}<\infty] - W_x[H_{\widetilde{\Gamma}}<\infty]\Big) \right) \\ & \wedge \Big( -  \sup_{z\in B_\infty(0,R)^c} W_z[H_{B(0,M+c'')}<\infty] \Big),
  \end{aligned}
\end{equation}
where we used scaling invariance of the Brownian motion to obtain the term in the second line. Note that the second member of the minimum does not depend on $N$ and converges to zero as $R \rightarrow \infty$ using the transience of Brownian motion. We will now show that the limit of the first member of the minimum as $N$ grows to infinity is non-negative.

For a closed set $F \subseteq \bbR^d$ we denote by $L_F = \sup\{0<t<\infty; Z_t\in F\}$ the time of last visit of the Brownian motion to $F$ (using the convention that $L_F = 0$ if the set on the right-hand side is empty). 
Clearly, for all $x\in B(0, R N)$ and all $\cC\subseteq B(0,MN)$, we have $\widetilde{\Gamma}\subseteq B(x,2NR)$. Thus, for any fixed  $\epsilon>0$ 
\begin{equation}\label{eq:sc2}
  W_x[H_{\widetilde{\Gamma}}<\infty]\leq W_x[H_{\widetilde{\Gamma}}\leq N^{2+\epsilon}] + W_x [L_{B(x,2NR)}> N^{2+\epsilon}],
\end{equation}
where the second summand on the right-hand side converges to $0$ as $N\to \infty$ uniformly in $x\in B(0, R N)$ and $\cC\subseteq B(0,MN)$. In fact, $W_x [L_{B(x,2NR)}> N^{2+\epsilon}] = W_0 [L_{B(0,2R)}> N^{\epsilon}] \to 0$ as $N\to \infty$ by scaling invariance and transience of the Brownian motion. Similar to~\cite[Proposition 4.1]{nitzschner2018entropic}, we define $\widehat{H}_{\widetilde{\Gamma}}$ to be the smallest integer multiple of $1/d$ bigger or equal than $H_{\widetilde{\Gamma}}$. Applying the strong Markov property at $H_{\widetilde{\Gamma}}$ and using translation invariance we get
\begin{equation}\label{eq:sc3}
W_x[H_{\widetilde{\Gamma}}\leq N^{2+\epsilon}] \leq W_x\Big[H_{\widetilde{\Gamma}}\leq N^{2+\epsilon}, |Z_{H_{\widetilde{\Gamma}}}-Z_{\widehat{H}_{\widetilde{\Gamma}}}|_{\infty}\leq \frac{L}{8}\Big] + W_0\Big[\sup_{0\leq t\leq 1/d}|Z_t|_{\infty}\geq \frac{L}{8}\Big].
\end{equation}
Note that  $W_0[\sup_{0\leq t\leq 1/d}|Z_t|_{\infty}\geq L/8]\to 0$ as $N\to \infty$. By Theorem 4 of~\cite{einmahl1989extensions} and the fact that $L\geq c'N^\eta$ for some $\eta>0$, there exists a probability space $(\overline{\Omega},\overline{\cF},\overline{P})$, a simple random walk $(\overline{X}_n)_{n\geq 0}$ on $\bbZ^d$ and a Brownian motion $(\overline{Z}_t)_{t\geq 0}$ on $\bbR^d$, both started at $x$, such that
\begin{equation}\label{eq:sc4}
  \overline{P}\Big[\max_{1\leq k\leq 2d N^{2+\epsilon}} | \overline{X}_k - \overline{Z}_{k/d}|_\infty < \frac{L}{8}\Big] \to 0, \quad \text{as $N\to \infty$.}
\end{equation}  
Denote by $H_{\widetilde{\Gamma}}^{\overline{Z}}$ and $H_{C}^{\overline{X}}$ the entrance times in $\widetilde{\Gamma}$ and $C$, associated to $\overline{Z}_\cdot$ and $\overline{X}_\cdot$ respectively. 
By combining~\eqref{eq:sc2}--\eqref{eq:sc4}, we obtain
\begin{equation}\label{eq:sc5}
  \begin{aligned}
    W_x[H_{\widetilde{\Gamma}}< \infty] &\leq \overline{P}\Big[H_{\widetilde{\Gamma}}\leq N^{2+\epsilon}, \max_{1\leq k\leq 2d N^{2+\epsilon}} | \overline{X}_k - \overline{Z}_{k/d}|_\infty < \tfrac{L}{8}, \\ &\qquad \qquad\qquad
    |\overline{Z}_{H^{\overline{Z}}_{\widetilde{\Gamma}}}-\overline{Z}_{\widehat{H}^{\overline{Z}}_{\widetilde{\Gamma}}}|_{\infty}\leq \tfrac{L}{8}\Big] + o(1)\\
      & \leq  \overline{P}[H_{C}^{\overline{X}} < \infty] + o(1), \quad \text{as $N\to \infty$,}
  \end{aligned}
\end{equation}
uniformly in $\cC\subseteq B(0,MN)$ and in $x \in B(0,RN)$. The result now follows by plugging~\eqref{eq:sc5} in~\eqref{eq:sc1} and by sending first $N\to \infty$ and then $R\to \infty$.
\end{proof}
\textbf{Acknowledgements.}
The authors wish to thank Alain-Sol Sznitman for useful discussions and valuable comments at various stages of this project.

\end{document}